\documentclass{amsart}

\usepackage{amsmath,amssymb,amsthm,epic,eepic}
\usepackage[all]{xy}

\theoremstyle{definition}
\newtheorem{definition}{Definition}
\newtheorem{question}[definition]{Question}

\theoremstyle{plain}
\newtheorem{lemma}[definition]{Lemma}
\newtheorem{proposition}[definition]{Proposition}
\newtheorem{theorem}[definition]{Theorem}
\newtheorem{corollary}[definition]{Corollary}

\DeclareMathOperator{\id}{id}

\DeclareMathOperator{\GA}{GA}

\DeclareMathOperator{\cod}{cod}

\DeclareMathOperator{\hol}{hol}
\DeclareMathOperator{\Lie}{Lie}
\DeclareMathOperator{\vol}{vol}
\DeclareMathOperator{\dev}{dev}
\DeclareMathOperator{\univ}{univ}
\DeclareMathOperator{\amb}{amb}
\DeclareMathOperator{\proj}{proj}
\DeclareMathOperator{\Fr}{Fr}

\begin{document}

\title[Continuity of the \'{A}lvarez class]{Continuity of the \'{A}lvarez class \\ under deformations}
\author{Hiraku Nozawa}
\address{Unit\'{e} de Math\'{e}matiques Pures et Applique\'{e}s, \'{E}cole Normale Sup\'{e}rieure de Lyon, 46 all\'{e}e d'Italie 69364 Lyon Cedex 07, France}
\email{nozawahiraku@06.alumni.u-tokyo.ac.jp}
\thanks{The author was partially supported by Grant-in-Aid for JSPS Fellows (19-4609), Postdoctoral Fellowship of French government (662014L) and Research Fellowship of Canon Foudation in Europe}
\subjclass[2000]{Primary: 53C12; Secondary: 37C85}

\begin{abstract}
A manifold $M$ with a foliation $\mathcal{F}$ is minimizable if there exists a Riemannian metric $g$ on $M$ such that every leaf of $\mathcal{F}$ is a minimal submanifold of $(M,g)$. For a closed manifold $M$ with a Riemannian foliation $\mathcal{F}$, \'{A}lvarez L\'{o}pez \cite{Alvarez Lopez} defined a cohomology class of degree $1$ called the \'{A}lvarez class whose triviality characterizes the minimizability of $(M,\mathcal{F})$. In this paper, we show that the family of the \'{A}lvarez classes of a smooth family of Riemannian foliations on a closed manifold is continuous with respect to the parameter. The \'{A}lvarez class has algebraic rigidity under certain topological conditions on $(M,\mathcal{F})$ as the author showed in \cite{Nozawa}. As a corollary of these two results, we show that under the same topological conditions the minimizability of Riemannian foliations is invariant under deformations.
\end{abstract}
\maketitle

\tableofcontents

\section{Introduction}
\addtocontents{toc}{\protect\setcounter{tocdepth}{1}}

\subsection*{The minimizability of Riemannian foliations} 
The minimizability of general foliations is characterized in terms of dynamical tools, for example, foliation cycles (Sullivan \cite{Sullivan 2}) or holonomy pseudogroups (Haefliger \cite{Haefliger}). On the other hand, remarkably, the minimizability of Riemannian foliations has a strong relation with the topology of manifolds. For example,
\begin{itemize}
\item The minimizability of an orientable Riemannian foliation of codimension $q$ on an orientable closed manifold is characterized by the nontriviality of the basic cohomology of degree $q$ by a theorem of Masa \cite{Masa}  (see Section~2 of Reinhart \cite{Reinhart} or Section~2.3 of Molino \cite{Molino} for the definition of the basic cohomology).
\item For a Riemannian foliation on a closed manifold, \'{A}lvarez L\'{o}pez \cite{Alvarez Lopez} defined the \'{A}lvarez class, which is a basic cohomology class of degree $1$ whose triviality characterizes the minimizability  (see Definition \ref{Definition : Alvarez class}).
\item In particular, this characterization of the minimizability by \'{A}lvarez L\'{o}pez implies that every Riemannian foliation on a closed manifold with zero first Betti number is minimizable. This is a generalization of a theorem of Ghys \cite{Ghys} on the simply connected case.
\item A Riemannian foliation $\mathcal{F}$ on a closed manifold $M$ is minimizable if $\pi_{1}M$ is of polynomial growth and $\mathcal{F}$ is developable by a result of the author \cite{Nozawa 2}.
\end{itemize}
In this paper, we show that the \'{A}lvarez classes of a smooth family of Riemannian foliations on a closed manifold is continuous with respect to the parameter. Combining this result with a rigidity theorem of the \'{A}lvarez class in Nozawa \cite{Nozawa}, we obtain the invariance of the minimizability of Riemannian foliations under deformations under the topological conditions in \cite{Nozawa}, which implies an algebraic rigidity of \'{A}lvarez class there (see Corollary~\ref{Corollary : Invariance 2} below). 

Note that the invariance of the minimizability under deformations does not hold for general foliations. We present examples of families of foliations in Section~\ref{Section : non Riemannian} to describe the situation.

\subsection*{A continuity theorem of the \'{A}lvarez class}
Our main result is as follows: Let $M$ be a closed manifold. Let $U$ be an open neighborhood of $0$ in $\mathbb{R}^{\ell}$. Let $\{\mathcal{F}^{t}\}_{t \in U}$ be a smooth family of Riemannian foliations on $M$ over $U$ (see Definitions~\ref{definition: families of general foliations} and~\ref{definition: families of foliations}).
\begin{theorem}\label{Theorem : Continuity}
The \'{A}lvarez class $\xi(\mathcal{F}^{t})$ of $(M,\mathcal{F}^{t})$ is continuous in $H^{1}(M;\mathbb{R})$ with respect to $t$.
\end{theorem}
Let us mention why Theorem~\ref{Theorem : Continuity} does not follow from the definition of the \'{A}lvarez class or the classical deformation theory. By definition, the \'{A}lvarez classes of $(M,\mathcal{F}^{t})$ is represented by the closed $1$-form obtained by orthogonally projecting the mean curvature form of $(M,\mathcal{F}^{t},g^{t})$ to the space of basic $1$-forms on $(M,\mathcal{F}^{t})$ for any bundle-like metric $g^{t}$ on $(M,\mathcal{F}^{t})$. But the space of basic $1$-forms on $(M,\mathcal{F}^{t})$ changes discontinuously with respect to $t$ in the space of $1$-forms on $M$, when the dimension of closures of generic leaves changes. Because of this discontinuity of the spaces of basic $1$-forms, we cannot obtain a continuous family of closed $1$-forms which represents the \'{A}lvarez classes directly from the definition. Furthermore, this discontinuity of the spaces of basic $1$-forms breaks the continuity of the domains of the families of basic Laplacians. Thus we cannot apply the classical technique of deformation theory using smooth families of self-adjoint operators to show Theorem~\ref{Theorem : Continuity}, at least directly.

Let us mention why Theorem~\ref{Theorem : Continuity} does not follow from the interpretation of the \'{A}lvarez class in terms of the holonomy homomorphism of the Molino's commuting sheaf of $(M,\mathcal{F}^{t})$ by a theorem of \'{A}lvarez L\'{o}pez \cite{Alvarez Lopez 2}. If the dimension of closures of generic leaves changes, the ranks of the family of Molino's commuting sheaves as flat vector bundles changes. Hence the family of Molino's commuting sheaves is not smooth as a family of flat vector bundles, and we cannot prove the continuity directly from the result of \cite{Alvarez Lopez 2}.

To prove Theorem~\ref{Theorem : Continuity}, we will take a suitable representative $\widetilde{\kappa}_{b}$ of the \'{A}lvarez class at $t=0$, which we will call the $\widetilde{\mathcal{F}}$-integrated component of the mean curvature form. Then we will approximate the \'{A}lvarez class by non-closed $1$-forms (see Section~\ref{Section : Proof of continuity theorem}). Some technical consideration on Riemannian foliations will be needed to take the representative $\widetilde{\kappa}_{b}$ of the \'{A}lvarez class at $t=0$ in Section~\ref{Section : Representative}, which is the main part of this article.

\subsection*{Deformation of minimizable Riemannian foliations}
Combining Theorem~\ref{Theorem : Continuity} with the characterization of the minimizability by the triviality of the \'{A}lvarez class by \'{A}lvarez L\'{o}pez \cite{Alvarez Lopez}, we have  
\begin{corollary}\label{Corollary : Closedness}
In parameter spaces of smooth families of Riemannian foliations on closed manifolds, the subsets consisting of parameters corresponding to minimizable Riemannian foliations are closed.
\end{corollary}
\noindent This corollary is not true for general foliations as we will see in Example~\ref{Example : Candel Conlon}.

A foliation $\mathcal{F}$ is defined to be of polynomial growth if the fundamental group of every leaf of $\mathcal{F}$ is of polynomial growth. A group $\Gamma$ is polycyclic if there exists a sequence $\{\Gamma_{i}\}_{i=1}^{n}$ of subgroups of $\Gamma$ such that $\Gamma_{0}=\Gamma$, $\Gamma_{n}=\{1\}$, $\Gamma_{i-1} \triangleright \Gamma_{i}$ and $\Gamma_{i}/\Gamma_{i+1}$ is cyclic for every $i$. Let $(M,\mathcal{F})$ be a closed manifold with a Riemannian foliation. If $\pi_{1}M$ is polycyclic or $\mathcal{F}$ is of polynomial growth, then the integration of the \'{A}lvarez class of $(M,\mathcal{F})$ along every closed path on $M$ is the exponential of an algebraic integer by a result of the author \cite{Nozawa}. By the totally disconnectedness of the set of algebraic integers in $\mathbb{R}$, we have the following corollary of Theorem~\ref{Theorem : Continuity}: Let $M$ be a closed manifold. Let $U$ be a connected open neighborhood of $0$ in $\mathbb{R}^{\ell}$. Let $\{\mathcal{F}^{t}\}_{t \in U}$ be a smooth family of Riemannian foliations on $M$ over $U$  (see Definitions~\ref{definition: families of general foliations} and~\ref{definition: families of foliations}).
\begin{corollary}\label{Corollary : Invariance 1}
If $\pi_{1}M$ is polycyclic or $\mathcal{F}^{t}$ is of polynomial growth for every $t$, then $\xi(\mathcal{F}^{t})=\xi(\mathcal{F}^{0})$ in $H^{1}(M,\mathbb{R})$ for every $t$.
\end{corollary}
The \'{A}lvarez class changes nontrivially for examples of families of solvable Lie foliations constructed by Meigniez \cite{Meigniez} (see also \cite{Meigniez 2}) as we mentioned in \cite{Nozawa}. Hence Corollary~\ref{Corollary : Invariance 1} is not true in general. By the characterization of the minimizability by the triviality of the \'{A}lvarez class by \'{A}lvarez L\'{o}pez \cite{Alvarez Lopez} and Corollary~\ref{Corollary : Invariance 1}, we have 
\begin{corollary}\label{Corollary : Invariance 2}
If $\pi_{1}M$ is polycyclic or $\mathcal{F}^{t}$ is of polynomial growth for every $t$, then one of the following holds:
\begin{enumerate}
\item For every $t$ in $U$, $(M,\mathcal{F}^{t})$ is minimizable.
\item For every $t$ in $U$, $(M,\mathcal{F}^{t})$ is not minimizable.
\end{enumerate}
\end{corollary}
\noindent Note that $\mathcal{F}$ is always of polynomial growth if $\dim \mathcal{F}=1$. Hence for Riemannian flows, the minimizability is invariant under deformations. As noted above, the invariance of the minimizability under deformations is not true for general foliations. For Riemannian foliations, it is not clear if the minimizability is invariant under deformations in general. We ask
\begin{question}\label{Question : Invariance}
Is the minimizability of Riemannian foliations on closed manifolds invariant under deformations?
\end{question}

Let $(M,\mathcal{F})$ be a closed $4$-manifold with a $\GA(1)$-Lie foliation (see Section \ref{Section : Lie foliations} for terminologies on Lie foliations). By a theorem of Matsumoto and Tsuchiya \cite{Matsumoto Tsuchiya}, $(M,\mathcal{F})$ is a homogeneous $\GA(1)$-Lie foliation up to a finite covering. In particular, $\pi_{1}M$ is isomorphic to a finite index subgroup of a lattice in a connected simply connected solvable Lie group. It is well known that a lattice of a connected simply connected solvable Lie group is polycyclic (see Raghunathan \cite{Raghunathan}). Since it follows from definition that every subgroup of a polycyclic group is polycyclic, $\pi_{1}M$ is polycyclic. On the other hand, a $\GA(1)$-Lie foliation on a closed manifold is non-minimizable, and an $\mathbb{R}^{2}$-Lie foliation on a closed manifold is minimizable. In fact, the following is true: 
\begin{proposition}\label{prop : GLiefoliation}
A $G$-Lie foliation $\mathcal{F}$ on a closed manifold $M$ is minimizable if and only if both of $G$ and $K$ are unimodular, where $K$ is a simply connected Lie group whose Lie algebra is the structural Lie algebra of the Lie foliation on the closure of a leaf of $\mathcal{F}$.
\end{proposition}
\noindent Proposition \ref{prop : GLiefoliation} follows from a generalization of a theorem of Masa \cite{Masa} by \'{A}lvarez L\'{o}pez \cite{Alvarez Lopez} to the non-orientable case and Theorem 1.2.4 of El Kacimi and Nicolau \cite{El Kacimi Nicolau} as follows: Let $q=\cod(M,\mathcal{F})$. A theorem of \'{A}lvarez L\'{o}pez \cite{Alvarez Lopez} implies that $(M,\mathcal{F})$ is minimizable if and only if the basic cohomology $H^{q}_{b}(M/\mathcal{F})$ of degree $q$ is nontrivial. Then Theorem 1.2.4 of El Kacimi and Nicolau \cite{El Kacimi Nicolau} implies that $H^{q}_{b}(M/\mathcal{F})$ is nontrivial if and only if $G$ and $K$ are unimodular. By the polycyclicity of $\pi_{1}M$, Corollary~\ref{Corollary : Invariance 2} and Proposition \ref{prop : GLiefoliation}, we get
\begin{corollary}\label{Corollary : Invariance 22}
A $\GA(1)$-Lie foliation on a closed $4$-manifold cannot be deformed into an $\mathbb{R}^{2}$-Lie foliation.
\end{corollary}
\noindent In dimensions lower than $4$, Corollary~\ref{Corollary : Invariance 22} is easily confirmed to be true is as a consequence of classification of Riemannian foliations. In higher dimensions, it is not clear if a similar result is true or not.

\subsection*{The invariance of basic cohomology of Riemannian foliations under deformations}
Let $U$ be a connected open neighborhood of $0$ in $\mathbb{R}^{\ell}$. Let $\{\mathcal{F}^{t}\}_{t \in U}$ be a smooth family of Riemannian foliations of codimension $q$ on $M$ over $U$. The dimension of the basic cohomology $H_{b}^{q}(M/\mathcal{F}^{t})$ of degree $q$ is either of $1$ or $0$ by a result of El Kacimi, Sergiescu and Hector \cite{El Kacimi Sergiescu Hector}. As remarked by \'{A}lvarez L\'{o}pez in the proof of Corollary~6.2 of \cite{Alvarez Lopez}, the triviality of the \'{A}lvarez class directly implies the nontriviality of $H_{b}^{q}(M/\mathcal{F}^{t})$. Hence Corollary~\ref{Corollary : Invariance 2} is paraphrased to
\begin{corollary}\label{Corollary : Invariance 3}
If $\pi_{1}M$ is polycyclic or $\mathcal{F}^{t}$ is of polynomial growth for every $t$, then we have $H_{b}^{q}(M/\mathcal{F}^{t}) \cong H_{b}^{q}(M/\mathcal{F}^{0})$ for every $t$ in $U$.
\end{corollary}
\noindent This corollary gives a partial positive answer to the following question asked by the author in the VIII International Colloquium on Differential Geometry at Santiago de Compostela (see \cite{Alvarez Lopez Garcia Rio}):
\begin{question}\label{Question : 2}
Is the basic cohomology of Riemannian foliations invariant under deformations?
\end{question}
The component of the basic cohomology of the degree equal to the codimension of Riemannian foliations is invariant under deformations if and only if the answer of Question~\ref{Question : Invariance} is true. We note that the answer of Question~\ref{Question : 2} is negative in degree lower than the codimension of Riemannian foliations. We have a simple counterexample as we present in Example~\ref{Example : Basic}.

\vspace{5pt}

\addtocontents{toc}{\protect\setcounter{tocdepth}{2}}
\noindent {\bf Acknowledgement.}\hspace{3pt}The author would like to express his deep gratitude to \'{E}tienne Ghys and Masayuki Asaoka for stimulative comments on Anosov flows. The author would like to express his deep gratitude to Jos\'{e} Ignacio Royo Prieto for pointing out a gap in the previous version of this paper. A part of the correction of this gap was established in the discussion with Jes\'{u}s Antonio \'{A}lvarez L\'{o}pez during the author's stay at the University of Santiago de Compostela in the spring of 2009. The author would like to express his deep gratitude to Jes\'{u}s Antonio \'{A}lvarez L\'{o}pez for his invitation, great hospitality and valuable discussion. The author would like to express his deep gratitude to Jes\'{u}s Antonio \'{A}lvarez L\'{o}pez also for his comments on this manuscript during the author's stay at Centre de Recerca Matem\`{a}tica in Barcelona in the summer of 2010. The author would like to express his deep gratitude to Centre de Recerca Matem\`{a}tica for great hospitality.

\section{Basic definitions}

\subsection{Families of foliations}

We use the terminology from Molino \cite{Molino}. We recall the definition of some basic terminology here to avoid confusion.

Let $(M,\mathcal{F})$ be a foliated manifold. By the integrability of $\mathcal{F}$, the Lie bracket on $C^{\infty}(TM)$ induces the Lie derivative with respect to vector fields tangent to the leaves
\begin{multline*}
C^{\infty} (T\mathcal{F}) \otimes C^{\infty} \left( \bigotimes^{r} (TM/T\mathcal{F}) \otimes \bigotimes^{s} (TM/T\mathcal{F})^{*} \right) \longrightarrow \\ C^{\infty} \left( \bigotimes^{r} (TM/T\mathcal{F}) \otimes \bigotimes^{s} (TM/T\mathcal{F})^{*} \right)
\end{multline*}
for all nonnegative integers $s$ and $r$.
\begin{definition}
\begin{enumerate}
\item An element $X$ of $C^{\infty}(TM/T\mathcal{F})$ is called a transverse field on $(M,\mathcal{F})$ if $L_{Y}X=0$ for every $Y$ in $C^{\infty}(T\mathcal{F})$. A vector field $Y$ on $M$ is called a basic vector field if $Y$ is mapped to a transverse field by the projection $C^{\infty}(TM) \longrightarrow C^{\infty}(TM/T\mathcal{F})$.
\item An element $g$ of $C^{\infty}(\bigotimes^{2} (TM/T\mathcal{F})^{*})$ is called a transverse metric if the following three conditions are satisfied:
\begin{enumerate}
\item $g(Y,Z)=g(Z,Y)$ for every $Y$ and $Z$ in $C^{\infty}(TM/T\mathcal{F})$,
\item $L_{X}g=0$ for every $X$ in $C^{\infty}(T\mathcal{F})$ and
\item $g_{x}(Z,Z) > 0$ for every point $x$ on $M$ for every nonzero vector $Z$ in $T_{x}M/T_{x}\mathcal{F}$.
\end{enumerate}
A Riemannian metric $g$ on $M$ is called a bundle-like metric if the restriction of $g$ to $\bigotimes^{2} (T\mathcal{F})^{\perp}$ is a transverse metric under the natural identification of $\bigotimes^{2} (T\mathcal{F})^{\perp}$ with $\bigotimes^{2}(TM/T\mathcal{F})$.
\item Let $q$ be the codimension of $(M,\mathcal{F})$. A transversal parallelism of $(M,\mathcal{F})$ is a $q$-tuple of transverse fields $X^{1}$, $X^{2}$, $\cdots$, $X^{q}$ on $(M,\mathcal{F})$ such that $\{(X^{1})_{x}, (X^{2})_{x}, \cdots, (X^{q})_{x}\}$ is a basis of $T_{x}M/T_{x}\mathcal{F}$ at each point $x$ on $M$.
\end{enumerate}
\end{definition}

We recall the definition of smooth families of foliations with transverse structures. Let $U$ be an open set in $\mathbb{R}^{\ell}$ which contains $0$. Let $M$ be a smooth manifold.
\begin{definition}\label{definition: families of general foliations}
A smooth family of $p$-dimensional foliations of $M$ over $U$ is defined by a $p$-dimensional smooth foliation $\mathcal{F}^{\amb}$ of $M \times U$ such that every leaf of $\mathcal{F}^{\amb}$ is contained in $M \times \{t\}$ for some $t$.
\end{definition}

For $t$ in $U$, let $\mathcal{F}^{t}$ be the foliation of $M \times \{t\}$ defined by the collection of the leaves of $\mathcal{F}^{\amb}$ contained in $M \times \{t\}$. Families of foliations are written as $\{\mathcal{F}^{t}\}_{t \in U}$ throughout this paper. $(\nu\mathcal{F})^{\amb}$ denotes the vector bundle over $M \times U$ defined by the kernel of the map $T(M \times U)/T\mathcal{F}^{\amb} \longrightarrow TU$ induced by the differential map of the second projection $M \times U \longrightarrow U$. We call $(\nu\mathcal{F})^{\amb}$ the family of normal bundles of $\{\mathcal{F}^{t}\}_{t \in U}$. Note that $(\nu\mathcal{F})^{\amb}|_{M \times \{t\}}$ is the normal bundle of the foliation $\mathcal{F}^{t}$ of $M \times \{t\}$ for each $t$.

\begin{definition}\label{definition: families of foliations}
\begin{enumerate}
\item A smooth family of Riemannian foliations of $M$ over $U$ is a pair consisting of a smooth family of foliation of $M \times U$ defined by $\mathcal{F}^{\amb}$ and a smooth metric $g^{\amb}$ on $(\nu\mathcal{F})^{\amb}$ such that the restriction of $g^{\amb}$ to the orthogonal normal bundle of $(M \times \{t\}, \mathcal{F}^{t})$ is a transverse metric on $(M \times \{t\}, \mathcal{F}^{t})$.
\item A smooth family of transversely parallelizable foliations of codimension $q$ of $M$ over $U$ is a pair of a smooth family $\{\mathcal{F}^{t}\}_{t \in U}$ of foliations of $M$ of codimension $q$ and a $q$-tuple of global sections $X^{1}_{\amb}$, $X^{2}_{\amb}$, $\cdots$, $X^{q}_{\amb}$ of $(\nu\mathcal{F})^{\amb}$ such that $\{X^{1}_{\amb}|_{M \times \{t\}}, X^{2}_{\amb}|_{M \times \{t\}}, \cdots, X^{q}_{\amb}|_{M \times \{t\}}\}$ is a transverse parallelism of $(M \times \{t\}, \mathcal{F}^{t})$ for each $t$.
\end{enumerate}
\end{definition}

\subsection{The \'{A}lvarez class}

We recall the definition of the \'{A}lvarez class of a closed manifold with a Riemannian foliation by \'{A}lvarez L\'{o}pez \cite{Alvarez Lopez}. We restrict ourselves to the case of oriented manifolds. The definition in the non-orientable case is done by lifting the foliation to the orientation cover as in \cite{Alvarez Lopez}.

Let $(M,\mathcal{F})$ be an oriented closed manifold with a Riemannian foliation. We fix a bundle-like metric $g$ on $(M,\mathcal{F})$. We have a direct sum decomposition
\begin{equation}\label{Equation : Orthogonal decomposition}
C^{\infty}(\wedge^{k} T^{*}M) = C^{\infty}_{b}(\wedge^{k} T^{*}M) \oplus C^{\infty}_{b}(\wedge^{k} T^{*}M)^{\perp}
\end{equation}
with respect to the metric induced by $g$, where $C^{\infty}_{b}(\wedge^{k} T^{*}M)$ is the space of basic $k$-forms on $(M,\mathcal{F})$. Let $\rho_{\mathcal{F}}$ be the first projection
\begin{equation}
\rho_{\mathcal{F}} \colon C^{\infty}(\wedge^{k} T^{*}M) \longrightarrow C^{\infty}_{b}(\wedge^{k} T^{*}M).
\end{equation}
We denote the mean curvature form of $(M,\mathcal{F},g)$ by $\kappa$ (see, for example, Section~10.5 of Candel and Conlon \cite{Candel Conlon} for the definition of the mean curvature form of $(M,\mathcal{F},g)$).
\begin{definition}\label{Definition : Alvarez class}
For an oriented closed manifold $M$ with a Riemannian foliation $\mathcal{F}$, we define a basic $1$-form $\kappa_{b}$ on $(M,\mathcal{F})$ by
\begin{equation}\label{Equation : Definition of kappa b}
\kappa_{b}=\rho_{\mathcal{F}}(\kappa)
\end{equation}
and call $\kappa_{b}$ the \'{A}lvarez form of $(M,\mathcal{F})$. This $\kappa_{b}$ is closed by Corollary 3.5 of \'{A}lvarez L\'{o}pez \cite{Alvarez Lopez}. We define the \'{A}lvarez class of $(M,\mathcal{F})$ by the cohomology class of $\kappa_{b}$ in $H^{1}(M;\mathbb{R})$. We denote the \'{A}lvarez class of $(M,\mathcal{F})$ by $\xi(\mathcal{F})$.
\end{definition}
Let $H^{1}_{b}(M/\mathcal{F})$ be the basic cohomology group of degree $1$ of $(M,\mathcal{F})$ (see Section~2 of Reinhart \cite{Reinhart} or Section~2.3 of Molino \cite{Molino} for the definition of the basic cohomology). \'{A}lvarez L\'{o}pez defined the \'{A}lvarez class as an element of $H^{1}_{b}(M/\mathcal{F})$ in \cite{Alvarez Lopez}. Since the canonical map $H^{1}_{b}(M/\mathcal{F}) \longrightarrow H^{1}(M/\mathcal{F})$ is injective as easily confirmed by definition, Definition~\ref{Definition : Alvarez class} gives the essentially same data as in \cite{Alvarez Lopez}.

The simple proof of the following lemma is due to a comment of \'{A}lvarez L\'{o}pez to the author:
\begin{lemma}\label{Lemma : Finite Covering}
Let $(M_{1},\mathcal{F}_{1})$ be a closed manifold with a Riemannian foliation. Let $p \colon M_{2} \longrightarrow M_{1}$ be a finite covering, and let $\mathcal{F}_{2}=p^{*}\mathcal{F}_{1}$, which is a Riemannian foliation on $M_{2}$. Then we have $\xi(\mathcal{F}_{2})=p^{*}\xi(\mathcal{F}_{1})$.
\end{lemma}

\begin{proof}
We take a bundle-like metric $g_{1}$ on $(M_{1},\mathcal{F}_{1})$. Then $p^{*}g_{1}$ is a bundle-like metric on $(M_{1},\mathcal{F}_{1})$. We consider orthogonal decompositions
\begin{align}
\Omega^{1}(M_{1}) & = \Omega^{1}_{b}(M_{1}/\mathcal{F}_{1}) \oplus \big(\Omega^{1}_{b}(M_{1}/\mathcal{F}_{1})\big)^{\perp}, \label{Equation : decomposition1} \\
\Omega^{1}(M_{2}) & = \Omega^{1}_{b}(M_{2}/\mathcal{F}_{2}) \oplus \big(\Omega^{1}_{b}(M_{2}/\mathcal{F}_{2})\big)^{\perp} \label{Equation : decomposition2}
\end{align}
\noindent with respect to the metric induced by $g_{1}$ and $p^{*}g_{1}$, respectively. By definition of metrics, we have 
\begin{align}
p^{*}\Omega^{1}_{b}(M_{1}/\mathcal{F}_{1}) & = p^{*}\Omega^{1}(M_{1}) \cap \Omega^{1}_{b}(M_{2}/\mathcal{F}_{2}), \\
p^{*}\big(\Omega^{1}_{b}(M_{1}/\mathcal{F}_{1})\big)^{\perp} & = p^{*}\Omega^{1}(M_{1}) \cap \big(\Omega^{1}_{b}(M_{2}/\mathcal{F}_{2})\big)^{\perp}.
\end{align}
\noindent Let $\rho_{\mathcal{F}_{1}}$ and $\rho_{\mathcal{F}_{2}}$ be the first projections on decompositions \eqref{Equation : decomposition1} and \eqref{Equation : decomposition2}, respectively. These equalities imply 
\begin{equation}\label{Equation : commute}
p^{*}\rho_{\mathcal{F}_{1}} = \rho_{\mathcal{F}_{2}}p^{*}.
\end{equation}
Let $\kappa_{i}$ be the mean curvature forms of $(M_{i},\mathcal{F}_{i})$ with respect to $g_{1}$ and $p^{*}g_{1}$, respectively. We get $\kappa_{2}=p^{*}\kappa_{1}$. By \eqref{Equation : commute}, we get
\begin{equation*}
p^{*}(\kappa_{1})_{b}=p^{*}\rho_{\mathcal{F}_{1}}(\kappa_{1}) = \rho_{\mathcal{F}_{2}}(p^{*}\kappa_{1}) =  \rho_{\mathcal{F}_{2}}(\kappa_{2})=(\kappa_{2})_{b}.  \qed 
\end{equation*}
\renewcommand{\qed}{} \end{proof}

\section{Fundamentals of Lie foliation theory}\label{Section : Lie foliations}

We summarize the fundamental facts of Lie foliation theory due to Fedida \cite{Fedida} and \cite{Fedida 2} (see also Section 4.2 of Molino \cite{Molino} or Section 4.3.1 of Moerdijk and Mr\v{c}un \cite{Moerdijk Mrcun}) to use in Sections~\ref{Section : Representative} and \ref{Section : Continuity}.

Let $G$ be a connected Lie group. Recall that a $G$-Lie foliation is a foliation with a transverse $(G,G)$-structure, where $G$ acts on $G$ by the right multiplication. A $G$-Lie foliation has a structure of $G'$-Lie foliation for any covering group $G'$ of $G$ as easily confirmed. Thus we will assume the simply connectedness of $G$ throughout this paper. We recall the following
\begin{definition}\label{Definition : Structural Lie algebra}
The Lie algebra of $G$ is called the structural Lie algebra of the Lie foliation.
\end{definition}
Let $(M,\mathcal{F})$ be a $G$-Lie foliation. Let $p_{M}^{\univ} \colon M^{\univ} \longrightarrow M$ be the universal cover of $M$. Fix a point $x_{0}^{\univ}$ in $M^{\univ}$, and let $x_{0}=p_{M}^{\univ}(x_{0}^{\univ})$.
\begin{description}
\item[(I)] We have a fiber bundle $\dev \colon M^{\univ} \longrightarrow G$ which maps $x_{0}^{\univ}$ to the unit element $e$ of $G$ and whose fibers are the leaves of the foliation $(p_{M}^{\univ})^{*}\mathcal{F}$ of $M^{\univ}$. This $\dev$ is called the developing map of $(M,\mathcal{F})$.
\item[(II)] We have a group homomorphism $\hol \colon \pi_{1}(M,x_{0}) \longrightarrow G$ such that
\begin{equation}\label{Equation : dev and hol}
\dev(\gamma \cdot x) = \dev(x) \cdot_{G}  \hol(\gamma)
\end{equation}
for $x$ in $M^{\univ}$ and $\gamma$ in $\pi_{1}(M,x_{0})$, where $\cdot_{G}$ is the multiplication of $G$. The image of $\hol$ is called the holonomy group of $(M,\mathcal{F})$. The holonomy group of $(M,\mathcal{F})$ is dense in $G$ if and only if the leaves of $\mathcal{F}$ are dense in $M$.
\item[(III)] Recall that the Maurer-Cartan form $\theta$ on $G$ is a $\Lie(G)$-valued $1$-form on $G$ defined by $\theta_{g}(v)=(R_{g})_{*}v$ for $g$ in $G$ and $v$ in $T_{g}G$, where $R_{g}$ is the right multiplication map by $g$. Since a $\Lie(G)$-valued $1$-form $\dev^{*}\theta$ on $M^{\univ}$ is invariant under the $\pi_{1}(M,x_{0})$-action, $\dev^{*}\theta$ induces a $\Lie(G)$-valued $1$-form $\Omega$ on $M$. The structure of a $G$-Lie foliation $(M,\mathcal{F})$ is determined by this $\Lie(G)$-valued $1$-form $\Omega$ on $M$. This $\Omega$ is called the Maurer-Cartan form of a $G$-Lie foliation $(M,\mathcal{F})$.
\item[(IV)] The Maurer-Cartan form $\Omega$ of a $G$-Lie foliation $(M,\mathcal{F})$ satisfies the equation 
\begin{equation}\label{Equation : Maurer-Cartan}
d\Omega + \frac{1}{2}[\Omega,\Omega]=0.
\end{equation}
Conversely, if a $\Lie(G)$-valued $1$-form $\Omega$ on $M$ satisfies \eqref{Equation : Maurer-Cartan} and $\Omega_{x} \colon T_{x}M \longrightarrow \Lie(G)$ is surjective for every $x$ in $M$, then $\Omega$ is the Maurer-Cartan form of a $G$-Lie foliation of $M$.
\item[(V)] Let $q = \cod (M,\mathcal{F})$. Let $\{\overline{X}^{j}\}_{j = 1}^{q}$ be a basis of $\Lie(G)$. Let $\{\overline{\omega}_{i}\}_{i = 1}^{q}$ be the dual basis of $\Lie(G)^{*}$. Let $\omega_{i}$ be the $1$-form on $M$ induced from a $\pi_{1}(M,x_{0})$-invariant $1$-form $\dev^{*}\overline{\omega}_{i}$ on $M^{\univ}$. Let $X^{j}$ be the vector field on $M$ orthogonal to $T\mathcal{F}$ such that $\omega_{i}(X^{j})=\delta_{ij}$ for each $i$ and $j$, where $\delta_{ij}$ is the Kronecker's delta. Then $\{X^{j}\}_{j = 1}^{q}$ is clearly a transverse parallelism of $(M,\mathcal{F})$. Here, the Maurer-Cartan form $\Omega$ of $(M,\mathcal{F})$ is given by the equation $\Omega_{x}((X^{j})_{x})=\overline{X}^{j}$ for each point $x$ on $M$.
\end{description}

\section{Reduction to the orientable transversely parallelizable case}

We reduce the proof of Theorem~\ref{Theorem : Continuity} to the case where $\{\mathcal{F}^{t}\}_{t \in U}$ is a family of transversely parallelizable foliations.

Let $\{\mathcal{F}^{t}\}_{t \in U}$ be a smooth family of Riemannian foliations of codimension $q$ of a closed manifold $M$ over $U$. Clearly we can assume that $U$ is contractible without loss of generality. Let $\Fr(\nu\mathcal{F})^{\amb}$ be the family of the frame bundles associated with the family $(\nu\mathcal{F})^{\amb}$ of vector bundles on $M$. The metric $g^{\amb}$ on $(\nu\mathcal{F})^{\amb}$ determines an $O(q)$-reduction $O(\nu\mathcal{F})^{\amb}$ of $\Fr(\nu\mathcal{F})^{\amb}$. We denote $O(\nu\mathcal{F})^{\amb}|_{M \times \{0\}}$ by $O(\nu\mathcal{F})^{0}$. Since $O(\nu\mathcal{F})^{\amb}$ is the total space of a $O(\nu\mathcal{F})^{0}$-bundle over a contractible base space $U$, we can trivialize $O(\nu\mathcal{F})^{\amb}$ as $O(\nu\mathcal{F})^{\amb} \cong O(\nu\mathcal{F})^{0} \times U$. By the standard construction of the Molino theory, we have the following lemma:

\begin{lemma}
There exists a foliation $\mathcal{G}^{\amb}$ of $O(\nu\mathcal{F})^{0} \times U$ and a $\frac{q(q+1)}{2}$-tuple of transverse fields of $(O(\nu\mathcal{F})^{0} \times U, \mathcal{G}^{\amb})$ defining a smooth family $\{\mathcal{G}^{t}\}_{t \in U}$ of transversely parallelizable foliations of codimension $\frac{q(q+1)}{2}$ of $O(\nu\mathcal{F})^{0}$ over $U$. 
\end{lemma}

\begin{proof}
Let $\{(V_{\lambda},\phi_{\lambda})\}$ be a Haefliger cocycle defining a foliation $\mathcal{F}^{\amb}$ of $M \times U$. Then $\{(O(\nu\mathcal{F})^{\amb}|_{V_{\lambda}},d\phi_{\lambda})\}$ is a Haefliger cocycle on $O(\nu\mathcal{F})^{\amb}$, where $d\phi_{\lambda}$ is the map induced on the frame bundle by $\phi_{\lambda}$. We define a foliation $\mathcal{G}^{\amb}$ of $O(\nu\mathcal{F})^{0} \times U$ pushing out the foliation of $O(\nu\mathcal{F})^{\amb}$ defined by the Haefliger cocycle $\{(O(\nu\mathcal{F})^{\amb}|_{V_{\lambda}},d\phi_{\lambda})\}$ by the trivialization $O(\nu\mathcal{F})^{\amb} \cong O(\nu\mathcal{F})^{0} \times U$. 

Let $\mathcal{G}^{t}=\mathcal{G}^{\amb}|_{O(\nu\mathcal{F})^{0} \times \{t\}}$. For each $t$, we can construct a transversely parallelism of $(O(\nu\mathcal{F})^{0} \times \{t\}, \mathcal{G}^{t})$ from the transverse Levi-Civita connection on $O(\nu\mathcal{F})^{\amb}|_{M \times \{t\}}$ and the canonical $1$-form on the frame bundle $O(\nu\mathcal{F})^{\amb}|_{M \times \{t\}}$ as in Section 5.1 of Molino \cite{Molino} or Theorem 4.20 of Moerdijk and Mr\v{c}un \cite{Moerdijk Mrcun}. Since the transverse Levi-Civita connections and the canonical $1$-forms on $(O(\nu\mathcal{F})^{0} \times \{t\}, \mathcal{G}^{t})$ are smooth with respect to the parameter $t$, we have a smooth family of transverse parallelisms.
\end{proof}

\begin{lemma}\label{Lemma : Reduction to TP case}
If $\xi(\mathcal{G}^{t})$ is continuous with respect to $t$, then $\xi(\mathcal{F}^{t})$ is also continuous with respect to $t$.
\end{lemma}

\begin{proof}
Let $\pi \colon O(\nu\mathcal{F})^{0} \times U \longrightarrow M \times U$ be the projection. By Lemma 7 of Nozawa \cite{Nozawa}, we have $(\pi|_{O(\nu\mathcal{F})^{0} \times \{t\}})^{*} \xi(\mathcal{F}^{t}) = \xi(\mathcal{G}^{t})$ for each $t$. Since $(\pi|_{O(\nu\mathcal{F})^{0} \times \{t\}})^{*} \colon H^{1}(M \times \{t\};\mathbb{R}) \longrightarrow H^{1}(O(\nu\mathcal{F})^{0} \times \{t\};\mathbb{R})$ is injective, the continuity of $\xi(\mathcal{F}^{t})$ follows from the continuity of  $\xi(\mathcal{G}^{t})$.
\end{proof}

By Lemmas~\ref{Lemma : Finite Covering} and \ref{Lemma : Reduction to TP case}, we have the following
\begin{lemma}\label{Lemma : Reduction to TP case 2}
The general case of Theorem~\ref{Theorem : Continuity} follows from the special case where
\begin{enumerate}
\item $\{\mathcal{F}^{t}\}_{t \in U}$ is a smooth family of transversely parallelizable foliations of $M$.
\item $M$ and the basic fibration of $\mathcal{F}^{0}$ are orientable. Moreover if we have another foliation $\widetilde{\mathcal{F}}$ on $M$, we can assume that $\widetilde{\mathcal{F}}$ is also orientable.
\end{enumerate}
\end{lemma}
Here, $\widetilde{\mathcal{F}}$ in the statement of (ii) will be $\widetilde{\mathcal{F}}^{0}$ which appears in Section~\ref{Section : Verification}. For the proof of Lemma~\ref{Lemma : Reduction to TP case 2}, we note that the finite covering $p \colon M_{1} \longrightarrow M_{2}$ induces an injection $p_{*} \colon H_{1}(M_{1};\mathbb{R}) \longrightarrow H_{1}(M_{2};\mathbb{R})$. Then $p^{*} \colon H^{1}(M_{2};\mathbb{R}) \longrightarrow H^{1}(M_{1};\mathbb{R})$ is injective by the Poincar\'{e} duality for closed manifolds $M_{1}$ and $M_{2}$.

\section{A representative $\widetilde{\kappa}_{b}$ of the \'{A}lvarez class}\label{Section : Representative}

\subsection{Definition~of~the~$\widetilde{\mathcal{F}}$-integrated~component~$\widetilde{\kappa}_{b}$~of~the~mean~curvature~form}\label{Section : Integrated}

We will define the $\widetilde{\mathcal{F}}$-integrated component $\widetilde{\kappa}_{b}$ of the mean curvature form for transversely parallelizable foliations in a way similar to that of the definition of the \'{A}lvarez form $\kappa_{b}$ for transversely parallelizable foliations in \'{A}lvarez L\'{o}pez \cite{Alvarez Lopez}. Let $(M,\mathcal{F})$ be a closed manifold with a transversely parallelizable foliation. Let $q=\cod (M,\mathcal{F})$. We take a transverse parallelism $\{X^{j}\}_{j = 1}^{\cod (M,\mathcal{F})}$ of $(M,\mathcal{F})$. Let $\{\omega_{i}\}_{i = 1}^{q}$ be the set of basic $1$-forms on $(M,\mathcal{F})$ such that $\omega_{i}(X^{j})=\delta_{ij}$, where $\delta_{ij}$ is the Kronecker's delta. Let 
\begin{equation}
\omega_{I} = \omega_{i_{1}} \wedge \omega_{i_{2}} \wedge \cdots \wedge \omega_{i_{k}}
\end{equation}
for a set $I=\{i_{1},i_{2},\cdots,i_{k}\}$, where $1 \leq i_{1} < i_{2} < \cdots < i_{k} \leq q$. Assume that we have the following diagram:
\begin{equation}\label{diagram:triangle}
\xymatrix{M \ar[r]^{\pi_{b}} \ar[dr]_{\pi_{\widetilde{\mathcal{F}}}} & W \ar[d] \\ & V, }
\end{equation}
where $\pi_{b} \colon M \longrightarrow W$ is the basic fibration of $(M,\mathcal{F})$ and $\pi_{\widetilde{\mathcal{F}}}$ is a submersion. Recall that the basic fibration of $(M,\mathcal{F})$ is a fiber bundle whose fibers are closures of leaves of $(M,\mathcal{F})$. We denote the foliations of $M$ defined by the fibers of $\pi_{b}$ and $\pi_{\widetilde{\mathcal{F}}}$ by $\mathcal{F}_{b}$ and $\widetilde{\mathcal{F}}$, respectively. From now on we denote $W$ and $V$ by $M/\mathcal{F}_{b}$ and $M/\widetilde{\mathcal{F}}$, respectively. 

We assume that $M$ and the fiber bundle $\pi_{\widetilde{\mathcal{F}}}$ are orientable. We fix a bundle-like metric $g$ on $(M,\mathcal{F})$. Then we define a map $\rho_{\widetilde{\mathcal{F}}}$ by
\begin{multline}\label{Equation : Definition of rho tilde F}
\rho_{\widetilde{\mathcal{F}}}\Big(\tau + \sum_{I \subset \{1,2,\cdots,q\}, \newline |I| = k} f^{I} \omega_{I} \Big) = \\ \frac{1}{\pi_{\widetilde{\mathcal{F}}}^{*} \left( \int_{\widetilde{\mathcal{F}}} \vol_{\widetilde{\mathcal{F}}} \right)} \sum_{I \subset \{1,2,\cdots,q\}, |I| = k} \Big(\int_{\widetilde{\mathcal{F}}} f^{I} \vol_{\widetilde{\mathcal{F}}} \Big) \omega_{I}
\end{multline}
for $f^{I}$ in $C^{\infty}(M)$ and $\tau$ in $C^{\infty} \big( T^{*}\mathcal{F} \otimes (\wedge^{k-1}T^{*}M) \big)$, where $\int_{\widetilde{\mathcal{F}}}$ is the integration along the fiber of $\pi_{\widetilde{\mathcal{F}}}$ with respect to the fixed orientation and $\vol_{\widetilde{\mathcal{F}}}$ is the fiberwise volume form of $\pi_{\widetilde{\mathcal{F}}}$ determined by the metric $g$. Note that we have a direct sum decomposition
\begin{equation}
C^{\infty}(\wedge^{k} T^{*}M) = (\ker \rho_{\widetilde{\mathcal{F}}})^{\perp} \oplus \ker \rho_{\widetilde{\mathcal{F}}}
\end{equation}
and $\rho_{\widetilde{\mathcal{F}}}$ is the first projection like in the case of $\rho_{\mathcal{F}}$.

\begin{definition}
We define the $\widetilde{\mathcal{F}}$-integrated component $\widetilde{\kappa}_{b}$ of the mean curvature form $\kappa$ of $(M,\mathcal{F},g)$ with respect to the transverse parallelism $\{\omega_{i}\}_{i=1}^{q}$ by
\begin{equation}\label{Equation : Definition of tilde kappa b}
\widetilde{\kappa}_{b}=\rho_{\widetilde{\mathcal{F}}}(\kappa).
\end{equation}
\end{definition}

Note that $\rho_{\widetilde{\mathcal{F}}}$ coincides with $\rho_{\mathcal{F}}$ and $\widetilde{\kappa}_{b} = \kappa_{b}$ if $\widetilde{\mathcal{F}} = \mathcal{F}_{b}$. It is the situation originally considered in \'{A}lvarez L\'{o}pez \cite{Alvarez Lopez}.

Note that $\rho_{\widetilde{\mathcal{F}}}$ depends on the choice of the transverse parallelism $\{X^{j}\}_{j=1}^{q}$. This is different from the case of $\rho_{\mathcal{F}}$, which is determined only by the metric $g$. We remark that there will be a natural choice of $\{\omega_{i}\}_{i=1}^{q}$, when we apply this construction in Section~\ref{Section : Continuity}. Note that
\begin{equation}\label{Equation : rho F and rho tilde F}
\rho_{\widetilde{\mathcal{F}}} \rho_{\mathcal{F}}=\rho_{\widetilde{\mathcal{F}}}
\end{equation}
by definition.

Note that $\widetilde{\kappa}_{b}$ may not be closed, while the \'{A}lvarez form $\kappa_{b}$ is closed by Corollary 3.5 of \'{A}lvarez L\'{o}pez \cite{Alvarez Lopez}. We do not define $\widetilde{\kappa}_{b}$ for the case where $M$ or $\pi_{\widetilde{\mathcal{F}}}$ is not orientable. This is because it is not used in this paper.

\subsection{The statement of Proposition~\ref{Proposition : Fiberwise average}}

We will state Proposition~\ref{Proposition : Fiberwise average}, which asserts that the \'{A}lvarez class of $(M,\mathcal{F})$ is represented by $\widetilde{\kappa}_{b}$ under certain conditions. These conditions will be naturally satisfied in our application in Section~\ref{Section : Continuity}. The proof of the essential part of Proposition~\ref{Proposition : Fiberwise average} occupies the rest of this section.

In this section we continue to use the notation defined in Section \ref{Section : Integrated}. Let $d_{1,0}$ be the composition of the de~Rham differential and the projection $C^{\infty}(\wedge^{\bullet+1} T^{*}M) \longrightarrow C^{\infty}((T\widetilde{\mathcal{F}}^{\perp})^{*} \otimes \wedge^{\bullet} T^{*}\widetilde{\mathcal{F}})$ determined by $g$.

\begin{lemma}\label{Lemma : Constant volume}
If each leaf of $(M,\widetilde{\mathcal{F}})$ is minimal with respect to $g$, then the function $\int_{\widetilde{\mathcal{F}}} \vol_{\widetilde{\mathcal{F}}}$ on the leaf space $M/\widetilde{\mathcal{F}}$ is constant.
\end{lemma}

\begin{proof}
By the assumption, we have $\kappa_{\widetilde{\mathcal{F}}}=0$. By the Rummler's formula (see the second formula in the proof of Proposition 1 in Rummler \cite{Rummler} or Lemma 10.5.6 of Candel and Conlon \cite{Candel Conlon}), we have $d_{1,0} \vol_{\widetilde{\mathcal{F}}}= - \kappa_{\widetilde{\mathcal{F}}} \wedge \vol_{\widetilde{\mathcal{F}}}$. Hence we have $d_{1,0} \vol_{\widetilde{\mathcal{F}}}=0$. Then
\begin{equation}
d \left( \int_{\widetilde{\mathcal{F}}} \vol_{\widetilde{\mathcal{F}}} \right) = \int_{\widetilde{\mathcal{F}}} d \vol_{\widetilde{\mathcal{F}}} = \int_{\widetilde{\mathcal{F}}} d_{1,0} \vol_{\widetilde{\mathcal{F}}} = 0.
\end{equation}
Here the second equality follows from the degree counting of the differential forms.
\end{proof}

\begin{proposition}\label{Proposition : Fiberwise average}
We assume that $M$, $\widetilde{\mathcal{F}}$ and the basic fibration $\pi_{b}$ of $(M,\mathcal{F})$ are orientable. We also suppose that 
\begin{description}
\item[(a)] The fixed bundle-like metric $g$ on $(M,\mathcal{F})$ is bundle-like also with respect to $\widetilde{\mathcal{F}}$.
\item[(b)] Each leaf of $(M,\widetilde{\mathcal{F}})$ is minimal with respect to $g$.
\item[(c)] We define functions $c_{i}^{jk}$ on $M$ by
\begin{equation}
d\omega_{i} = \sum_{1 \leq j < k \leq q} c_{i}^{jk} \omega_{j} \wedge \omega_{k}.
\end{equation}
Then $c_{i}^{jk}|_{\widetilde{L}}$ is a constant for each fiber $\widetilde{L}$ of $\pi_{\widetilde{\mathcal{F}}}$.
\item[(d)] Let $\proj_{\widetilde{\mathcal{F}}} \colon C^{\infty}(TM/T\mathcal{F}) \longrightarrow C^{\infty}(TM/T\widetilde{\mathcal{F}})$ be the canonical projection. Then $\proj_{\widetilde{\mathcal{F}}}X_{i}$ is a transverse field on $(M,\widetilde{\mathcal{F}})$ for each $i$.
\end{description}
Then we have
\begin{enumerate}
\item $\widetilde{\kappa}_{b}$ is a closed $1$-form on $M$ and

\item $[\kappa_{b}]=[\widetilde{\kappa}_{b}]$ in $H^{1}(M;\mathbb{R})$.
\end{enumerate}
\end{proposition}

We will show (i) here. (ii) will be shown at the end of this section.

\begin{proof}[Proof of Proposition~\ref{Proposition : Fiberwise average} (i)]

By Lemma~\ref{Lemma : Constant volume}, the function $\pi_{\widetilde{\mathcal{F}}}^{*} \left( \int_{\widetilde{\mathcal{F}}} \vol_{\widetilde{\mathcal{F}}}  \right)$ on $M$ is constant. Let
\begin{equation}\label{Equation : C}
C=\pi_{\widetilde{\mathcal{F}}}^{*} \left( \int_{\widetilde{\mathcal{F}}} \vol_{\widetilde{\mathcal{F}}} \right),
\end{equation}
and
\begin{equation}\label{Equation : h i and omega i}
\kappa_{b} = \sum_{i=1}^{q} h^{i} \omega_{i}.
\end{equation}
By the condition (c), we have 
\begin{equation}\label{Equation : Assumption c}
\rho_{\widetilde{\mathcal{F}}} \Big(\sum_{i=1}^{q} h^{i} d\omega_{i} \Big) = \sum_{i=1}^{q} \left( \int_{\widetilde{\mathcal{F}}} h^{i} \vol_{\widetilde{\mathcal{F}}} \right) d\omega_{i}.
\end{equation}

Using \eqref{Equation : Definition of tilde kappa b}, \eqref{Equation : rho F and rho tilde F} and \eqref{Equation : Definition of kappa b}, in this order, we have
\begin{equation}\label{Equation : Re}
\widetilde{\kappa}_{b} = \rho_{\widetilde{\mathcal{F}}}\kappa = \rho_{\widetilde{\mathcal{F}}} \rho_{\mathcal{F}} \kappa = \rho_{\widetilde{\mathcal{F}}} \kappa_{b}.
\end{equation}
We have 
\begin{align*}
& \textstyle \, d\widetilde{\kappa}_{b} \\
\textstyle = & \textstyle \, d\rho_{\widetilde{\mathcal{F}}} \kappa_{b} \\
\textstyle = & \textstyle \, \frac{1}{C} d \Big(\sum_{i=1}^{q} \big( \int_{\widetilde{\mathcal{F}}} h^{i} \vol_{\widetilde{\mathcal{F}}} \big) \omega_{i} \Big) \\
\textstyle = & \textstyle \, \frac{1}{C} \sum_{i=1}^{q} \big( \int_{\widetilde{\mathcal{F}}} dh^{i} \wedge \vol_{\widetilde{\mathcal{F}}} \big) \wedge \omega_{i} +  \frac{1}{C} \sum_{i=1}^{q} \big( \int_{\widetilde{\mathcal{F}}} h^{i} d\vol_{\widetilde{\mathcal{F}}} \big) \wedge \omega_{i} + \\
& \textstyle \hspace{250pt} \frac{1}{C} \sum_{i=1}^{q} \big( \int_{\widetilde{\mathcal{F}}} h^{i} \vol_{\widetilde{\mathcal{F}}} \big) d\omega_{i} \\
\textstyle = & \textstyle \, \rho_{\widetilde{\mathcal{F}}} \Big(\sum_{i=1}^{q} dh^{i} \wedge \omega_{i} \Big) + \frac{1}{C} \sum_{i=1}^{q} \big( \int_{\widetilde{\mathcal{F}}} h^{i} d\vol_{\widetilde{\mathcal{F}}} \big) \wedge \omega_{i} + \rho_{\widetilde{\mathcal{F}}} \Big(\sum_{i=1}^{q} h^{i} d\omega_{i} \Big).
\end{align*}
Here we used \eqref{Equation : Re} in the first equality. The second equality follows from the combination of \eqref{Equation : Definition of rho tilde F}, \eqref{Equation : C} and \eqref{Equation : h i and omega i}. We used the commutativity of the integration along the fiber with $d$ in the third equality. The fourth equality follows from the equations \eqref{Equation : Definition of rho tilde F} and \eqref{Equation : Assumption c}. We have
\begin{align*}
& \textstyle \, \rho_{\widetilde{\mathcal{F}}} \Big(\sum_{i=1}^{q} dh^{i} \wedge \omega_{i} \Big) + \frac{1}{C} \sum_{i=1}^{q} \big( \int_{\widetilde{\mathcal{F}}} h^{i} d\vol_{\widetilde{\mathcal{F}}} \big) \wedge \omega_{i} + \rho_{\widetilde{\mathcal{F}}} \Big(\sum_{i=1}^{q} h^{i} d\omega_{i} \Big) \\
\textstyle = & \textstyle \, \rho_{\widetilde{\mathcal{F}}} \Big(d\big(\sum_{i=1}^{q} h^{i} \omega_{i} \big)\Big) + \frac{1}{C} \sum_{i=1}^{q} \big( \int_{\widetilde{\mathcal{F}}} h^{i} d\vol_{\widetilde{\mathcal{F}}} \big) \wedge \omega_{i} \\
\textstyle = & \textstyle \, \rho_{\widetilde{\mathcal{F}}} \Big(d\big(\sum_{i=1}^{q} h^{i} \omega_{i} \big)\Big) + \frac{1}{C}  \sum_{i=1}^{q} \big( \int_{\widetilde{\mathcal{F}}} h^{i} d_{1,0} \vol_{\widetilde{\mathcal{F}}} \big) \wedge \omega_{i} \\
\textstyle = & \textstyle \, \rho_{\widetilde{\mathcal{F}}} (d\kappa_{b}) - \frac{1}{C} \sum_{i=1}^{q} \big( \int_{\widetilde{\mathcal{F}}} h^{i} \kappa_{\widetilde{\mathcal{F}}} \wedge \vol_{\widetilde{\mathcal{F}}} \big) \wedge \omega_{i}.
\end{align*}
\noindent Here, in the first equality, we combined the first and the third terms. The second equality follows from degree comparison. The third equality follows from \eqref{Equation : h i and omega i} and the Rummler's formula (see the second formula in the proof of Proposition 1 in Rummler \cite{Rummler} or Lemma 10.5.6 of Candel and Conlon \cite{Candel Conlon}).

The first term of the last line is $0$, because $\kappa_{b}$ is closed by Corollary 3.5 of \'{A}lvarez L\'{o}pez \cite{Alvarez Lopez}. Since $\kappa_{\widetilde{\mathcal{F}}}$ is $0$ by the condition (b), the second term is also $0$. Hence (i) is proved.
\end{proof}

\subsection{$\xi(\mathcal{F})=[\widetilde{\kappa}_{b}]$ on the fibers of $\pi_{\widetilde{F}}$}

We will show a lemma which will be used in the proof of Proposition~\ref{Proposition : Fiberwise average} to show the restriction of $\xi(\mathcal{F})$ and $ [\widetilde{\kappa}_{b}]$ to the fibers of $\pi_{\widetilde{\mathcal{F}}}$ are equal. Here, $\alpha$ will be considered to be $\xi(\mathcal{F})-[\widetilde{\kappa}_{b}]$ in the application in Section 5.6.

\begin{lemma}\label{Lemma : Integral along flows}
Assume that the conditions (a) and (b) of Proposition~\ref{Proposition : Fiberwise average} are satisfied. Let $M'$ be an orientable submanifold of $M$ which is a union of fibers of $\pi_{\widetilde{\mathcal{F}}}$. Let $\{\phi_{t}\}_{t \in [0,1]}$ be the flow on $M'$ generated by a vector field $X$ on $M'$. Assume that $\omega_{i}(X)$ is constant on each fiber of $\pi_{\widetilde{\mathcal{F}}}$ for each $i$. Then we have 
\begin{equation}
\int_{M'} \Big( \int_{\gamma_{x}} \kappa_{b} \Big) \vol_{M'}(x) = \int_{M'} \Big( \int_{\gamma_{x}} \widetilde{\kappa}_{b} \Big) \vol_{M'}(x),
\end{equation}
where $\gamma_{x}$ is the orbit of $x$ of $\{\phi_{t}\}_{t \in [0,1]}$, and $\vol_{M'}$ is the volume form on $M'$ determined by $g$.
\end{lemma}

\begin{proof}
The function $\pi_{\widetilde{\mathcal{F}}}^{*} \left( \int_{\widetilde{\mathcal{F}}} \vol_{\widetilde{\mathcal{F}}} \right)$ on $M'$ is constant by the condition (b) and Lemma~\ref{Lemma : Constant volume}. We take a real number $C$ and functions $h^{i}$ on $M$ as \eqref{Equation : C} and \eqref{Equation : h i and omega i}. Then we have
\begin{align*}
& \textstyle \, \int_{M'} \Big( \int_{\gamma_{x}} \rho_{\widetilde{\mathcal{F}}} (\kappa) \Big) \vol_{M'}(x) \\
\textstyle = & \textstyle \, \int_{M'} \Big( \int_{\gamma_{x}} \rho_{\widetilde{\mathcal{F}}} (\kappa_{b}) \Big) \vol_{M'}(x) \\ 
\textstyle = & \textstyle \, \frac{1}{C} \int_{M'} \Big( \int_{\gamma_{x}} \Big( \sum_{i=1}^{q} \int_{\widetilde{\mathcal{F}}} h^{i} \vol_{\widetilde{\mathcal{F}}} \Big) \omega_{i} \Big) \vol_{M'}(x) \\
\textstyle = & \textstyle \, \frac{1}{C} \int_{M'} \Big( \int_{0}^{1}  \Big( \sum_{i=1}^{q} \int_{\widetilde{\mathcal{F}}} h^{i} \vol_{\widetilde{\mathcal{F}}} \Big)_{\gamma_{x}(t)} \omega_{i}(X)_{\gamma_{x}(t)} dt \Big) \vol_{M'}(x) 
\end{align*}
\noindent Here, we used \eqref{Equation : Re} in the first equality. We used \eqref{Equation : Definition of rho tilde F} and \eqref{Equation : h i and omega i} in the second equality.

By the assumption, $\omega_{i}(X)$ is constant on the fibers of $\pi_{\widetilde{\mathcal{F}}}$. Then we have
\begin{align*}
& \textstyle \, \frac{1}{C} \int_{M'} \Big( \int_{0}^{1} \sum_{i=1}^{q} \Big( \int_{\widetilde{\mathcal{F}}} h^{i} \vol_{\widetilde{\mathcal{F}}} \Big)_{\gamma_{x}(t)} \omega_{i}(X)_{\gamma_{x}(t)} dt \Big) \vol_{M'}(x) \\ 
\textstyle = & \textstyle \, \frac{1}{C} \int_{M'} \Big( \int_{0}^{1} \Big( \sum_{i=1}^{q} \int_{\widetilde{\mathcal{F}}} h^{i} \omega_{i}(X)_{\gamma_{x}(t)} \vol_{\widetilde{\mathcal{F}}} \Big) dt \Big) \vol_{M'}(x) \\
\textstyle = & \textstyle \, \frac{1}{C} \int_{M'} \Big( \int_{0}^{1} \Big( \int_{\widetilde{\mathcal{F}}} \gamma_{x}^{*} \kappa_{b}(X) \vol_{\widetilde{\mathcal{F}}} \Big) dt \Big) \vol_{M'}(x) \\
\textstyle = & \textstyle \, \frac{1}{C} \int_{M'} \Big( \int_{\widetilde{\mathcal{F}}} \Big( \int_{0}^{1} \gamma_{x}^{*} \kappa_{b}(X) dt \Big) \vol_{\widetilde{\mathcal{F}}} \Big) \vol_{M'}(x) \\
\textstyle = & \textstyle \, \frac{1}{C} \int_{M'} \Big( \int_{\widetilde{\mathcal{F}}} \Big( \int_{\gamma_{x}} \kappa_{b} \Big) \vol_{\widetilde{\mathcal{F}}} \Big) \vol_{M'}(x). 
\end{align*}
By the condition (a), we have $\int_{M'} f \vol_{M'} = \int_{M'/\widetilde{\mathcal{F}}} \Big( \int_{\widetilde{\mathcal{F}}} f \vol_{\widetilde{\mathcal{F}}} \Big) \vol_{M'/\widetilde{\mathcal{F}}}$ for a function $f$ on $M'$, where $\vol_{M'/\widetilde{\mathcal{F}}}$ is the volume form on $M'$ determined by $g$. Hence we have
\begin{align*}
& \textstyle \, \frac{1}{C} \int_{M'} \Big( \int_{\widetilde{\mathcal{F}}} \Big( \int_{\gamma_{x}} \kappa_{b} \Big) \vol_{\widetilde{\mathcal{F}}} \Big) \vol_{M'}(x) \\
\textstyle = & \textstyle \, \frac{1}{C} \int_{M'/\widetilde{\mathcal{F}}} \Big( \int_{\widetilde{\mathcal{F}}} \Big( \int_{\widetilde{\mathcal{F}}} \Big( \int_{\gamma_{x}} \kappa_{b} \Big) \vol_{\widetilde{\mathcal{F}}} \Big) \vol_{\widetilde{\mathcal{F}}} \Big) \vol_{M'/\widetilde{\mathcal{F}}} \\
\textstyle = & \textstyle \, \int_{M'/\widetilde{\mathcal{F}}} \Big( \int_{\widetilde{\mathcal{F}}} \Big( \int_{\gamma_{x}} \kappa_{b} \Big) \vol_{\widetilde{\mathcal{F}}} \Big) \vol_{M'/\widetilde{\mathcal{F}}} \\
\textstyle = & \textstyle \, \int_{M'} \Big( \int_{\gamma_{x}} \kappa_{b} \Big) \vol_{M'}(x).
\end{align*}
The proof of Lemma~\ref{Lemma : Integral along flows} is completed.
\end{proof} 

We fix a point $x_{0}$ in $M$. Let $\widetilde{F}$ be the fiber of $\pi_{\widetilde{\mathcal{F}}}$ containing $x_{0}$. Let
\begin{equation}
q_{\widetilde{F}}=\cod (\widetilde{F},\mathcal{F}|_{\widetilde{F}}).
\end{equation}
We assume that the condition (d) of Proposition~\ref{Proposition : Fiberwise average} is satisfied. We define transverse fields $Y^{1}$, $Y^{2}$, $\cdots$, $Y^{q}$ on $(M,\mathcal{F})$ by
\begin{equation}
Y^{j} = \sum_{i=1}^{q} a^{j}_{i} X^{i},
\end{equation}
where we choose a nondegenerate matrix $(a^{j}_{i})_{ 1 \leq j \leq q, 1 \leq i \leq q}$ so that $(Y^{1})_{x_{0}}$, $(Y^{2})_{x_{0}}$, $\cdots$, $(Y^{q_{\widetilde{F}}})_{x_{0}}$ are tangent to $\widetilde{F}$. By the condition (d) of Proposition~\ref{Proposition : Fiberwise average}, the vector fields $Y^{1}$, $Y^{2}$, $\cdots$, $Y^{q_{\widetilde{F}}}$ are basic with respect to $\widetilde{\mathcal{F}}$. Hence $Y^{1}$, $Y^{2}$, $\cdots$, $Y^{q_{\widetilde{F}}}$ are tangent to $\widetilde{F}$ at every point of $\widetilde{F}$. We denote the vector subspace $\oplus_{j=1}^{q_{\widetilde{F}}}\mathbb{R} (Y^{j}|_{\widetilde{F}})$ of the Lie algebra of transverse fields on $(\widetilde{F},\mathcal{F}|_{\widetilde{F}})$ by $\mathfrak{g}$. Then:
\begin{lemma}\label{Lemma : Condition d 2}
Assume that the conditions (c) and (d) of Proposition~\ref{Proposition : Fiberwise average} are satisfied.
\begin{enumerate}
\item $\mathfrak{g}$ is a Lie subalgebra of the Lie algebra of transverse fields on $(\widetilde{F},\mathcal{F}|_{\widetilde{F}})$.
\item $(\widetilde{F},\mathcal{F}|_{\widetilde{F}})$ is a Lie foliation.
\end{enumerate}
\end{lemma}

\begin{proof}
Let us prove (i). We take basic $1$-forms $\zeta_{i}$ by $\zeta_{i}(Y^{j})=\delta_{ij}$, where $\delta_{ij}$ is the Kronecker's delta. By the condition (c), we have
\begin{equation}\label{Equation : omega'}
d\zeta_{i} = \sum_{1 \leq j < k \leq q} c_{i}^{jk}\zeta_{j} \wedge \zeta_{k},
\end{equation}
where $c_{i}^{jk}$ are functions on $M$ whose restrictions to $\widetilde{F}$ are constant. Clearly, we have
\begin{equation}
\zeta_{i}([Y^{j},Y^{k}]) = -2d\zeta_{i}(Y^{j},Y^{k}) = -2c_{i}^{jk}
\end{equation}
For any transverse field $Z$ on $(M, \mathcal{F})$, we have $Z = \sum_{i=1}^{q} \zeta_{i}(Z) Y^{i}$. Hence we have
\begin{equation}\label{Equation : Y}
[Y^{j},Y^{k}] = \sum_{i=1}^{q} \zeta_{i}([Y^{j},Y^{k}]) Y^{i} = -2 \sum_{i=1}^{q} c_{i}^{jk} Y^{i}.
\end{equation}
Recall the notation $q_{\widetilde{F}}=\cod (\widetilde{F},\mathcal{F}|_{\widetilde{F}})$. Consider the case of $1 \leq j < k \leq q_{\widetilde{F}}$. Note that $[Y^{j},Y^{k}]$ is also tangent to $\widetilde{F}$ at every point on $\widetilde{F}$ because $Y^{j}$ and $Y^{k}$ are tangent to $\widetilde{F}$ at every point on $\widetilde{F}$. Then $c_{i}^{jk}$ must be $0$ for $q_{\widetilde{F}}+1 \leq i \leq q$ by \eqref{Equation : Y}, because $Y^{q_{\widetilde{F}}+1}$, $\cdots$, $Y^{q}$ are not tangent to $\widetilde{F}$ at $x_{0}$. Hence $[Y^{j},Y^{k}]$ is contained in $\mathfrak{g}$. Thus (i) is proved.

Let us prove (ii). Here, $\mathfrak{g}$ is a Lie algebra by (i). We define a $\mathfrak{g}$-valued $1$-form $\Omega$ on $\widetilde{F}$ by
\begin{equation}\label{Equation : Omega}
\Omega_{x}((Y^{j})_{x})=Y^{j}
\end{equation}
for every point $x$ on $\widetilde{F}$ and every $j$. For the proof of (ii), it suffices to show that $\Omega$ satisfies the Maurer-Cartan equation $d\Omega + \frac{1}{2}[\Omega,\Omega] = 0$ by (III) and (IV) of Section~\ref{Section : Lie foliations}. This is proved in a way similar to the argument of Theorem 4.24 of Moerdijk and Mr\v{c}un \cite{Moerdijk Mrcun} as follows: For $1 \leq j < k \leq q_{\widetilde{F}}$, we have
\begin{align*}
& \textstyle \, d\Omega(Y^{j},Y^{k}) + \frac{1}{2}[\Omega,\Omega](Y^{j},Y^{k}) \\
\textstyle = & \textstyle \, \frac{1}{2} \big( Y^{j}(\Omega(Y^{k})) - Y^{k}(\Omega(Y^{j})) - \Omega([Y^{j},Y^{k}]) \big) + \frac{1}{2}[\Omega(Y^{j}),\Omega(Y^{k})] \\
\textstyle = & \textstyle \, \frac{1}{2}\big( Y^{j}(\Omega(Y^{k})) - Y^{k}(\Omega(Y^{j})) \big) \\
\textstyle = & \textstyle \, 0. 
\end{align*}
In the second equality, we used the equality $\Omega([Y^{j},Y^{k}])=[\Omega(Y^{j}),\Omega(Y^{k})]$ which follows from the definition of $\Omega$. The last equality follows from the fact that the $\mathfrak{g}$-valued functions $\Omega(Y^{j})$ and $\Omega(Y^{k})$ are constant on $\widetilde{F}$.
\end{proof}

Recall that $\mathcal{F}_{b}$ denotes the foliation of $M$ defined by the fibers of the basic fibration $\pi_{b} \colon M \longrightarrow M/\mathcal{F}_{b}$.

\begin{lemma}\label{Lemma : Fiber of pi tilde}
Assume that the conditions (c) and (d) of Proposition~\ref{Proposition : Fiberwise average} are satisfied. Assume that a basic closed $1$-form $\alpha$ on $(\widetilde{F},\mathcal{F}|_{\widetilde{F}})$ satisfies
\begin{equation}
\int_{\widetilde{F}} \Big( \int_{\gamma_{x}} \alpha \Big) \vol_{\widetilde{F}}(x) = 0
\end{equation}
for the flow $\{\phi_{t}\}_{t \in [0,1]}$ on $\widetilde{F}$ generated by every vector field $X$ such that $\omega_{i}(X)$ is constant on $\widetilde{F}$ for each $i$, where $\gamma_{x}$ is the orbit of $x$ of the flow $\{\phi_{t}\}_{t \in [0,1]}$. Then we have 
\begin{enumerate}
\item $\alpha|_{L_{b}}=0$ for each leaf $L_{b}$ of $(\widetilde{F},\mathcal{F}_{b}|_{\widetilde{F}})$ and
\item $[\alpha]=0$ in $H^{1}(\widetilde{F};\mathbb{R})$.
\end{enumerate}
\end{lemma}

\begin{proof}
Let us show (i) in the case where $L_{b}$ is the leaf $F_{b}$ of $(\widetilde{F},\mathcal{F}_{b}|_{\widetilde{F}})$ which contains $x_{0}$. The proof of the general case is similar.

Here $(\widetilde{F},\mathcal{F}|_{\widetilde{F}})$ is a Lie foliation by Lemma~\ref{Lemma : Condition d 2} (ii). We take a connected Lie group $G$ so that $(\widetilde{F},\mathcal{F}|_{\widetilde{F}})$ is a $G$-Lie foliation. We can assume the simply connectedness of $G$ as noted in the second paragraph of Section~\ref{Section : Lie foliations}.

Let $p_{\widetilde{F}}^{\univ} \colon \widetilde{F}^{\univ} \longrightarrow \widetilde{F}$ be the universal covering of $\widetilde{F}$. Fix a point $x_{0}^{\univ}$ in the fiber over $x_{0}$. Let $\dev \colon \widetilde{F} \longrightarrow G$ be the developing map of the Lie foliation $(\widetilde{F},\mathcal{F}|_{\widetilde{F}})$ which maps $x_{0}^{\univ}$ to the unit element $e$ of $G$. Let $\hol$ be the holonomy homomorphism $\pi_{1}(\widetilde{F},x_{0}) \longrightarrow G$ of the Lie foliation $(\widetilde{F},\mathcal{F}|_{\widetilde{F}})$. By (I) of Section~\ref{Section : Lie foliations}, every basic $1$-form on $( \widetilde{F}^{\univ},(p_{\widetilde{F}}^{\univ})^{*}(\mathcal{F}|_{\widetilde{F}}) )$ is the pullback of a $1$-form on $G$ by $\dev$. Hence there exists a $1$-form $\overline{\alpha}$ on $G$ such that
\begin{equation}\label{Equation : Pull back}
(p_{\widetilde{F}}^{\univ})^{*} \alpha = \dev^{*} \overline{\alpha}.
\end{equation}
By (II) of Section~\ref{Section : Lie foliations} and the invariance of $(p_{\widetilde{F}}^{\univ})^{*} \alpha$ under the $\pi_{1}(\widetilde{F},x_{0})$-action on $\widetilde{F}^{\univ}$, we have
\begin{equation}\label{Equation : right invariance 1}
R_{\hol(\gamma)}^{*} \overline{\alpha} = \overline{\alpha}
\end{equation}
for $\gamma$ in $\pi_{1}(\widetilde{F},x_{0})$, where $R_{\hol(\gamma)} \colon G \longrightarrow G$ is the right multiplication map by $\hol(\gamma)$. Let $H$ be the closure of the image of $\hol$ in $G$. Note that $H$ is a proper subgroup of $G$ if the leaves of $(\widetilde{F},\mathcal{F}|_{\widetilde{F}})$ are not dense by (II) in Section~\ref{Section : Lie foliations}. It follows from \eqref{Equation : right invariance 1} that 
\begin{equation}\label{Equation : right invariance 2}
R_{g}^{*} \overline{\alpha} = \overline{\alpha}
\end{equation}
for every $g$ in $H$.

In the sequel, for a path $\gamma$ in $\widetilde{F}$, we denote a lift of $\gamma$ to $\widetilde{F}^{\univ}$ by $\gamma^{\univ}$. By \eqref{Equation : right invariance 2}, we have
\begin{equation}\label{Equation : Integration 1}
\int_{\gamma} \alpha = \int_{\gamma^{\univ}} (p_{\widetilde{F}}^{\univ})^{*} \alpha = \int_{\gamma^{\univ}} \dev^{*} \overline{\alpha} = \int_{\dev_{*}\gamma^{\univ}} \overline{\alpha}.
\end{equation}

Let $\gamma$ be a closed path on $\widetilde{F}$ whose endpoints are $x_{0}$. We denote the element of $\pi_{1}(\widetilde{F},x_{0})$ represented by $\gamma$ by the same symbol $\gamma$. Let $\overline{X}_{\gamma}$ be the left invariant vector field on $G$ such that $\exp \overline{X}_{\gamma} = \hol(\gamma)$. Let $X^{\univ}_{\gamma}$ be a lift of $\overline{X}_{\gamma}$ to $\widetilde{F}^{\univ}$ which is invariant by the action of $\hol(\pi_{1}(\widetilde{F},x_{0}))$. Let $X_{\gamma}$ be the vector field on $\widetilde{F}$ whose lift to $\widetilde{F}^{\univ}$ is $X^{\univ}_{\gamma}$. Let $\gamma_{x}$ be the orbit of $x$ in $\widetilde{F}$ of the flow $\{\phi_{t}\}_{t \in [0,1]}$ generated by $X_{\gamma}$. It follows that $\dev_{*}\gamma_{x}^{\univ}$ is an orbit of the flow generated by $\overline{X}_{\gamma}$ from the definition of $\overline{X}_{\gamma}$ and $\gamma_{x}$. Hence we have
\begin{equation}\label{Equation : dev gamma univ}
\dev_{*}\gamma_{x}^{\univ}(t) = \dev_{*} \gamma_{x}^{\univ}(0) \cdot \exp (t\overline{X}_{\gamma}).
\end{equation}
We take the lifts $\gamma^{\univ}$ and $\gamma_{x_{0}}^{\univ}$ of $\gamma$ and $\gamma_{x_{0}}$ to $\widetilde{F}^{\univ}$ so that $\gamma^{\univ}(0)=\gamma^{\univ}_{x_{0}}(0)=x^{\univ}_{0}$, respectively. Then, by \eqref{Equation : dev and hol} and \eqref{Equation : dev gamma univ}, we have 
\begin{align}\label{Equation : Closing 1}
\dev_{*} \gamma_{x_{0}}^{\univ}(0) & = e = \dev_{*} \gamma^{\univ}(0), \\
\dev_{*} \gamma_{x_{0}}^{\univ}(1) & = \hol(\gamma) = \dev_{*} \gamma^{\univ}(1). \label{Equation : Closing 2}
\end{align}
\noindent Since $G$ is simply connected, \eqref{Equation : Integration 1}, \eqref{Equation : Closing 1}, \eqref{Equation : Closing 2} and the Stokes theorem imply
\begin{equation}\label{Equation : F 2}
\int_{\gamma} \alpha = \int_{\dev_{*}\gamma^{\univ}} \overline{\alpha} = \int_{\dev_{*}\gamma_{x_{0}}^{\univ}} \overline{\alpha} = \int_{\gamma_{x_{0}}} \alpha.\end{equation}

It follows from \eqref{Equation : dev gamma univ} and $\dev(\gamma_{x_{0}}^{\univ}(0))=e$ that
\begin{equation}\label{Equation : gamma x and gamma x_0}
\dev_{*}\gamma_{x}^{\univ} = (R_{\dev(\gamma_{x}^{\univ}(0)) \cdot \dev(\gamma_{x_{0}}^{\univ}(0))^{-1}})_{*} \dev_{*}\gamma_{x_{0}}^{\univ} = (R_{\dev(\gamma_{x}^{\univ}(0))})_{*} \dev_{*}\gamma_{x_{0}}^{\univ}.
\end{equation}
If $\gamma$ and $x$ are contained in $F_{b}$, then we can take $\gamma_{x}^{\univ}$ so that $\dev(\gamma_{x}^{\univ}(0))$ is contained in $H$. By using \eqref{Equation : Integration 1}, \eqref{Equation : gamma x and gamma x_0}, \eqref{Equation : right invariance 2} and \eqref{Equation : Integration 1}, in this order, we have
\begin{equation}\label{Equation : F 3}
\int_{\gamma_{x}} \alpha = \int_{\dev_{*}\gamma_{x}^{\univ}} \overline{\alpha}  = \int_{(R_{\dev(\gamma_{x}^{\univ}(0))})_{*} \dev_{*}\gamma_{x_{0}}^{\univ}} \overline{\alpha} = \int_{\dev_{*}\gamma_{x_{0}}^{\univ}} \overline{\alpha} = \int_{\gamma_{x_{0}}} \alpha.
\end{equation}

Since $X_{\gamma}$ satisfies the condition of $X$ in the statement of Lemma~\ref{Lemma : Fiber of pi tilde}, we have
\begin{equation}\label{Equation : F 1}
\int_{\widetilde{F}} \Big( \int_{\gamma_{x}} \alpha \Big) \vol_{\widetilde{F}}(x) = 0
\end{equation}
by the assumption. 

By \eqref{Equation : F 2}, \eqref{Equation : F 3} and \eqref{Equation : F 1}, we have
\begin{equation}
\int_{\gamma} \alpha = 0.
\end{equation}
Then $\alpha|_{F_{b}}$ is exact. Hence there exists a basic function $h$ on $(F_{b},\mathcal{F}|_{F_{b}})$ such that $dh=\alpha|_{F_{b}}$. But since the leaves of $(F_{b},\mathcal{F}|_{F_{b}})$ are dense, $h$ is constant. Then we have $\alpha=dh=0$, which completes the proof of (i).

In the following five paragraphs, we will show
\begin{equation}\label{Equation : F 4}
\int_{\gamma_{x}} \alpha = \int_{\gamma_{x_{0}}} \alpha
\end{equation}
for every point $x$ on $\widetilde{F}$.

Let us show that $[\alpha]$ in $H^{1}(\widetilde{F};\mathbb{R})$ is contained in the image of $(\pi_{b}|_{\widetilde{F}})^{*} \colon H^{1}(\widetilde{F}/\mathcal{F}_{b};\mathbb{R}) \longrightarrow H^{1}(\widetilde{F};\mathbb{R})$, where $\pi_{b} \colon M \longrightarrow M/\mathcal{F}_{b}$ is the basic fibration of $(M,\mathcal{F})$. Since $\alpha$ is basic with respect to $\mathcal{F}$, we have $\phi^{*}\alpha=\alpha$ for a diffeomorphism $\phi$ which maps each leaf of $\mathcal{F}$ to itself. Each leaf $L$ of $\mathcal{F}$ is dense in the leaf $L_{b}$ of $\mathcal{F}_{b}|_{\widetilde{F}}$ which contains $L$. Hence the orbits of the group of diffeomorphisms which map each leaf of $\mathcal{F}$ to itself is dense in $L_{b}$. So we have $(L_{Y}\alpha)_{x}=0$ for every point $x$ on $M$ and every $Y$ in $T_{x}M$ tangent to $L_{b}$. Since $\alpha|_{L_{b}}$ is zero by (i), $\alpha$ is basic with respect to $\mathcal{F}_{b}$. Thus $[\alpha]$ in $H^{1}(\widetilde{F};\mathbb{R})$ is contained in $(\pi_{b}|_{\widetilde{F}})^{*}(H^{1}(\widetilde{F}/\mathcal{F}_{b};\mathbb{R}))$ in $H^{1}(\widetilde{F};\mathbb{R})$.

The path $\gamma_{x}$ may not be closed in general. But let us show that $(\pi_{b})_{*}\gamma_{x}$ is closed, where $\pi_{b} \colon M \longrightarrow M/\mathcal{F}_{b}$ is the basic fibration of $(M,\mathcal{F})$. Let $H$ be the Lie subgroup of $G$ defined by the closure of $\hol(\pi_{1}(\widetilde{F},x_{0}))$. The structural Lie algebra of the Lie foliation $(F_{b},\mathcal{F}|_{F_{b}})$ is $\Lie(H)$, and hence $\dim H = \cod (F_{b},\mathcal{F}|_{F_{b}})$. By the equivariance of the developing map in \eqref{Equation : dev and hol}, the map $\dev \colon \widetilde{F}^{\univ} \longrightarrow G$ induces a map $\dev_{G/H} \colon \widetilde{F} \longrightarrow G/H$. Furthermore, $\dev_{G/H}$ induces a map $\varpi \colon \widetilde{F}/\mathcal{F}_{b} \longrightarrow G/H$. Since $\varpi$ is a submersion between two manifolds of the same dimension, $\varpi$ is a covering map. Since $\varpi$ is injective as easily confirmed, $\varpi$ is a diffeomorphism.

Let $X_{b}$ be a vector field on $\widetilde{F}/\mathcal{F}_{b}$ induced from $X_{\gamma}$. Let $\overline{X}_{G/H}$ be a vector field on $G/H$ induced from a vector field $\overline{X}_{\gamma}$ on $G$. By definition of $X_{\gamma}$ and $\overline{X}_{\gamma}$, we have
\begin{equation}\label{Equation : Vector fields}
\varpi_{*}\overline{X}_{G/H} = X_{b}.
\end{equation}
Recall that $\gamma_{x}$ is the orbit of $x$ of the flow $\{\phi_{t}\}_{0 \leq t \leq 1}$ generated by $X_{\gamma}$ from time zero to time one. Thus $(\pi_{b})_{*}\gamma_{x}$ is the orbit of $x$ of the flow generated by $X_{b}$ from time zero to time one. By \eqref{Equation : Vector fields}, $\varpi$ maps an orbit of the flow generated by $X_{b}$ from time zero to time one to an orbit of the flow on $G/H$ generated by the vector field $\overline{X}_{G/H}$ from time zero to time one. Here the time one map of the flow generated by $\overline{X}_{G/H}$ is the identity, because this map is induced by the time one map of the flow on $G$ generated by $\overline{X}_{\gamma}$, which is the right multiplication map by an element of $\hol(\pi_{1}(\widetilde{F},x_{0}))$ by definition of $\overline{X}_{\gamma}$. Hence $(\pi_{b})_{*}\gamma_{x}$ is a closed path on $\widetilde{F}/\mathcal{F}_{b}$ for each $x$.

The homology class determined by $(\pi_{b})_{*}\gamma_{x}$ in $\widetilde{F}/\mathcal{F}_{b}$ is independent of $x$. This is because $\gamma_{x}$ and $\gamma_{y}$ are bounded by a $1$-parameter family of closed paths on $\widetilde{F}/\mathcal{F}_{b}$ of the form $\{\gamma_{l(s)}\}_{0 \leq s \leq 1}$, where $l$ is a path on $\widetilde{F}$ such that $l(0)=x$ and $l(1)=y$ for every two points $x$ and $y$ in $\widetilde{F}/\mathcal{F}_{b}$.

Thus, by the argument in the previous three paragraphs, $[\alpha]$ is contained in $(\pi_{b}|_{\widetilde{F}})^{*}(H^{1}(\widetilde{F}/\mathcal{F}_{b};\mathbb{R}))$, and $(\pi_{b})_{*}\gamma_{x}$ determines the same homology class of $\widetilde{F}/\mathcal{F}_{b}$ for every $x$. Hence \eqref{Equation : F 4} is proved.

By \eqref{Equation : F 2}, \eqref{Equation : F 1} and \eqref{Equation : F 4}, we have
\begin{equation}
\int_{\gamma} \alpha = 0.
\end{equation}
Hence (ii) is proved.
\end{proof}

\subsection{Two lemmas on a fiber bundle over $S^1$ with fiberwise Lie foliations}

We will show two lemmas to use in the next section. Note that $(\widetilde{F},\mathcal{F}_{b}|_{\widetilde{F}})$ is transversely orientable by the assumption of the orientability of both of $\widetilde{\mathcal{F}}$ and the basic fibration $\pi_{b} \colon M \longrightarrow M/\mathcal{F}_{b}$ of $(M,\mathcal{F})$.

\begin{lemma}\label{Lemma : Existence of nabla}
Assume that the conditions (c) and (d) of Proposition~\ref{Proposition : Fiberwise average} are satisfied. Let $\gamma \colon S^1 \longrightarrow M$ be a smooth embedding in $M$. Assume that $\gamma$ is transverse to the fibers of $\pi_{\widetilde{\mathcal{F}}}$, and that $\pi_{\widetilde{\mathcal{F}}} \circ \gamma$ is an embedding. Let $M' = \pi_{\widetilde{\mathcal{F}}}^{-1}(\pi_{\widetilde{\mathcal{F}}}(S^1))$. Then there exists a flat connection $\nabla$ on $\pi_{\widetilde{\mathcal{F}}}|_{M'}$ which satisfies the following four conditions:
\begin{description}
\item[(A)] The holonomy map $f \colon \widetilde{F} \longrightarrow \widetilde{F}$ of $\nabla$ preserves the foliation $\mathcal{F}_{b}|_{\widetilde{F}}$.
\item[(B)] The holonomy map $f$ of $\nabla$ preserves a transverse volume form  $\mu_{\widetilde{F}/\mathcal{F}_{b}}$ of $(\widetilde{F},\mathcal{F}_{b}|_{\widetilde{F}})$.
\item[(C)] We denote the inverse map $\pi_{\widetilde{\mathcal{F}}} \circ \gamma(S^1) \longrightarrow S^1$ of $\pi_{\widetilde{\mathcal{F}}} \circ \gamma$ by $\varphi$. We define a section $\gamma_{1}$ of $\pi_{\widetilde{\mathcal{F}}}|_{M'}$ by $\gamma_{1} = \gamma \circ \varphi$. Then $\gamma_{1}$ is a section parallel to $\nabla$.
\item[(D)] There exists a vector field $Z_{\nabla}$ such that the orbits of the flow generated by $Z_{\nabla}$ is parallel to $\nabla$ and the restriction of $\omega_{i}(Z_{\nabla})$ to each fiber of $\pi_{\widetilde{\mathcal{F}}}$ is constant.
\end{description}
\end{lemma}

\begin{proof}
Let $\hat{X}^{j}$ be a section of $TM|_{M'}$ on $M'$ such that $\hat{X}^{j}$ is projected to $X^{j}$ by the canonical projection $C^{\infty}(TM|_{M'}) \longrightarrow C^{\infty}( (TM/T\mathcal{F})|_{M'})$ for $1 \leq j \leq q$. There exists a vector field $Y$ tangent to $\mathcal{F}$ defined on $\gamma_{1}(S^1)$ and functions $h_{1}$, $h_{2}$, $\cdots$, $h_{q}$ on $S^1$ such that 
\begin{equation}\label{Equation : Tangent vectors}
(D\gamma_{1})_{t} \left( \frac{\partial}{\partial t} \right) = Y_{\gamma_{1}(t)} + \sum_{j=1}^{q} h_{j}(t)(\hat{X}^{j})_{\gamma_{1}(t)},
\end{equation}
where $(D\gamma_{1})_{t}$ is the differential map of $\gamma_{1}$ at a point $t$. Let $Y'$ be a vector field on $M'$ which is tangent to $\mathcal{F}$ and whose restriction to $\gamma_{1}(S^1)$ is equal to $Y$. We define a vector field $Z_{\nabla}$ on $M'$ by
\begin{equation}
Z_{\nabla} = Y' + \sum_{j=1}^{q} ((\pi_{\widetilde{\mathcal{F}}}|_{M'})^{*}h_{j}) \hat{X}^{j}.
\end{equation}
The restriction of $Z_{\nabla}$ to $\gamma_{1}(S^1)$ is equal to the tangent vectors of $\gamma_{1}$ by \eqref{Equation : Tangent vectors}. $Z_{\nabla}$ is basic with respect to $\mathcal{F}$ and transverse to the fibers of $\pi_{\widetilde{\mathcal{F}}}$. We define a connection $\nabla$ on $\pi_{\widetilde{\mathcal{F}}}$ by the line field tangent to $Z_{\nabla}$ at each point on $M'$. It is trivial that $\nabla$ is flat, because every connection on a fiber bundle over $S^1$ is flat.

Let us show that $\nabla$ satisfies the conditions (A), (B), (C) and (D). Here, $\hat{X}^{j}$ is basic with respect to $\widetilde{\mathcal{F}}|_{M'}$ by the condition (d) of Proposition~\ref{Proposition : Fiberwise average}. Hence $Z_{\nabla}$ is also basic with respect to $\widetilde{\mathcal{F}}|_{M'}$ by definition. On a foliated manifold, the flow generated by a basic vector field maps leaves to leaves by Proposition 2.2 of Molino \cite{Molino}. Then the flow generated by $Z_{\nabla}$ also maps fibers of $\pi_{\widetilde{\mathcal{F}}}|_{M'}$ to fibers of $\pi_{\widetilde{\mathcal{F}}}|_{M'}$. The time one map of the flow generated by $Z_{\nabla}$ maps $\widetilde{F}$ to $\widetilde{F}$ itself. Since the orbits of the flow generated by $Z_{\nabla}$ are parallel to $\nabla$ by definition, the time one map of the flow generated by $Z_{\nabla}$ is the holonomy of $\nabla$. This proves that $\nabla$ satisfies the condition (D). Since the restriction of $Z_{\nabla}$ to $\gamma_{1}(S^1)$ is equal to the tangent vectors of $\gamma_{1}$, the condition (C) is satisfied. Since $Z_{\nabla}$ is basic with respect to $\mathcal{F}$, the flow generated by $Z_{\nabla}$ maps the leaves of $\mathcal{F}$ to the leaves of $\mathcal{F}$. Since the leaves of $\mathcal{F}_{b}$ are the closures of the leaves of $\mathcal{F}$, the flow generated by $Z_{\nabla}$ maps the leaves of $\mathcal{F}_{b}$ to the leaves of $\mathcal{F}_{b}$. Hence $f$ satisfies the condition (A). By the conditions (c), (d) of Proposition~\ref{Proposition : Fiberwise average} and Lemma~\ref{Lemma : Condition d 2} (ii), $(\widetilde{F},\mathcal{F}|_{\widetilde{F}})$ is a Lie foliation. Let $G$ be a connected Lie group such that $(\widetilde{F},\mathcal{F}|_{\widetilde{F}})$ is a $G$-Lie foliation. We can assume the simply connectedness of $G$ as noted in the second paragraph of Section~\ref{Section : Lie foliations}. Let $H$ be the Lie subgroup of $G$ such that $\Lie(H)$ is the structural Lie algebra of the Lie foliation $(F_{b},\mathcal{F}|_{F_{b}})$. We denote the universal cover of $\widetilde{F}$ by 
\begin{equation}
p^{\univ}_{\widetilde{F}} : \widetilde{F}^{\univ} \longrightarrow \widetilde{F}.
\end{equation}
\noindent We recall the notation
\begin{equation}
q_{\widetilde{F}} = \cod (\widetilde{F},\mathcal{F}|_{\widetilde{F}})
\end{equation}
and let 
\begin{equation}
q_{b,\widetilde{F}}=\cod (\widetilde{F},\mathcal{F}_{b}|_{\widetilde{F}}).
\end{equation}
\noindent Note that $\dim G = q_{\widetilde{F}}$ and the codimension of $H$ in $G$ is equal to $q_{b,\widetilde{F}}$. We regard $(\Lie(G)/\Lie(H))^{*}$ as a subset of $\Lie(G)^{*}$ consisting of the elements whose restriction to $\Lie(H)$ is $0$. Fix a basis $\{\overline{\beta}_{1}$, $\overline{\beta}_{2}$, $\cdots$, $\overline{\beta}_{q_{\widetilde{F}}}\}$ of $\Lie(G)^{*}$ so that $\{\overline{\beta}_{1}$, $\overline{\beta}_{2}$, $\cdots$, $\overline{\beta}_{q_{b,\widetilde{F}}}\}$ is a basis of $(\Lie(G)/\Lie(H))^{*}$. Note that $\dev_{G}^{*}\overline{\beta}_{j}$ is $\pi_{1}\widetilde{F}$-invariant. Let $\beta_{j}$ be the $1$-form on $\widetilde{F}$ induced by the $\pi_{1}\widetilde{F}$-invariant $1$-form $\dev_{G}^{*}\overline{\beta}_{j}$ on $\widetilde{F}^{\univ}$ for $1 \leq j \leq q_{\widetilde{F}}$. Then the restriction of $\beta_{j}$ to each leaf of $\widetilde{\mathcal{F}}$ is zero for $1 \leq j \leq q_{b,\widetilde{F}}$. Note that the Maurer-Cartan form of $(\widetilde{F},\mathcal{F}_{b}|_{\widetilde{F}})$ is given by the equation \eqref{Equation : Omega}. Hence, by (V) of Section~\ref{Section : Lie foliations}, we can write
\begin{equation}
\beta_{j} = \sum_{i=1}^{q_{\widetilde{F}}} b_{j}^{i} \omega_{i}|_{\widetilde{F}}
\end{equation}
for each $j$ for some constants $b_{j}^{i}$, $1 \leq i \leq q_{\widetilde{F}}$.

We define a transverse volume form $\mu_{\widetilde{F}/\mathcal{F}_{b}}$ on $(\widetilde{F},\mathcal{F}_{b}|_{\widetilde{F}})$ by 
\begin{equation}\label{Equation : Lie derivative 1}
\mu_{\widetilde{F}/\mathcal{F}_{b}} = \beta_{1} \wedge \beta_{2} \wedge \cdots \wedge \beta_{q_{b,\widetilde{F}}}.
\end{equation}

Since $\mu_{\widetilde{F}/\mathcal{F}_{b}}$ is closed, we have 
\begin{equation}\label{Equation : Lie derivative 2}
L_{Z_{\nabla}} \mu_{\widetilde{F}/\mathcal{F}_{b}} = d\iota_{Z_{\nabla}} \mu_{\widetilde{F}/\mathcal{F}_{b}} = d\iota_{Z_{\nabla}} (\beta_{1} \wedge \beta_{2} \wedge \cdots \wedge \beta_{q_{b,\widetilde{F}}}).
\end{equation}
Here we recall $q_{b,\widetilde{F}}=\cod (\widetilde{F},\mathcal{F}_{b}|_{\widetilde{F}})$. Note that $\omega_{i}(Z_{\nabla})$ is constant on $\widetilde{F}$ by definition of $Z_{\nabla}$. Then $\beta_{j}(Z_{\nabla})$ is also constant on $\widetilde{F}$. We can write $d\beta_{i}$ as a sum of $\beta_{j} \wedge \beta_{k}$ on $M$ as
\begin{equation}
d\beta_{i} = \sum_{1 \leq j < k \leq q} c_{i}^{jk} \beta_{j} \wedge \beta_{k}.
\end{equation}
Then the restriction of $c_{i}^{jk}$ to $\widetilde{F}$ is constant by the condition (c) of Proposition~\ref{Proposition : Fiberwise average}. 
Then there exists a constant $C_{0}$ such that
\begin{equation}\label{Equation : Lie derivative 3}
\begin{array}{l}
d\iota_{Z_{\nabla}} (\beta_{1} \wedge \beta_{2} \wedge \cdots \wedge \beta_{q_{b,\widetilde{F}}})_{x} 
= C_{0} \big( \beta_{1} \wedge \beta_{2} \wedge \cdots \wedge \beta_{q_{b,\widetilde{F}}} \big)_{x}
\end{array}
\end{equation}
for every point $x$ on $\widetilde{F}$. By \eqref{Equation : Lie derivative 2} and \eqref{Equation : Lie derivative 3}, we have
\begin{equation}
(L_{Z_{\nabla}} \mu_{\widetilde{F}/\mathcal{F}_{b}})_{x} = C_{0} (\mu_{\widetilde{F}/\mathcal{F}_{b}})_{x}
\end{equation}
for every point $x$ on $\widetilde{F}$. In the same way, there exists a constant $C_{0}^{t}$ such that 
\begin{equation}\label{Equation : Lie derivative 4}
(L_{Z_{\nabla}} \mu_{\widetilde{F}/\mathcal{F}_{b}})_{x} = C_{0}^{t} (\mu_{\widetilde{F}/\mathcal{F}_{b}})_{x}
\end{equation}
for every point $x$ on $\pi_{\widetilde{\mathcal{F}}}^{-1}(t)$ for each $t$ in $\pi_{\widetilde{\mathcal{F}}} \circ \gamma (S^1)$. By \eqref{Equation : Lie derivative 4}, $f$ satisfies
\begin{equation}\label{Equation : f and mu}
f^{*}\mu_{\widetilde{F}/\mathcal{F}_{b}} = C_{1} \mu_{\widetilde{F}/\mathcal{F}_{b}}
\end{equation}
for a constant $C_{1}$. Let $\widetilde{F}/\mathcal{F}_{b}$ be the leaf space of $(\widetilde{F},\mathcal{F}_{b}|_{\widetilde{F}})$. $\widetilde{F}/\mathcal{F}_{b}$ is a closed manifold. This $\widetilde{F}/\mathcal{F}_{b}$ is orientable by the transverse orientability of $(\widetilde{F},\mathcal{F}_{b}|_{\widetilde{F}})$. Let $\overline{f}$ be the map induced by $f$ on $\widetilde{F}/\mathcal{F}_{b}$. Since $\mu_{\widetilde{F}/\mathcal{F}_{b}}$ is basic, it is a pull back of a volume form $\overline{\mu}_{\widetilde{F}/\mathcal{F}_{b}}$ on $\widetilde{F}/\mathcal{F}_{b}$. We have 
\begin{equation}
\overline{f}^{*}\overline{\mu}_{\widetilde{F}/\mathcal{F}_{b}} = C_{1} \overline{\mu}_{\widetilde{F}/\mathcal{F}_{b}}
\end{equation}
Hence $C_{1}$ is equal to $1$, because $C_{1}$ is equal to the mapping degree of a diffeomorphism $\overline{f} \colon \widetilde{F}/\mathcal{F}_{b} \longrightarrow \widetilde{F}/\mathcal{F}_{b}$. It follows from \eqref{Equation : f and mu} that $\nabla$ satisfies the condition (B).
\end{proof}

A section of $C^{\infty}(\wedge^{k}T^{*}\mathcal{F})$ is called a leafwise $k$-form on $(M,\mathcal{F})$. If $k=\dim \mathcal{F}$, a leafwise $k$-form is called a leafwise volume form on $(M,\mathcal{F})$. The wedge product induces a natural operation $C^{\infty}(\wedge^{k}T^{*}\mathcal{F}) \otimes C^{\infty}(\wedge^{q}(TM/T\mathcal{F})^{*}) \longrightarrow C^{\infty}(\wedge^{q+k}T^{*}M)$.

\begin{lemma}\label{Lemma : Bundle over S^1}
Let $\gamma \colon S^1 \longrightarrow M$ be a smooth embedding in $M$. Assume that $\gamma$ is transverse to the fibers of $\pi_{\widetilde{\mathcal{F}}}$, and that $\pi_{\widetilde{\mathcal{F}}} \circ \gamma$ is an embedding. Let $M' = \pi_{\widetilde{\mathcal{F}}}^{-1}(\pi_{\widetilde{\mathcal{F}}}(\gamma(S^1)))$. Fix a fiber $\widetilde{F}$ of $\pi_{\widetilde{\mathcal{F}}}|_{M'}$ and a point $x_{0}$ on $\widetilde{F}$. Let $\nabla$ be a flat connection on $\pi_{\widetilde{\mathcal{F}}}|_{M'}$. Let $f \colon \widetilde{F} \longrightarrow \widetilde{F}$ be the holonomy of the flat connection $\nabla$ with respect to the path $\pi_{\widetilde{\mathcal{F}}} \circ \gamma$. For $x$ on $\widetilde{F}$, let $\gamma_{x}^{\nabla}$ be the lift of the path $\pi_{\widetilde{\mathcal{F}}} \circ \gamma$ to $M'$ such that $\gamma_{x}(0)=x$ and $\gamma_{x}^{\nabla}$ is parallel to $\nabla$. Let $\alpha$ be a closed $1$-form on $M$ such that $[\alpha|_{\widetilde{F}}]=0$ in $H^{1}(\widetilde{F};\mathbb{R})$ and $\alpha|_{L_{b}}=0$ on each leaf $L_{b}$ of $(\widetilde{F},\mathcal{F}_{b}|_{\widetilde{F}})$.

We assume that $\nabla$ satisfies the conditions (A), (B) and (C) in the statement of Lemma~\ref{Lemma : Existence of nabla}. Let $\vol_{\widetilde{F}}$ be a volume form on $\widetilde{F}$. Then we have
\begin{equation}\label{Equation : Average}
\int_{\widetilde{F}} \left( \int_{\gamma_{x}^{\nabla}} \alpha \right) \vol_{\widetilde{F}} = \left( \int_{\widetilde{F}} \vol_{\widetilde{F}} \right) \left( \int_{\gamma_{x_{0}}^{\nabla}} \alpha \right).
\end{equation}
\end{lemma}
\begin{proof}
First, we shall show that there exists an isotopy $\{\phi_{s}\}_{s \in [0,1]}$ on $\widetilde{F}$ such that 
\begin{enumerate}
\item $\phi_{0}=\id_{\widetilde{F}}$, 
\item $f \circ \phi_{1}$ preserves $\vol_{\widetilde{F}}$,
\item $\phi_{s}$ maps each leaf of $\widetilde{\mathcal{F}}$ to itself and 
\item $\phi_{s}$ fixes $x_{0}$
\end{enumerate}
by the assumption and a leafwise version of Moser's argument in below. The leafwise version of Moser's argument was used by Ghys in \cite{Ghys} and by Hector, Macias and Saralegui in \cite{Hector Macias Saralegui}. Let $\eta_{\mathcal{F}_{b}}$ be the leafwise volume form on $(\widetilde{F},\mathcal{F}_{b}|_{\widetilde{F}})$ such that
\begin{equation}\label{Equation : Volume 1}
\eta_{\mathcal{F}_{b}} \wedge \mu_{\widetilde{F}/\mathcal{F}_{b}} = \vol_{\widetilde{F}}.
\end{equation}
Since each leaf $L_{b}$ of $\mathcal{F}_{b}|_{\widetilde{F}}$ is compact and oriented by the assumption, $f$ maps the fundamental class of $L_{b}$ to the fundamental class of $f(L_{b})$. Then we have a leafwise $(\dim \mathcal{F}_{b} -1)$-form $\sigma$ on $(\widetilde{F},\mathcal{F}_{b}|_{\widetilde{F}})$ such that $d\sigma = f^{*}\eta_{\mathcal{F}_{b}} - \eta_{\mathcal{F}_{b}}$. By adding a closed $(\dim \mathcal{F}_{b} -1)$-form supported on an open neighborhood of $x_{0}$ to $\sigma$, we can modify $\sigma$ so that $\sigma_{x_{0}}=0$ and $d\sigma = f^{*}\eta_{\mathcal{F}_{b}} - \eta_{\mathcal{F}_{b}}$ are satisfied. Since $\eta_{\mathcal{F}_{b}}$ is a leafwise volume form on $(\widetilde{F},\mathcal{F}_{b}|_{\widetilde{F}})$, there exists a vector field $Y$ on $\widetilde{F}$ tangent to leaves of $\widetilde{\mathcal{F}}$ such that $-\iota_{Y}\eta_{\mathcal{F}_{b}} = \sigma$. Let $\{\phi_{s}\}_{s \in \mathbb{R}}$ be the flow generated by $Y$. Let $\eta_{s} = \eta_{\mathcal{F}_{b}} + sd\sigma$. We have
\begin{equation}
\frac{d(\phi_{s}^{*}\eta_{s})}{ds} = \phi_{s}^{*}\Big(L_{Y}\eta_{s} + \frac{d\eta_{s}}{ds}\Big) = \phi_{s}^{*}d(\iota_{Y}\eta_{\mathcal{F}_{b}} + \sigma) = 0.
\end{equation}
Hence we have 
\begin{equation}\label{Equation : Volume 2}
\phi_{1}^{*} f^{*} \eta_{\mathcal{F}_{b}} = \phi_{1}^{*} \eta_{1}= \eta_{0} = \eta_{\mathcal{F}_{b}}.
\end{equation}
Thus $f \circ \phi_{1}$ preserves $\eta_{\mathcal{F}_{b}}$. Here, $\phi_{s}$ maps each leaf of $\widetilde{\mathcal{F}}$ to itself, because $Y$ is tangent to leaves of $\widetilde{\mathcal{F}}$. Then clearly $\phi_{1}$ preserves the transverse volume form $\mu_{\widetilde{F}/\mathcal{F}_{b}}$. Since $f$ preserves $\mu_{\widetilde{F}/\mathcal{F}_{b}}$ by the assumption, by \eqref{Equation : Volume 1} and \eqref{Equation : Volume 2}, we have
\begin{equation}
(f \circ \phi_{1})^{*} \vol_{\widetilde{F}} = (f \circ \phi_{1})^{*} \eta_{\mathcal{F}_{b}} \wedge (f \circ \phi_{1})^{*}\mu_{\widetilde{F}/\mathcal{F}_{b}} = \eta_{\mathcal{F}_{b}} \wedge \mu_{\widetilde{F}/\mathcal{F}_{b}} = \vol_{\widetilde{F}}.
\end{equation}
Hence $f \circ \phi_{1}$ preserves $\vol_{\widetilde{F}}$. Since $\sigma_{x_{0}}=0$, we have $Y_{x_{0}}=0$. This implies that $\phi_{s}$ fixes $x_{0}$.

Using $\{\phi_{s}\}_{s \in [0,1]}$, we can construct a smooth family $\{\nabla_{s}\}_{s \in [0,1]}$ of flat connections on $\pi_{\widetilde{\mathcal{F}}}|_{M'}$ such that $\nabla_{0}=\nabla$ and the holonomy of $\nabla_{s}$ with respect to $\pi_{\widetilde{\mathcal{F}}} \circ \gamma$ is $f \circ \phi_{s}$. Since each $\phi_{s}$ fixes $x_{0}$, we can take $\{\nabla_{s}\}_{s \in [0,1]}$ so that
\begin{equation}\label{Equation : gamma x 0}
\gamma_{x_{0}}^{\nabla} = \gamma_{x_{0}}^{\nabla_{s}}
\end{equation}
for $0 \leq s \leq 1$. For $x$ in $\widetilde{F}$, let $\gamma_{x}^{\nabla_{1}}$ be the lift of the path $\pi_{\widetilde{\mathcal{F}}} \circ \gamma$ to $M$ such that $\gamma_{x}(0)=x$ and $\gamma_{x}$ is parallel to $\nabla_{1}$. We take a function $h$ on $\widetilde{F}$ so that $dh = \alpha|_{\widetilde{F}}$. For each point $x$ in $\widetilde{F}$, we have
\begin{equation}\label{Equation : Stokes}
\begin{array}{c}
\int_{\gamma_{x}^{\nabla_{0}}} \alpha + \big( h(\phi_{1} \circ f(x)) - h(f(x)) \big) - \int_{\gamma_{x}^{\nabla_{1}}} \alpha = 0
\end{array}
\end{equation}
by the Stokes' theorem. By the assumption on $\alpha$, the restriction of $h$ to each leaf of $\mathcal{F}_{b}|_{\widetilde{F}}$ is constant. Since $\phi_{1} \circ f$ maps each leaf of $\widetilde{\mathcal{F}}$ to itself, we have $h(\phi_{1} \circ f(x)) - h(f(x))=0$. Hence, by \eqref{Equation : Stokes}, we have
\begin{equation}
\int_{\gamma_{x}^{\nabla_{0}}} \alpha - \int_{\gamma_{x}^{\nabla_{1}}} \alpha = 0
\end{equation}
for each point $x$ on $\widetilde{F}$. Then we have
\begin{equation}\label{Equation : Stokes 1}
\int_{\widetilde{F}} \left( \int_{\gamma_{x}^{\nabla_{0}}} \alpha \right) \vol_{\widetilde{F}} - \int_{\widetilde{F}} \left( \int_{\gamma_{x}^{\nabla_{1}}} \alpha \right) \vol_{\widetilde{F}} = \int_{\widetilde{F}} \left( \int_{\gamma_{x}^{\nabla_{0}}} \alpha - \int_{\gamma_{x}^{\nabla_{1}}} \alpha \right) \vol_{\widetilde{F}} \\
 = 0.
\end{equation}
For each point $x$ on $\widetilde{F}$, we have
\begin{equation}\label{Equation : Stokes 3}
\begin{array}{c}
\big( h(x) - h(x_{0}) \big) + \int_{\gamma_{x}^{\nabla_{1}}} \alpha + \big( h(f \circ \phi_{1}(x_{0})) - h(f(x)) \big) - \int_{\gamma_{x_{0}}^{\nabla_{1}}} \alpha = 0
\end{array}
\end{equation}
by the Stokes' theorem. We get
\begin{equation}\label{Equation : Stokes 2}
\int_{\widetilde{F}} \left( \int_{\gamma_{x}^{\nabla_{1}}} \alpha \right) \vol_{\widetilde{F}} - \left( \int_{\widetilde{F}} \vol_{\widetilde{F}} \right) \left( \int_{\gamma_{x_{0}}^{\nabla}} \alpha \right) = 0
\end{equation}
as follows:
\begin{align*}
& \textstyle \int_{\widetilde{F}} \left( \int_{\gamma_{x}^{\nabla_{1}}} \alpha \right) \vol_{\widetilde{F}} - \left( \int_{\widetilde{F}} \vol_{\widetilde{F}} \right) \left( \int_{\gamma_{x_{0}}^{\nabla}} \alpha \right) \\
\textstyle = & \textstyle \int_{\widetilde{F}} \left( \int_{\gamma_{x}^{\nabla_{1}}} \alpha \right) \vol_{\widetilde{F}} - \left( \int_{\widetilde{F}} \vol_{\widetilde{F}} \right) \left( \int_{\gamma_{x_{0}}^{\nabla_{1}}} \alpha \right) \\
\textstyle = & \textstyle \int_{\widetilde{F}} \left( \int_{\gamma_{x}^{\nabla_{1}}} \alpha \right) \vol_{\widetilde{F}} - \int_{\widetilde{F}} \left( \int_{\gamma_{x_{0}}^{\nabla_{1}}} \alpha \right) \vol_{\widetilde{F}} \\
\textstyle = & \textstyle \int_{\widetilde{F}} \left( \int_{\gamma_{x}^{\nabla_{1}}} \alpha - \int_{\gamma_{x_{0}}^{\nabla_{1}}} \alpha \right) \vol_{\widetilde{F}} \label{Equation : Stokes 2} \\
\textstyle = & \textstyle \int_{\widetilde{F}} \Big( -  \big( h(x) - h(x_{0}) \big) - \big( h(f \circ \phi_{1}(x_{0})) - h(f \circ \phi_{1}(x)) \big) \Big) \vol_{\widetilde{F}} \\
\textstyle = & \textstyle - \int_{\widetilde{F}} (h(x) - h(x_{0})) \vol_{\widetilde{F}} + \int_{\widetilde{F}} ( h(x) - h(x_{0})) \phi_{1}^{*}f^{*}\vol_{\widetilde{F}} \\
\textstyle = & \textstyle \, 0,
\end{align*}
where we used \eqref{Equation : gamma x 0} in the first equality, \eqref{Equation : Stokes 3} in the fourth equality and $\phi_{1}^{*}f^{*}\vol_{\widetilde{F}} =\vol_{\widetilde{F}}$ in the last equality. The equation \eqref{Equation : Average} follows from \eqref{Equation : Stokes 1} and \eqref{Equation : Stokes 2}.
\end{proof}

\subsection{$\xi(\mathcal{F})=[\widetilde{\kappa}_{b}]$ on $M$}

We will show a lemma which will be used to complete the proof of Proposition~\ref{Proposition : Fiberwise average} (ii). Note that $\alpha$ will be considered to be $\xi(\mathcal{F})-[\widetilde{\kappa}_{b}]$ in the application in Section 5.6.

\begin{lemma}\label{Lemma : Equality on M'}
Assume that the conditions (a), (b), (c) and (d) of Proposition~\ref{Proposition : Fiberwise average} are satisfied. Assume that $\gamma_{0}$ in $\pi_1(M,x_{0})$ satisfies the following conditions: $\gamma_{0}$ is represented by a smooth path $l_{0}' \colon [0,1] \longrightarrow M$ which factors a smooth embedding $l_{0} \colon S^1 \longrightarrow M$ and is transverse to the fibers of $\pi_{\widetilde{\mathcal{F}}}$, and $\pi_{\widetilde{\mathcal{F}}} \circ l_{0}$ is a smooth embedding. Let $\alpha$ be a closed basic $1$-form on $(M,\mathcal{F})$ which satisfies the following conditions:
\begin{enumerate}
\item $\alpha|_{L_{b}} = 0$ for each leaf $L_{b}$ of $(\widetilde{F},\mathcal{F}_{b}|_{\widetilde{F}})$.
\item $\alpha|_{\widetilde{F}}$ is exact.
\item For any submanifold $M'$ of $M$ which is a union of fibers of $\pi_{\widetilde{\mathcal{F}}}$, 
\begin{equation}
\int_{M'} \Big( \int_{\gamma_{x}} \alpha \Big) \vol_{M'}(x)= 0
\end{equation}
is satisfied for the flow $\{\phi_{t}\}_{t \in [0,1]}$ generated by every vector field $X$ such that $\omega_{i}(X)$ is constant on each fiber of $\pi_{\widetilde{\mathcal{F}}}$ for each $i$, where $\gamma_{x}$ is the orbit of $x$ of the flow $\{\phi_{t}\}_{t \in [0,1]}$.
\end{enumerate}
Then we have
\begin{equation}
\int_{\gamma_{0}} \alpha = 0.
\end{equation}
\end{lemma}

\begin{proof}
Let
\begin{equation}
K=\pi_{\widetilde{\mathcal{F}}} \circ l_{0} (S^1), \quad M'=\pi_{\widetilde{\mathcal{F}}}^{-1}(K).
\end{equation} 
$M'$ is a submanifold of $M$ which is a union of fibers of $\pi_{\widetilde{\mathcal{F}}}$ by the assumption on $\gamma_{0}$.

By the condition (a), we have
\begin{equation}\label{Equation : M'1}
\int_{M'} \Big( \int_{\gamma_{x}} \alpha \Big) \vol_{M'}(x) = \int_{K} \left( \int_{\pi_{\widetilde{\mathcal{F}}} } \Big( \int_{\gamma_{x}} \alpha \Big) \vol_{\widetilde{\mathcal{F}}} \right) \vol_{K}.
\end{equation}

By the conditions (c), (d) and Lemma~\ref{Lemma : Existence of nabla}, there exists a flat connection $\nabla$ on the fiber bundle $M' \longrightarrow K$ which satisfies the conditions (A), (B), (C) and (D) in the statement of Lemma~\ref{Lemma : Existence of nabla}. By condition (D) and the assumption (iii), we have
\begin{equation}\label{Equation : M'0}
\int_{M'} \Big( \int_{\gamma_{x}} \alpha \Big) \vol_{M'}(x) = 0.
\end{equation}
Since the assumptions of Lemma~\ref{Lemma : Bundle over S^1} are satisfied by the conditions (A), (B) and (C), we have
\begin{equation}\label{Equation : M'2}
\int_{\pi_{\widetilde{\mathcal{F}}}^{-1}(t)} \Big( \int_{\gamma_{x}} \alpha \Big) \vol_{\pi_{\widetilde{\mathcal{F}}}^{-1}(t)} = \left( \int_{\pi_{\widetilde{\mathcal{F}}}^{-1}(t)} \vol_{\pi_{\widetilde{\mathcal{F}}}^{-1}(t)} \right)  \left( \int_{\gamma_{x_{0}}} \alpha \right)
\end{equation}
for each $t$ in $K$ by Lemma~\ref{Lemma : Bundle over S^1}. By condition (b) and Lemma~\ref{Lemma : Constant volume}, the volume of fibers of $\pi_{\widetilde{\mathcal{F}}}$ is constant. By \eqref{Equation : M'1}, \eqref{Equation : M'0} and \eqref{Equation : M'2}, we have
\begin{equation}
\int_{\gamma_{x_{0}}} \alpha = 0.
\end{equation}
Hence Lemma~\ref{Lemma : Equality on M'} is proved.
\end{proof}

\subsection{Proof of Proposition~\ref{Proposition : Fiberwise average} (ii)}

By the homotopy exact sequence of the fiber bundle $\pi_{\widetilde{\mathcal{F}}}$, we have an exact sequence 
\begin{equation}
\xymatrix{\pi_2(M/\widetilde{\mathcal{F}},\pi_{\widetilde{\mathcal{F}}}(x_{0})) \ar[r] & \pi_1(\widetilde{F},x_{0}) \ar[r]^{\iota} & \pi_1(M,x_{0}) \ar[r]^>>>>>{(\pi_{\widetilde{\mathcal{F}}})_{*}} & \pi_1(M/\widetilde{\mathcal{F}},\pi_{\widetilde{\mathcal{F}}}(x_{0})) \ar[r] & 0.}
\end{equation}
By Lemmas~\ref{Lemma : Integral along flows} and \ref{Lemma : Fiber of pi tilde}, we have $\kappa_{b}|_{L_{b}} - \widetilde{\kappa}_{b}|_{L_{b}}$ for every leaf $L_{b}$ of $(\widetilde{F},\mathcal{F}_{b}|_{\widetilde{F}})$ and $[\kappa_{b}|_{\widetilde{F}}] - [\widetilde{\kappa}_{b}|_{\widetilde{F}}]=0$. This implies that $[\kappa_{b}] - [\widetilde{\kappa}_{b}]$ vanishes on the image of $\iota$. Then Lemma~\ref{Lemma : Equality on M'} implies $\int_{\gamma} (\kappa_{b} - \widetilde{\kappa}_{b}) = 0$ for $\gamma$ in $\pi_1(M,x_{0})$ such that 
\begin{itemize}
\item $\gamma$ is transverse to the fibers of $\pi_{\widetilde{\mathcal{F}}}$ and
\item $\pi_{\widetilde{\mathcal{F}}} \circ \gamma$ is an embedding.
\end{itemize}
Note that $\pi_1(M/\widetilde{\mathcal{F}},\pi_{\widetilde{\mathcal{F}}}(x_{0}))$ is generated by the loops of the form $(\pi_{\widetilde{\mathcal{F}}})_{*} \gamma$ where $\gamma$ runs on all of the closed paths satisfying these two conditions. Hence we have $[\kappa_{b}] = [\widetilde{\kappa}_{b}]$ in $H^1(M;\mathbb{R})$.

\section{Continuity of the \'{A}lvarez classes}\label{Section : Continuity}

We will show Theorem~\ref{Theorem : Continuity} for smooth families of orientable transversely parallelizable foliations using Proposition~\ref{Proposition : Fiberwise average}. We will use the notation of Definition~\ref{definition: families of foliations} (ii).

\subsection{A family version of the Molino theory}

Let $U$ be a connected open set in $\mathbb{R}^{\ell}$ which contains $0$. Let $M$ be a closed manifold, and $\{\mathcal{F}^{t}\}_{t \in U}$ be a smooth family of orientable transversely parallelizable foliations of $M$ over $U$ given by a smooth foliation $\mathcal{F}^{\amb}$ of $M \times U$. We define a distribution $D$ on $M \times U$ by
\begin{equation}\label{Equation : D}
D_{(x,t)}=\{v \in T_{(x,t)}(M \times U) \, | \, vf(x,t)=0, \forall f \in C^{\infty}_{b}(M \times U,\mathcal{F}^{\amb})\},
\end{equation}
where $C^{\infty}_{b}(M \times U,\mathcal{F}^{\amb})$ is the space of basic functions on $(M \times U, \mathcal{F}^{\amb})$. By the standard argument of the Molino theory on $D$, we shall obtain the following properties of $D$ similar to the properties of the basic foliation of transversely parallelizable foliations:
\begin{lemma}\label{Lemma : Molino}
\begin{enumerate}
\item $(M \times U,\mathcal{F}^{\amb})$ is fiberwise transitive, that is, for each two points $(x,t)$ and $(y,t)$ in $M \times U$ with the same second coordinates, there exists a diffeomorphism $f$ of $M \times U$ which preserves $\mathcal{F}^{\amb}$ and satisfies $f(x)=y$.
\item The dimension of $D_{(x,t)}$ is independent of $x$.
\item $D|_{M \times \{t\}}$ is integrable, and we have a foliation $\mathcal{D}^{t}$ of $M \times \{t\}$ defined by $D|_{M \times \{t\}}$.
\item The leaf space $(M \times \{t\})/\mathcal{D}^{t}$ is a closed manifold and the canonical projection $M \times \{t\} \longrightarrow (M \times \{t\})/\mathcal{D}^{t}$ is a smooth fiber bundle with compact fibers for each $t$.
\end{enumerate}
\end{lemma}

\begin{proof}
Fix $t_{0}$ on $U$. Let $\hat{X}^{j}_{\amb}$ be a vector field on $M \times U$ which is projected to $X^{j}_{\amb}$ by the canonical projection $C^{\infty}(TM) \longrightarrow C^{\infty}(TM/T\mathcal{F})$. By the compactness of $M$, for a relative compact open neighborhood $U'$ of $t_{0}$ in $U$, each $\hat{X}^{i}_{\amb}$ generates a flow $\{\phi_{i}^{s}\}_{s \in \mathbb{R}}$ on $M \times U'$. By the proof of Theorem 4.8 of Moerdijk and Mr\v{c}un \cite{Moerdijk Mrcun}, for each two points $(x,t_{0})$ and $(y,t_{0})$ in $M \times \{t_{0}\}$, there exists a diffeomorphism $f$ of $M \times \{t_{0}\}$ which is a composition of $\phi_{1}^{s_{1}}|_{M \times \{t_{0}\}}$, $\phi_{2}^{s_{2}}|_{M \times \{t_{0}\}}$, $\cdots$, $\phi_{q}^{s_{q}}|_{M \times \{t_{0}\}}$ for some $s_{i}$ and diffeomorphisms of $M \times \{t_{0}\}$ preserving each leaf of $\mathcal{F}^{t_{0}}$ whose supports are contained in a foliated chart of $\mathcal{F}^{t_{0}}$.  Since $X^{i}_{\amb}$ is basic with respect to $(M \times U', \mathcal{F}^{\amb})$, $\phi_{i}^{s_{i}}$ preserves $\mathcal{F}^{\amb}$. We can extend diffeomorphisms of $M \times \{t_{0}\}$ which preserve each leaf of $\mathcal{F}^{t_{0}}$ and whose supports are contained in a foliated chart of $\mathcal{F}^{t_{0}}$ to diffeomorphisms of $M \times U'$ preserving each leaf of $\mathcal{F}^{\amb}$. Then $f$ extends to a diffeomorphism of $M \times U'$ preserving $\mathcal{F}^{\amb}$ as a composite of $\phi_{1}^{s_{1}}$, $\phi_{2}^{s_{2}}$, $\cdots$, $\phi_{q}^{s_{q}}$ and diffeomorphisms of $M \times U'$ preserving each leaf of $\mathcal{F}^{\amb}$. This proves (i). 

Properties (ii), (iii) and (iv) follow from (i) and the proof of Theorem 4.3 of Moerdijk and Mr\v{c}un \cite{Moerdijk Mrcun}. We write down the proof for the sake of completeness. Since $\mathcal{D}|_{M \times U'}$ is preserved by a diffeomorphism of $M \times U'$ preserving $\mathcal{F}^{\amb}$ by definition, (ii) directly follows from (i). Let $(x,t_{0})$ be a point on $M \times \{t_{0}\}$. By (ii), $\mathcal{D}|_{M \times \{t_{0}\}}$ is a vector bundle on $M \times \{t_{0}\}$. Let $Z_{1}$ and $Z_{2}$ be two local sections of $\mathcal{D}|_{M \times \{t_{0}\}}$ defined near $(x,t_{0})$. For every $f$ in $C^{\infty}_{b}(M \times U,\mathcal{F}^{\amb})$, we have $[Z_{1},Z_{2}] f = Z_{1} Z_{2} f - Z_{2} Z_{1} f = 0$. Then $[Z_{1},Z_{2}]$ is a section of $\mathcal{D}|_{M \times \{t\}}$. Hence $\mathcal{D}|_{M \times \{t\}}$ is integrable, which proves (iii). Let $L$ be a leaf of the foliation defined by $\mathcal{D}|_{M \times \{t_{0}\}}$. Let $x$ be a point in $L$. Let $k(t_{0})=\dim M - \dim D^{t_{0}}$. By definition of $\mathcal{D}$, there exists basic functions $f_{1}$, $f_{2}$, $\cdots$, $f_{k(t_{0})}$ on $(M \times U,\mathcal{F}^{\amb})$ such that $df_{1} \wedge df_{2} \wedge \cdots \wedge df_{k(t_{0})}$ is nonzero at $x$. Since each $f_{i}$ is basic, $df_{1} \wedge df_{2} \wedge \cdots \wedge df_{k(t_{0})}$ is nowhere vanishing on an open saturated neighborhood $U'$ of $L$ in $(M \times U,\mathcal{F}^{\amb})$. Then the map $\phi$ defined by
\begin{equation}
\begin{array}{cccc}
\phi \colon & U' & \longrightarrow & \mathbb{R}^{k(t_{0})} \\
            & z & \longmapsto     & (f_{1}(z),f_{2}(z),\cdots,f_{k(t_{0})}(z))
\end{array}
\end{equation}
is a submersion such that one of the fibers of $\phi$ is equal to $L$. Shrinking $U'$, we can assume that the fibers of $\phi$ are connected. Since each fiber of $\phi$ is saturated by $D^{t}$, each leaf of $D^{t}$ near $L$ coincides with a fiber of $\phi$. Then $\phi$ gives a local trivialization of a fiber bundle. Hence (iv) is proved.
\end{proof}

Since $D$ is a closed subset of $T(M \times U)$ by definition of $D$, the dimension of $D_{(x,t)}$ is upper semicontinuous with respect to $t$. If the dimension of $D_{(x,t)}$ is constant with respect to $t$, then the leaves of $\mathcal{D}$ are fibers of a smooth submersion whose restriction to $M \times \{t\}$ is equal to the canonical map $M \times \{t\} \longrightarrow (M \times \{t\})/\mathcal{D}^{t}$ for each $t$. In this case, the continuity of the \'{A}lvarez class follows without Proposition~\ref{Proposition : Fiberwise average} (see Example~\ref{Example : smooth family of Molino sheaves}). When the dimension of $D$ jumps, we have only a family of smooth proper submersions defined by $D$ which changes discontinuously with respect to $t$.

\subsection{Verification of the conditions of Proposition~\ref{Proposition : Fiberwise average}}\label{Section : Verification}

 Let $U$ be a connected open set in $\mathbb{R}^{\ell}$ which contains $0$. Let $M$ be a closed manifold, and $\{\mathcal{F}^{t}\}_{t \in U}$ be a smooth family of orientable transversely parallelizable foliations of $M$ over $U$ given by a smooth foliation $\mathcal{F}^{\amb}$ of $M \times U$. We define a distribution $D$ on $M \times U$ by \eqref{Equation : D}. By Lemma~\ref{Lemma : Molino} (iv), $D|_{M \times \{t\}}$ defines a foliation $\mathcal{D}^{t}$ of $M \times \{t\}$ whose leaves are fibers of a submersion. We denote the projection $M \times \{0\} \longrightarrow (M \times \{0\})/\mathcal{D}^{0}$ by $\pi_{\widetilde{\mathcal{F}}}^{0}$.

 To apply Proposition~\ref{Proposition : Fiberwise average} to our situation, we prepare two lemmas.

\begin{lemma}\label{Lemma : proper submersion}
There exist an open neighborhood $U'$ of $0$ in $U$ and a smooth proper submersion $\pi_{\widetilde{\mathcal{F}}}^{\amb} \colon M \times U' \longrightarrow (M \times \{0\})/\mathcal{D}^{0}$ such that 
\begin{enumerate}
\item $\pi_{\widetilde{\mathcal{F}}}^{\amb}|_{M \times \{0\}} = \pi_{\widetilde{\mathcal{F}}}^{0}$ and
\item each fiber of $\pi_{\widetilde{\mathcal{F}}}|_{M \times \{t\}}$ is saturated by the leaves of $\mathcal{F}^{t}$ for each $t$ in $U'$.
\end{enumerate}
\end{lemma}

\begin{proof}
Let $k=\dim M - \dim D^{0}$. For each point $(x,0)$ on $M \times \{0\}$, there exists a $k$-tuple of leafwise constant functions $f_{x1}$, $f_{x2}$, $\cdots$, $f_{xk}$ globally defined on $(M \times U,\mathcal{F}^{\amb})$ such that $(df_{x1} \wedge df_{x2} \wedge \cdots \wedge df_{xk})_{(x,0)}$ is nonzero by definition of $D$. Then $df_{x1} \wedge df_{x2} \wedge \cdots \wedge df_{xk}$ is nowhere vanishing on an open saturated neighborhood $V_{x}$ of $(x,0)$ in $(M \times U,\mathcal{F}^{\amb})$. We define a map $\phi_{x}$ by
\begin{equation}
\begin{array}{cccc}
\phi_{x} \colon & V_{x} & \longrightarrow & \mathbb{R}^{k} \\
                &  z    & \longmapsto     & (f_{x1}(z), f_{x2}(z), \cdots, f_{xk}(z)). 
\end{array}
\end{equation}
This $\phi_{x}$ is a submersion, because $df_{x1} \wedge df_{x2} \wedge \cdots \wedge df_{xk}$ has no zero on $V_{x}$. We can assume that the fibers of $\phi_{x}$ is connected after shrinking $V_{x}$. Let $V'_{x} = \phi_{x}^{-1}\big(\phi_{x}(V_{x} \cap (M \times \{0\})) \big)$. This $V'_{x}$ is also an open neighborhood of $x$ in $M \times U$. Since $\phi_{x}(V'_{x}) = \phi_{x}(V_{x})$, for each point $z$ on $V'_{x}$, there exists a leaf $L_{z}$ of $D^{0}$ such that $\phi_{x}(z)=\phi_{x}(L_{z})$. Since the fibers of $\phi_{x}$ are connected, $L_{z}$ is unique. We define a map by 
\begin{equation}
\begin{array}{cccc}
\psi_{x} \colon & V'_{x} & \longrightarrow & V'_{x} /\mathcal{D}^{0} \\
                &  z    & \longmapsto     & L_{z}. 
\end{array}
\end{equation}
$\psi_{x}$ is a smooth submersion which maps each leaf of $\mathcal{F}^{t}$ to a point. Note that $\psi_{x}|_{M \times \{0\}}$ is the restriction of the projection $\pi_{\widetilde{\mathcal{F}}}^{0}$ to $M \times \{0\}$ by definition. It follows that $\psi_{x}|_{M \times \{t\}}$ is a submersion, because $\psi_{x}|_{M \times \{t\}}$ is of the same rank with $\psi_{x}|_{M \times \{0\}}$. By the compactness of $M$, there exists a finite set of points $\{x_{j}\}_{j=1}^{n}$ such that $\cup_{j=1}^{n}V'_{x_{j}}$ contains $M \times \{0\}$. There exists an open neighborhood $U_{1}$ of $0$ in $U$ such that $M \times U_{1}$ is contained in $\cup_{j=1}^{n}V'_{x_{j}}$. Let $\{\rho_{j}\}_{j=1}^{n}$ be a partition of unity on $(M \times \{0\}) / \mathcal{D}^{0}$ with respect to a covering $\{\pi\big(V'_{x_{j}} \cap (M \times \{0\}) \big)\}_{j=1}^{n}$. We fix a smooth embedding $\iota \colon M \times \{0\} \longrightarrow \mathbb{R}^{m}$ to the $m$-dimensional Euclidean space. We define a map $\Psi_{1}$ by
\begin{equation}
\begin{array}{cccc}
\Psi_{1} \colon & M \times U_{1} & \longrightarrow & \mathbb{R}^{m} \\
            & z & \longmapsto & \sum_{j=1}^{n} \rho_{j}(\psi_{x_{sj}}(z)) \iota (\psi_{x_{sj}}(z)).
\end{array}
\end{equation}
Note that each leaf of $\mathcal{F}^{t}$ is mapped to a point by $\Psi_{1}$ by definition. $\Psi|_{M \times \{0\}}$ is equal to $j \circ \pi_{\mathcal{D}}^{0}$ by definition. Since $\psi_{x}|_{M \times \{t\}}$ is a submersion to $(M \times \{0\})/\widetilde{\mathcal{F}}$ and $\iota$ is an embedding, there exists an open neighborhood $U_{2}$ of $0$ in $U_{1}$ such that $\Psi_{1}|_{M \times \{t\}}$ is a map of constant rank for every $t$ in $U_{2}$. Hence $\Psi_{1}(M \times \{t\})$ is a smooth submanifold of $\mathbb{R}^{m}$, and $\Psi_{1}|_{M \times \{t\}}$ is a smooth submersion on the image for $t$ in $U_{2}$. There exists an open neighborhood $U'$ of $0$ in $U_{2}$ such that $\Psi_{1}(M \times \{t\})$ is the image of a section of the normal bundle in a tubular neighborhood $W$ of $\iota(M \times \{0\})$ for every $t$ in $U'$. We denote the projection $W \longrightarrow \iota(M \times \{0\})$ of the tubular neighborhood by $p_{W}$. Let $\pi_{\widetilde{\mathcal{F}}}^{\amb} = p_{W} \circ j \circ \Psi_{1}|_{M \times U'}$. Then $\pi_{\widetilde{\mathcal{F}}}^{\amb}$ is a submersion and an extension of $\pi_{\widetilde{\mathcal{F}}}^{0}$ which satisfies the given conditions.
\end{proof}

Let $\pi_{\widetilde{\mathcal{F}}}^{t}=\pi_{\widetilde{\mathcal{F}}}^{\amb}|_{M \times \{t\}}$ for $t$ in $U'$. We write $\widetilde{\mathcal{F}}^{t}$ for a foliation  of $M \times \{t\}$ defined by the fibers of $\pi_{\widetilde{\mathcal{F}}}^{t}$.

\begin{lemma}\label{Lemma : tilde g}
There exists a smooth family $\{g^{t}\}_{t \in U'}$ of Riemannian metrics on $M$ such that 
\begin{enumerate}
\item $g^{t}$ is bundle-like with respect to both of $(M,\mathcal{F}^{t})$ and $(M,\widetilde{\mathcal{F}}^{t})$ and 
\item the leaves of $\widetilde{\mathcal{F}}^{t}$ are minimal submanifolds of $(M \times \{t\},g^{t})$ for each $t$ in $U'$.
\end{enumerate}
In particular, the conditions (a) and (b) of Proposition~\ref{Proposition : Fiberwise average} are satisfied by $M \times \{0\}$, $\mathcal{F}^{0}$, $\widetilde{\mathcal{F}}^{0}$ and $g^{0}$.
\end{lemma}

\begin{proof}
It is well known that a Riemannian foliation $\mathcal{G}$ on a closed manifold $N$ defined by a proper submersion is minimizable. For example, see Corollary 2 of Haefliger \cite{Haefliger}. Then there exists a Riemannian metric $g^{\amb}_{1}$ on $M \times U'$ which is bundle-like with respect to the foliation defined by the fibers of $\pi_{\widetilde{\mathcal{F}}}$ and each leaf of $\widetilde{\mathcal{F}}$ is a minimal submanifold of $(M \times U', g^{\amb}_{1})$. Let $g^{t}_{1}=g^{\amb}_{1}|_{M \times \{t\}}$. Then the leaves of $\widetilde{\mathcal{F}}^{t}$ are minimal submanifolds of $(M \times \{t\}, g^{t}_{1})$. Let $\widetilde{\chi}^{t}$ be the characteristic form of $(M \times \{t\},\widetilde{\mathcal{F}}^{t},g_{1}^{t})$. We can take a family of metrics $\{g^{t}_{2}\}_{t \in U}$ on a family of vector bundles $T(M \times \{t\})/T\mathcal{F}^{t} \cong (T\mathcal{F}^{t})^{\perp}$ on $M$ such that $g^{t}_{2}$ is transverse with respect to both of $\mathcal{F}^{t}$ and $\widetilde{\mathcal{F}}^{t}$. We can extend the family of metrics $\{g^{t}_{2}\}_{t \in U}$ on $\{(T\mathcal{F}^{t})^{\perp}\}_{t \in U}$ to a family of Riemannian metrics $\{g^{t}\}_{t \in U}$ on $M$ so that the characteristic form of $(M,\widetilde{\mathcal{F}}^{t},g^{t})$ is equal to $\widetilde{\chi}^{t}$. Then $g^{t}$ is bundle-like with respect to both of $\mathcal{F}^{t}$ and $\widetilde{\mathcal{F}}^{t}$. By the Rummler's formula (see the second formula in the proof of Proposition 1 in Rummler \cite{Rummler} or Lemma 10.5.6 of Candel and Conlon \cite{Candel Conlon}),  the mean curvature form of a Riemannian manifold with a foliation is determined only by the characteristic form and the orthogonal complement of the tangent bundle of the foliation. Since the characteristic form of $(M \times \{t\},\widetilde{\mathcal{F}}^{t},g^{t})$ is equal to $\chi^{t}$, the leaves of $\widetilde{\mathcal{F}}^{t}$ are minimal submanifolds of $(M \times \{t\},g^{t})$.
\end{proof}

We will confirm that the conditions (c) and (d) of Proposition~\ref{Proposition : Fiberwise average} are satisfied in the present situation. 

\begin{lemma}\label{Lemma : abcd}
The conditions (c) and (d) of Proposition~\ref{Proposition : Fiberwise average} are satisfied by $(M,\mathcal{F}^{0},g^{0})$ and $\pi_{\widetilde{\mathcal{F}}}^{0}$.
\end{lemma}

\begin{proof}
Let $\{\omega_{i}\}_{i = 1}^{q}$ on $(M \times U,\mathcal{F}^{\amb})$ be the set of basic $1$-forms on $(M \times U,\mathcal{F}^{\amb})$ such that $\omega_{i}(X^{j}_{\amb})=\delta_{ij}$, where $\delta_{ij}$ is the Kronecker's delta. $d\omega_{i}$ is written as 
\begin{equation}
d\omega_{i} = \sum_{1 \leq j < k \leq q} c_{i}^{jk} \omega_{j} \wedge \omega_{k}
\end{equation}
for some functions $c_{i}^{jk}$ on $M \times U$. We have
\begin{equation}
d\omega_{i}(X^{j}_{\amb},X^{k}_{\amb}) = c_{i}^{jk}.
\end{equation}
Since $d\omega_{i}$ and each $\omega_{j}$ are basic forms and $X^{j}_{\amb}$ is a transverse field on $(M \times U,\mathcal{F}^{\amb})$, $c_{i}^{jk}$ is a basic function on $(M \times U,\mathcal{F}^{\amb})$. Hence the restriction of $c_{i}^{jk}$ to each fiber of $\widetilde{\mathcal{F}}^{0}$ is a constant by definition of $\mathcal{D}$. This proves that the condition (c) is satisfied. 

Let $\beta$ be a $1$-form on $(M \times \{0\}) /\widetilde{\mathcal{F}}^{0}$. Then $\pi_{\widetilde{\mathcal{F}}}^{*}\beta$ is a basic $1$-forms on $(M \times U,\mathcal{F}^{\amb})$. It follows that $\pi_{\widetilde{\mathcal{F}}}^{*}\beta(X^{i}_{\amb})$ is a global basic function on $(M \times U,\mathcal{F}^{\amb})$. Hence the restriction of $\pi_{\widetilde{\mathcal{F}}}^{*}\beta(X^{i}_{\amb})$ to each fiber of $\widetilde{\mathcal{F}}^{0}$ is a constant by definition of $\mathcal{D}$. Hence the image of $X^{i}_{\amb}|_{M \times \{0\}}$ by the canonical projection $C^{\infty}(T(M \times \{0\})/T\mathcal{F}^{0}) \longrightarrow C^{\infty}(T(M \times \{0\})/T\widetilde{\mathcal{F}}^{0})$ is a transverse field on $(M \times \{0\},\widetilde{\mathcal{F}}^{0})$. This proves that the condition (d) is satisfied.
\end{proof}

\subsection{Proof of the continuity Theorem~\ref{Theorem : Continuity}}\label{Section : Proof of continuity theorem}

Let $M$ be a closed manifold and $\{\mathcal{F}^{t}\}_{t \in U}$ be a smooth family of transversely parallelizable foliations of $M$ over $U$. We consider the distribution $D$ defined by the equation \eqref{Equation : D}. By Lemma~\ref{Lemma : Molino}, a proper submersion $\pi_{\widetilde{\mathcal{F}}}^{0} \colon M \times \{0\} \longrightarrow (M \times \{0\})/ \mathcal{D}^{0}$ is defined by the restriction of $D$ to $M \times \{0\}$. By Lemma~\ref{Lemma : proper submersion}, we can take an open neighborhood $U'$ of $0$ in $U$ and a proper submersion $\pi_{\widetilde{\mathcal{F}}}^{\amb} \colon M \times U' \longrightarrow (M \times \{0\})/ \mathcal{D}^{0}$ such that $\pi_{\widetilde{\mathcal{F}}}|_{M \times \{0\}}=\pi_{\widetilde{\mathcal{F}}}^{0}$ and each fiber of $\pi_{\widetilde{\mathcal{F}}}|_{M \times \{t\}}$ are saturated by the leaves of $\mathcal{F}^{t}$ for each $t$ in $U'$. We denote the foliation of $M \times \{t\}$ defined by the fibers of $\pi_{\widetilde{\mathcal{F}}}|_{M \times \{t\}}$ by $\widetilde{\mathcal{F}}^{t}$. By Lemma~\ref{Lemma : tilde g}, we take a smooth family of metrics $\{g^{t}\}_{t \in U'}$ on $M$ such that the fibers of $\pi_{\widetilde{\mathcal{F}}}^{t}$ are minimal submanifolds of $(M \times \{t\}, g^{t})$, and $g^{t}$ is bundle-like with respect to both of $(M,\mathcal{F}^{t})$ and $(M,\widetilde{\mathcal{F}}^{t})$ for each $t$ in $U'$. We denote the mean curvature form and the \'{A}lvarez form of $(M,\mathcal{F}^{t},g^{t})$ by $\kappa^{t}$ and $\kappa_{b}^{t}$, respectively. We define $\widetilde{\kappa}_{b}^{t}=\rho_{\widetilde{\mathcal{F}}}(\kappa^{t})$.

Let us show Theorem~\ref{Theorem : Continuity} by using Proposition~\ref{Proposition : Fiberwise average} and Corollary 4.23 of Dom\'{i}nguez in \cite{Dominguez}.

\begin{proof}[Proof of Theorem~\ref{Theorem : Continuity}]
By Lemma~\ref{Lemma : Reduction to TP case 2}, it suffices to show the case where $M$, $\widetilde{\mathcal{F}}^{0}$ and the basic fibration of $(M,\mathcal{F}^{0})$ are orientable. By Lemmas~\ref{Lemma : tilde g} and \ref{Lemma : abcd}, the conditions (a), (b), (c) and (d) of Proposition~\ref{Proposition : Fiberwise average} are satisfied. Hence, by Proposition~\ref{Proposition : Fiberwise average}, $\widetilde{\kappa}_{b}^{0}$ is closed and $[\widetilde{\kappa}_{b}^{0}]=[\kappa_{b}^{0}]$. By Corollary 4.23 of Dom\'{i}nguez \cite{Dominguez}, we can modify the component $g^{t}|_{T\mathcal{F}^{t} \otimes T\mathcal{F}^{t}}$ along the leaves of $\{g^{t}\}_{t \in T'}$ so that $\kappa^{0}=\widetilde{\kappa}_{b}^{0}$. Note that $\widetilde{\kappa}_{b}^{t}$ may not be closed for nonzero parameter $t$ in $T'$.

For a smooth loop $\gamma$ in $M$, we have the following evaluation:
\begin{align*}
& \,\, \Big| \int_{\gamma} \Big( \widetilde{\kappa}_{b}^{0} - \kappa_{b}^{t} \Big) \Big| \\
\leq & \,\, \Big| \int_{\gamma} \Big( \widetilde{\kappa}_{b}^{0} - \widetilde{\kappa}_{b}^{t} \Big) \Big| + \Big| \int_{\gamma} \Big( \widetilde{\kappa}_{b}^{t} - \kappa_{b}^{t} \Big) \Big| \\
= & \,\, \Big| \int_{\gamma} \Big( \widetilde{\kappa}_{b}^{0} - \widetilde{\kappa}_{b}^{t} \Big) \Big| + \Big| \int_{\gamma} \rho_{\mathcal{F}} \Big( \widetilde{\kappa}_{b}^{t} - \kappa^{t} \Big) \Big| \\
\leq & \,\, \Big| \int_{\gamma} \Big( \widetilde{\kappa}_{b}^{0} - \widetilde{\kappa}_{b}^{t} \Big) \Big| + \Big( \sup_{s \in S^1} \Big|\Big| \frac{d\gamma}{ds}(s) \Big|\Big| \Big) \Big( \sup_{x \in M \times \{t\}} \Big|\Big| \widetilde{\kappa}_{b}^{t}(x) - \kappa^{t}(x) \Big|\Big| \Big),
\end{align*}
where $|| \cdot ||$ is a norm induced by $g^{t}$. Since $\widetilde{\kappa}_{b}^{t}$ converges to $\widetilde{\kappa}_{b}^{0}=\kappa^{0}$, the first and the second term converges to $0$ as $t$ tends to $0$. Then we have $\lim_{t \rightarrow 0} \int_{\gamma} \kappa_{b}^{t} = \int_{\gamma} \widetilde{\kappa}_{b}^{0}$ and the proof is completed.
\end{proof}

By Proposition 5.3 of \'{A}lvarez L\'{o}pez \cite{Alvarez Lopez}, every closed $1$-form cohomologous to the \'{A}lvarez class of $(M,\mathcal{F}^{0})$ is realized as the \'{A}lvarez form of $(M,\mathcal{F}^{0},g)$ for some bundle-like metric $g$. Proposition 5.3 of \'{A}lvarez L\'{o}pez is simpler to prove than the Corollary 4.23 of Dom\'{i}nguez \cite{Dominguez}, which is used in the proof of Theorem~\ref{Theorem : Continuity} above. But we do not know if we can replace Corollary 4.23 of \cite{Dominguez} by Proposition 5.3 of \cite{Alvarez Lopez} in the proof of Theorem~\ref{Theorem : Continuity}. 

In fact, by Proposition 5.3 of \cite{Alvarez Lopez}, we can modify the component $g^{t}|_{T\mathcal{F}^{t} \otimes T\mathcal{F}^{t}}$ along the leaves of $\{g^{t}\}_{t \in T'}$ so that $\kappa^{0}_{b}=\widetilde{\kappa}_{b}^{0}$. But we do not know if $\Big| \int_{\gamma} \Big( \widetilde{\kappa}_{b}^{t} - \kappa_{b}^{t} \Big) \Big|$ converges to $0$ as $t$ goes to $0$ here. Note that $\Big| \int_{\gamma} \Big( \widetilde{\kappa}_{b}^{t} - \kappa^{t} \Big) \Big|$ may not converge to~$0$ as $t$ goes to $0$ in this situation. This is because $\kappa_{b}^{t}$ is defined by integrating the mean curvature form on each leaf closure of $\mathcal{F}^{t}$ and the dimension of the closures of leaves of $\mathcal{F}^{t}$ can change on any small open neighborhood of $0$.

\section{Examples of Riemannian foliations}

\addtocontents{toc}{\protect\setcounter{tocdepth}{1}}
\subsection{A special case where the families of Molino's commuting sheaves are smooth}\label{Example : smooth family of Molino sheaves}
Let $\{\mathcal{F}^{t}\}_{t \in T}$ be a family of Riemannian foliations on a closed manifold $M$. If the dimension of the closures of generic leaves of $\mathcal{F}^{t}$ is constant with respect to $t$, then the family of Molino's commuting sheaves of $\{\mathcal{F}^{t}\}$ is smooth (see pages~125--130 and Section 5.3 of Molino \cite{Molino} for the definition of the Molino's commuting sheaf of a Riemannian foliation). Since the \'{A}lvarez class of $(M,\mathcal{F}^{t})$ is computed from the holonomy homomorphism of the Molino's commuting sheaf of $(M,\mathcal{F}^{t})$ by Theorem 1.1 of \'{A}lvarez L\'{o}pez \cite{Alvarez Lopez 2}, the \'{A}lvarez classes of this family are continuous with respect to $t$. Our main continuity Theorem~\ref{Theorem : Continuity} is essential in the case where the dimension of the closures of leaves changes. If the dimension of the closures of leaves changes, we cannot prove the continuity of the \'{A}lvarez class as above or directly by an application of deformation theory to Molino's commuting sheaves. In fact, the family of the Molino's commuting sheaves must be discontinuous in this case, because the rank of the Molino's commuting sheaf is equal to the dimension of the closures of generic leaves of $\mathcal{F}^{t}$.

\subsection{Families of homogeneous Lie foliations}
Let $p \colon L \longrightarrow G$ be a surjective homomorphism between Lie groups. Let $\Gamma$ be a uniform lattice of $L$. A foliation $\mathcal{F}$ on a homogeneous space $\Gamma \backslash L$ is induced by the fibers of $p$. This $\mathcal{F}$ has a structure of a $G$-Lie foliation. Such $\mathcal{F}$ is called a homogeneous $G$-Lie foliation. By deforming $L$, $G$, $p$ and $\Gamma$, we may produce families of homogeneous Lie foliations. The \'{A}lvarez class is computed in terms of Lie theory by the interpretation of the \'{A}lvarez class as a first secondary characteristic class of Molino's commuting sheaf by \'{A}lvarez L\'{o}pez (Theorem 1.1 of \cite{Alvarez Lopez 2}). But the author does not know an example of a family of Riemannian foliations whose \'{A}lvarez classes change nontrivially obtained in this way. In many cases, the \'{A}lvarez class does not change as we will see in the following. If $G$ is nilpotent, then $\mathcal{F}$ is of polynomial growth. Then the \'{A}lvarez class does not change under deformations of $\mathcal{F}$ by Corollary~\ref{Corollary : Invariance 1}. If $L$ is solvable, then $\Gamma$ is polycyclic (see Proposition 3.7 of Raghunathan \cite{Raghunathan}). Then the \'{A}lvarez class does not change under deformations of $\mathcal{F}$ by Corollary~\ref{Corollary : Invariance 1}. If $G$ is semisimple, the structural Lie algebra of the Lie foliation defined on the closure of leaves of $\mathcal{F}$ is semisimple. Then $\mathcal{F}$ is minimizable by Theorem 2 of Nozawa \cite{Nozawa}.

\subsection{Meigniez's examples: Families of solvable Lie foliations}
Meigniez constructed plenty of families of solvable Lie foliations which are not homogeneous by a surgery construction on homogeneous Lie foliations in \cite{Meigniez} (see also \cite{Meigniez 2}, in particular, pages 119--122 for an explicit example). These families contain many examples of families of Lie foliations whose \'{A}lvarez classes change nontrivially.

\subsection{Basic cohomology of Riemannian foliations is not invariant under deformations}\label{Example : Basic}

We present an example of a family of Riemannian foliations whose basic cohomology changes. Let $M = S^1 \times S^3$. Let $\sigma$ be the free $S^1$-action on $S^3$ whose orbits are fibers of the Hopf fibration. Let $\rho$ be the $T^2$-action on $M$ which is the product of the principal $S^1$-action on the first $S^1$-component and $\sigma$. For each element $v$ of $\Lie(T^2) - \{0\}$, let $\mathcal{F}_{v}$ be the Riemannian flow on $M$ whose leaves are the orbits of an $\mathbb{R}$-subaction of $\rho$ whose infinitesimal action is given by $v$. Then we have a smooth family $\{\mathcal{F}_{v}\}_{v \in \Lie(T^2) - \{0\}}$ of Riemannian flows on $M$. Let $v_{1}$ and $v_{2}$ be the infinitesimal generators of the principal $S^1$-action on the first $S^1$-component and $\sigma$, respectively. Since $M / \mathcal{F}_{v_{1}} = S^3$ and $M / \mathcal{F}_{v_{2}} = M / \sigma = S^1 \times S^2$, clearly the dimension of $H^{1}_{b}(M / \mathcal{F}_{v_{1}})$ and $H^{1}_{b}(M / \mathcal{F}_{v_{2}})$ are different.

\section{Examples of non-Riemannian foliations}\label{Section : non Riemannian}

\subsection{Turbulization}

Let $\mathcal{F}^{1}$ be a product foliation $T^{2} = \bigsqcup_{\theta_{1} \in S^{1}} \{\theta_{1}\} \times S^1$ on $T^{2}$. Let $\mathcal{F}^{0}$ be a turbulization of $\mathcal{F}^{1}$ along a closed curve $\{\frac{1}{2}\} \times S^{1}$. This $\mathcal{F}^{0}$ is not minimizable by a theorem of Sullivan (see \cite{Sullivan 2}), because $\mathcal{F}^{0}$ has a tangent homology defined by the Reeb component. Note that $\mathcal{F}^{0}$ is a limit of $1$-dimensional foliations on $T^{2}$ which are diffeomorphic to $\mathcal{F}^{1}$. Thus we have a family of $1$-dimensional foliations on $T^{2}$ parametrized by $[0,1]$ such that only $\mathcal{F}^{0}$ is not minimizable.

\vspace{10pt}

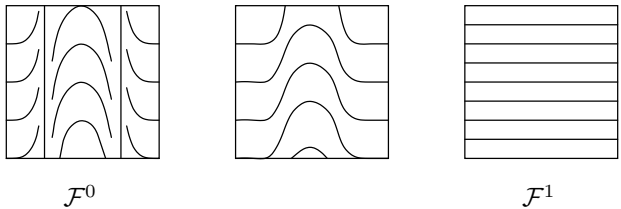
\begin{center}
%WinTpicVersion3.08
\unitlength 0.1in
\begin{picture}( 32.0000,  8.8000)( 16.0000,-14.8000)
% BOX 2 0 3 0
% 2 4000 600 4800 1400
% 
\special{pn 8}%
\special{pa 4000 600}%
\special{pa 4800 600}%
\special{pa 4800 1400}%
\special{pa 4000 1400}%
\special{pa 4000 600}%
\special{fp}%
% BOX 2 0 3 0
% 2 2800 600 3600 1400
% 
\special{pn 8}%
\special{pa 2800 600}%
\special{pa 3600 600}%
\special{pa 3600 1400}%
\special{pa 2800 1400}%
\special{pa 2800 600}%
\special{fp}%
% BOX 2 0 3 0
% 2 1600 600 2400 1400
% 
\special{pn 8}%
\special{pa 1600 600}%
\special{pa 2400 600}%
\special{pa 2400 1400}%
\special{pa 1600 1400}%
\special{pa 1600 600}%
\special{fp}%
% LINE 2 0 3 0
% 2 4000 1200 4800 1200
% 
\special{pn 8}%
\special{pa 4000 1200}%
\special{pa 4800 1200}%
\special{fp}%
% LINE 2 0 3 0
% 2 4000 1100 4800 1100
% 
\special{pn 8}%
\special{pa 4000 1100}%
\special{pa 4800 1100}%
\special{fp}%
% LINE 2 0 3 0
% 2 4000 1300 4800 1300
% 
\special{pn 8}%
\special{pa 4000 1300}%
\special{pa 4800 1300}%
\special{fp}%
% LINE 2 0 3 0
% 2 4000 700 4800 700
% 
\special{pn 8}%
\special{pa 4000 700}%
\special{pa 4800 700}%
\special{fp}%
% LINE 2 0 3 0
% 2 4000 900 4800 900
% 
\special{pn 8}%
\special{pa 4000 900}%
\special{pa 4800 900}%
\special{fp}%
% LINE 2 0 3 0
% 2 4000 1000 4800 1000
% 
\special{pn 8}%
\special{pa 4000 1000}%
\special{pa 4800 1000}%
\special{fp}%
% LINE 2 0 3 0
% 2 4000 800 4800 800
% 
\special{pn 8}%
\special{pa 4000 800}%
\special{pa 4800 800}%
\special{fp}%
% LINE 2 0 3 0
% 2 1800 600 1800 1400
% 
\special{pn 8}%
\special{pa 1800 600}%
\special{pa 1800 1400}%
\special{fp}%
% LINE 2 0 3 0
% 2 2200 600 2200 1400
% 
\special{pn 8}%
\special{pa 2200 600}%
\special{pa 2200 1400}%
\special{fp}%
% SPLINE 2 0 3 0
% 6 1600 1200 1640 1200 1690 1190 1730 1150 1770 1030 1770 1030
% 
\special{pn 8}%
\special{pa 1600 1200}%
\special{pa 1632 1200}%
\special{pa 1664 1198}%
\special{pa 1694 1188}%
\special{pa 1718 1166}%
\special{pa 1736 1140}%
\special{pa 1750 1110}%
\special{pa 1760 1080}%
\special{pa 1766 1048}%
\special{pa 1770 1030}%
\special{sp}%
% SPLINE 2 0 3 0
% 6 2400 1000 2360 1000 2310 990 2270 950 2230 830 2230 830
% 
\special{pn 8}%
\special{pa 2400 1000}%
\special{pa 2368 1000}%
\special{pa 2336 998}%
\special{pa 2306 988}%
\special{pa 2282 966}%
\special{pa 2264 940}%
\special{pa 2252 910}%
\special{pa 2242 880}%
\special{pa 2234 848}%
\special{pa 2230 830}%
\special{sp}%
% SPLINE 2 0 3 0
% 6 1600 1000 1640 1000 1690 990 1730 950 1770 830 1770 830
% 
\special{pn 8}%
\special{pa 1600 1000}%
\special{pa 1632 1000}%
\special{pa 1664 998}%
\special{pa 1694 988}%
\special{pa 1718 966}%
\special{pa 1736 940}%
\special{pa 1750 910}%
\special{pa 1760 880}%
\special{pa 1766 848}%
\special{pa 1770 830}%
\special{sp}%
% SPLINE 2 0 3 0
% 6 1600 1400 1640 1400 1690 1390 1730 1350 1770 1230 1770 1230
% 
\special{pn 8}%
\special{pa 1600 1400}%
\special{pa 1632 1400}%
\special{pa 1664 1398}%
\special{pa 1694 1388}%
\special{pa 1718 1366}%
\special{pa 1736 1340}%
\special{pa 1750 1310}%
\special{pa 1760 1280}%
\special{pa 1766 1248}%
\special{pa 1770 1230}%
\special{sp}%
% SPLINE 2 0 3 0
% 6 1600 800 1640 800 1690 790 1730 750 1770 630 1770 630
% 
\special{pn 8}%
\special{pa 1600 800}%
\special{pa 1632 800}%
\special{pa 1664 798}%
\special{pa 1694 788}%
\special{pa 1718 766}%
\special{pa 1736 740}%
\special{pa 1750 710}%
\special{pa 1760 680}%
\special{pa 1766 648}%
\special{pa 1770 630}%
\special{sp}%
% SPLINE 2 0 3 0
% 6 2400 800 2360 800 2310 790 2270 750 2230 630 2230 630
% 
\special{pn 8}%
\special{pa 2400 800}%
\special{pa 2368 800}%
\special{pa 2336 798}%
\special{pa 2306 788}%
\special{pa 2282 766}%
\special{pa 2264 740}%
\special{pa 2252 710}%
\special{pa 2242 680}%
\special{pa 2234 648}%
\special{pa 2230 630}%
\special{sp}%
% SPLINE 2 0 3 0
% 6 2400 1400 2360 1400 2310 1390 2270 1350 2230 1230 2230 1230
% 
\special{pn 8}%
\special{pa 2400 1400}%
\special{pa 2368 1400}%
\special{pa 2336 1398}%
\special{pa 2306 1388}%
\special{pa 2282 1366}%
\special{pa 2264 1340}%
\special{pa 2252 1310}%
\special{pa 2242 1280}%
\special{pa 2234 1248}%
\special{pa 2230 1230}%
\special{sp}%
% SPLINE 2 0 3 0
% 6 2400 1200 2360 1200 2310 1190 2270 1150 2230 1030 2230 1030
% 
\special{pn 8}%
\special{pa 2400 1200}%
\special{pa 2368 1200}%
\special{pa 2336 1198}%
\special{pa 2306 1188}%
\special{pa 2282 1166}%
\special{pa 2264 1140}%
\special{pa 2252 1110}%
\special{pa 2242 1080}%
\special{pa 2234 1048}%
\special{pa 2230 1030}%
\special{sp}%
% SPLINE 2 0 3 0
% 8 1840 890 1870 740 1900 670 1990 600 2080 660 2110 730 2150 890 2150 890
% 
\special{pn 8}%
\special{pa 1840 890}%
\special{pa 1846 858}%
\special{pa 1850 826}%
\special{pa 1856 796}%
\special{pa 1864 764}%
\special{pa 1872 734}%
\special{pa 1884 704}%
\special{pa 1898 676}%
\special{pa 1918 646}%
\special{pa 1942 622}%
\special{pa 1970 604}%
\special{pa 1998 600}%
\special{pa 2028 610}%
\special{pa 2056 632}%
\special{pa 2078 658}%
\special{pa 2094 686}%
\special{pa 2106 716}%
\special{pa 2116 746}%
\special{pa 2124 776}%
\special{pa 2132 808}%
\special{pa 2138 840}%
\special{pa 2146 870}%
\special{pa 2150 890}%
\special{sp}%
% SPLINE 2 0 3 0
% 8 1840 1290 1870 1140 1900 1070 1990 1000 2080 1060 2110 1130 2150 1290 2150 1290
% 
\special{pn 8}%
\special{pa 1840 1290}%
\special{pa 1846 1258}%
\special{pa 1850 1226}%
\special{pa 1856 1196}%
\special{pa 1864 1164}%
\special{pa 1872 1134}%
\special{pa 1884 1104}%
\special{pa 1898 1076}%
\special{pa 1918 1046}%
\special{pa 1942 1022}%
\special{pa 1970 1004}%
\special{pa 1998 1000}%
\special{pa 2028 1010}%
\special{pa 2056 1032}%
\special{pa 2078 1058}%
\special{pa 2094 1086}%
\special{pa 2106 1116}%
\special{pa 2116 1146}%
\special{pa 2124 1176}%
\special{pa 2132 1208}%
\special{pa 2138 1240}%
\special{pa 2146 1270}%
\special{pa 2150 1290}%
\special{sp}%
% SPLINE 2 0 3 0
% 8 1840 1090 1870 940 1900 870 1990 800 2080 860 2110 930 2150 1090 2150 1090
% 
\special{pn 8}%
\special{pa 1840 1090}%
\special{pa 1846 1058}%
\special{pa 1850 1026}%
\special{pa 1856 996}%
\special{pa 1864 964}%
\special{pa 1872 934}%
\special{pa 1884 904}%
\special{pa 1898 876}%
\special{pa 1918 846}%
\special{pa 1942 822}%
\special{pa 1970 804}%
\special{pa 1998 800}%
\special{pa 2028 810}%
\special{pa 2056 832}%
\special{pa 2078 858}%
\special{pa 2094 886}%
\special{pa 2106 916}%
\special{pa 2116 946}%
\special{pa 2124 976}%
\special{pa 2132 1008}%
\special{pa 2138 1040}%
\special{pa 2146 1070}%
\special{pa 2150 1090}%
\special{sp}%
% SPLINE 2 0 3 0
% 8 1880 1400 1900 1300 1930 1240 2000 1200 2060 1240 2090 1310 2120 1400 2120 1400
% 
\special{pn 8}%
\special{pa 1880 1400}%
\special{pa 1886 1368}%
\special{pa 1892 1338}%
\special{pa 1898 1306}%
\special{pa 1910 1276}%
\special{pa 1924 1248}%
\special{pa 1948 1222}%
\special{pa 1976 1204}%
\special{pa 2006 1202}%
\special{pa 2036 1214}%
\special{pa 2060 1238}%
\special{pa 2076 1268}%
\special{pa 2086 1298}%
\special{pa 2096 1328}%
\special{pa 2106 1358}%
\special{pa 2116 1388}%
\special{pa 2120 1400}%
\special{sp}%
% SPLINE 2 0 3 0
% 13 2800 1000 2910 1000 2960 1000 3010 950 3050 860 3090 770 3180 700 3290 760 3340 880 3420 990 3490 1000 3600 1000 3600 1000
% 
\special{pn 8}%
\special{pa 2800 1000}%
\special{pa 2832 998}%
\special{pa 2864 998}%
\special{pa 2896 1000}%
\special{pa 2930 1002}%
\special{pa 2960 1000}%
\special{pa 2986 984}%
\special{pa 3006 956}%
\special{pa 3022 928}%
\special{pa 3036 898}%
\special{pa 3048 868}%
\special{pa 3058 840}%
\special{pa 3070 810}%
\special{pa 3084 780}%
\special{pa 3102 752}%
\special{pa 3126 728}%
\special{pa 3154 708}%
\special{pa 3182 700}%
\special{pa 3214 704}%
\special{pa 3244 718}%
\special{pa 3272 740}%
\special{pa 3294 764}%
\special{pa 3310 792}%
\special{pa 3322 822}%
\special{pa 3332 852}%
\special{pa 3342 884}%
\special{pa 3354 916}%
\special{pa 3370 946}%
\special{pa 3392 970}%
\special{pa 3416 988}%
\special{pa 3446 998}%
\special{pa 3478 1000}%
\special{pa 3512 1000}%
\special{pa 3544 1000}%
\special{pa 3576 1000}%
\special{pa 3600 1000}%
\special{sp}%
% SPLINE 2 0 3 0
% 13 2800 1200 2910 1200 2960 1200 3010 1150 3050 1060 3090 970 3180 900 3290 960 3340 1080 3420 1190 3490 1200 3600 1200 3600 1200
% 
\special{pn 8}%
\special{pa 2800 1200}%
\special{pa 2832 1198}%
\special{pa 2864 1198}%
\special{pa 2896 1200}%
\special{pa 2930 1202}%
\special{pa 2960 1200}%
\special{pa 2986 1184}%
\special{pa 3006 1156}%
\special{pa 3022 1128}%
\special{pa 3036 1098}%
\special{pa 3048 1068}%
\special{pa 3058 1040}%
\special{pa 3070 1010}%
\special{pa 3084 980}%
\special{pa 3102 952}%
\special{pa 3126 928}%
\special{pa 3154 908}%
\special{pa 3182 900}%
\special{pa 3214 904}%
\special{pa 3244 918}%
\special{pa 3272 940}%
\special{pa 3294 964}%
\special{pa 3310 992}%
\special{pa 3322 1022}%
\special{pa 3332 1052}%
\special{pa 3342 1084}%
\special{pa 3354 1116}%
\special{pa 3370 1146}%
\special{pa 3392 1170}%
\special{pa 3416 1188}%
\special{pa 3446 1198}%
\special{pa 3478 1200}%
\special{pa 3512 1200}%
\special{pa 3544 1200}%
\special{pa 3576 1200}%
\special{pa 3600 1200}%
\special{sp}%
% SPLINE 2 0 3 0
% 13 2800 1400 2910 1400 2960 1400 3010 1350 3050 1260 3090 1170 3180 1100 3290 1160 3340 1280 3420 1390 3490 1400 3600 1400 3600 1400
% 
\special{pn 8}%
\special{pa 2800 1400}%
\special{pa 2832 1398}%
\special{pa 2864 1398}%
\special{pa 2896 1400}%
\special{pa 2930 1402}%
\special{pa 2960 1400}%
\special{pa 2986 1384}%
\special{pa 3006 1356}%
\special{pa 3022 1328}%
\special{pa 3036 1298}%
\special{pa 3048 1268}%
\special{pa 3058 1240}%
\special{pa 3070 1210}%
\special{pa 3084 1180}%
\special{pa 3102 1152}%
\special{pa 3126 1128}%
\special{pa 3154 1108}%
\special{pa 3182 1100}%
\special{pa 3214 1104}%
\special{pa 3244 1118}%
\special{pa 3272 1140}%
\special{pa 3294 1164}%
\special{pa 3310 1192}%
\special{pa 3322 1222}%
\special{pa 3332 1252}%
\special{pa 3342 1284}%
\special{pa 3354 1316}%
\special{pa 3370 1346}%
\special{pa 3392 1370}%
\special{pa 3416 1388}%
\special{pa 3446 1398}%
\special{pa 3478 1400}%
\special{pa 3512 1400}%
\special{pa 3544 1400}%
\special{pa 3576 1400}%
\special{pa 3600 1400}%
\special{sp}%
% SPLINE 2 0 3 0
% 1 2780 1000
% 
\special{pn 8}%
\special{pa 3600 1400}%
\special{pa 3600 1400}%
\special{pa 3600 1400}%
\special{pa 3600 1400}%
\special{pa 3600 1400}%
\special{pa 3600 1400}%
\special{pa 3600 1400}%
\special{pa 3600 1400}%
\special{pa 3600 1400}%
\special{pa 3600 1400}%
\special{pa 3600 1400}%
\special{pa 3600 1400}%
\special{pa 3600 1400}%
\special{pa 3600 1400}%
\special{pa 3600 1400}%
\special{pa 3600 1400}%
\special{pa 3600 1400}%
\special{pa 3600 1400}%
\special{pa 3600 1400}%
\special{pa 3600 1400}%
\special{pa 3600 1400}%
\special{pa 3600 1400}%
\special{pa 3600 1400}%
\special{pa 3600 1400}%
\special{pa 3600 1400}%
\special{pa 3600 1400}%
\special{pa 3600 1400}%
\special{pa 3600 1400}%
\special{pa 3600 1400}%
\special{pa 3600 1400}%
\special{pa 3600 1400}%
\special{pa 3600 1400}%
\special{pa 3600 1400}%
\special{pa 3600 1400}%
\special{pa 3600 1400}%
\special{pa 3600 1400}%
\special{sp}%
% STR 2 0 3 0
% 3 1900 1550 1900 1650 2 0
% $\mathcal{F}^{0}$
\put(19.0000,-16.5000){\makebox(0,0)[lb]{$\mathcal{F}^{0}$}}%
% STR 2 0 3 0
% 3 4300 1550 4300 1650 2 0
% $\mathcal{F}^{1}$
\put(43.0000,-16.5000){\makebox(0,0)[lb]{$\mathcal{F}^{1}$}}%
% SPLINE 2 0 3 0
% 4 2800 800 2910 800 2960 800 3060 600
% 
\special{pn 8}%
\special{pa 2800 800}%
\special{pa 2832 800}%
\special{pa 2864 800}%
\special{pa 2896 800}%
\special{pa 2930 802}%
\special{pa 2960 800}%
\special{pa 2986 790}%
\special{pa 3008 772}%
\special{pa 3024 746}%
\special{pa 3036 714}%
\special{pa 3046 680}%
\special{pa 3054 640}%
\special{pa 3060 600}%
\special{sp}%
% SPLINE 2 0 3 0
% 4 3600 800 3490 800 3440 800 3340 600
% 
\special{pn 8}%
\special{pa 3600 800}%
\special{pa 3570 800}%
\special{pa 3538 800}%
\special{pa 3504 800}%
\special{pa 3472 802}%
\special{pa 3440 800}%
\special{pa 3414 790}%
\special{pa 3394 772}%
\special{pa 3378 746}%
\special{pa 3366 714}%
\special{pa 3356 680}%
\special{pa 3348 640}%
\special{pa 3340 600}%
\special{sp}%
% SPLINE 2 0 3 0
% 4 3280 1400 3200 1340 3090 1400 3090 1400
% 
\special{pn 8}%
\special{pa 3280 1400}%
\special{pa 3256 1374}%
\special{pa 3230 1354}%
\special{pa 3204 1342}%
\special{pa 3176 1342}%
\special{pa 3148 1354}%
\special{pa 3120 1376}%
\special{pa 3092 1400}%
\special{pa 3090 1400}%
\special{sp}%
\end{picture}%

\end{center}

\vspace{5pt}

\subsection{Deformation of an example of Candel and Conlon}\label{Example : Candel Conlon}

We present an example of a family $\{\mathcal{F}^{t}\}_{t \in [0,1]}$ of $1$-dimensional foliations on $S^3$ such that $\mathcal{F}^{0}$ is not minimizable and $\mathcal{F}^{t}$ is minimizable if $t$ is nonzero. In a similar way, we will construct a family $\{\mathcal{H}^{t}\}_{t \in [0,1]}$ of $1$-dimensional foliations on $S^3$ such that $\mathcal{H}^{1}$ is minimizable and $\mathcal{H}^{t}$ is minimizable if $t$ is not equal to $1$. Here $\mathcal{F}^{0}$ and $\mathcal{H}^{0}$ are the example constructed by Candel and Conlon in Example 10.5.19 of \cite{Candel Conlon}. These examples show that the minimizability is neither open nor closed in families of foliations in general.

We restate the construction of the example $\mathcal{F}^{0}$ of Candel and Conlon here. We consider the $2$-dimensional product foliation $S^{1} \times D^{2} = \bigsqcup_{t \in S^{1}} \{t\} \times D^{2}$ on the solid torus. Turbulizing this product foliation around the axis $S^{1} \times \{0\}$, we obtain a singular foliation $\mathcal{S}$ on $S^{1} \times D^{2}$ whose leaves are trumpet-like surfaces and the axis $S^{1} \times \{0\}$. We foliate $S^{1} \times D^{2}$ by a $1$-dimensional foliation $\mathcal{G}^{0}$ so that each leaf of $\mathcal{S}$ is saturated by leaves of $\mathcal{G}^{0}$ and the leaves of $\mathcal{G}^{0}$ are transverse to the boundary of the solid torus. We obtain a foliation $\mathcal{F}^{0}$ on $S^{3}$ by pasting two copies of $(S^{1} \times D^{2},\mathcal{G}^{0})$. This $(S^3,\mathcal{F}^{0})$ is non-minimizable as Candel and Conlon showed by a theorem of Sullivan in Example 10.5.19 of \cite{Candel Conlon}. 

We construct $\mathcal{F}^{t}$ for nonzero $t$ in $[0,1]$. Let $L_{1}$ and $L_{2}$ be two closed leaves of $\mathcal{F}^{0}$ which are axes of solid tori. For $t$ in $[0,1]$, let $\mathcal{F}^{t}$ be the smooth foliation obtained from $\mathcal{F}^{0}$ by replacing both of $L_{1}$ and $L_{2}$ to solid tori $K^{t}_{1}$ and $K^{t}_{2}$ of radius $t$ with the product foliation $K^{t}_{i} = S^{1} \times D^{2} = \bigsqcup_{x \in D^{2}} S^{1} \times \{x\}$ for $i=1$ and $2$. Thus we have a smooth family $\{\mathcal{F}^{t}\}_{t \in [0,1]}$ of $1$-dimensional foliations on $S^{3}$.

\vspace{30pt}

\begin{center}
%WinTpicVersion3.08
\unitlength 0.1in
\begin{picture}( 54.4700, 20.3000)(  3.9100,-30.3000)
% ELLIPSE 2 0 3 0
% 4 716 1912 1041 2814 1041 2814 1041 2814
% 
\special{pn 8}%
\special{ar 716 1912 326 902  0.0000000 6.2831853}%
% ELLIPSE 2 0 3 0
% 4 716 1912 1041 2814 1041 2814 1041 2814
% 
\special{pn 8}%
\special{ar 716 1912 326 902  0.0000000 6.2831853}%
% ELLIPSE 2 0 3 0
% 4 716 1912 1041 2814 1041 2814 1041 2814
% 
\special{pn 8}%
\special{ar 716 1912 326 902  0.0000000 6.2831853}%
% ELLIPSE 2 0 3 0
% 4 3596 1912 3921 2814 3921 2814 3921 2814
% 
\special{pn 8}%
\special{ar 3596 1912 326 902  0.0000000 6.2831853}%
% ELLIPSE 2 0 3 0
% 4 3596 1912 3921 2814 3921 2814 3921 2814
% 
\special{pn 8}%
\special{ar 3596 1912 326 902  0.0000000 6.2831853}%
% LINE 2 0 3 0
% 8 718 1010 2478 1010 718 2810 2486 2810 3606 1010 5526 1010 3606 2810 5526 2810
% 
\special{pn 8}%
\special{pa 718 1010}%
\special{pa 2478 1010}%
\special{fp}%
\special{pa 718 2810}%
\special{pa 2486 2810}%
\special{fp}%
\special{pa 3606 1010}%
\special{pa 5526 1010}%
\special{fp}%
\special{pa 3606 2810}%
\special{pa 5526 2810}%
\special{fp}%
% ELLIPSE 2 1 3 0
% 4 5518 1900 5838 2790 5518 1010 5518 2820
% 
\special{pn 8}%
\special{ar 5518 1900 320 890  1.5707963 1.6699699}%
\special{ar 5518 1900 320 890  1.7294740 1.8286476}%
\special{ar 5518 1900 320 890  1.8881517 1.9873253}%
\special{ar 5518 1900 320 890  2.0468294 2.1460029}%
\special{ar 5518 1900 320 890  2.2055071 2.3046806}%
\special{ar 5518 1900 320 890  2.3641848 2.4633583}%
\special{ar 5518 1900 320 890  2.5228624 2.6220360}%
\special{ar 5518 1900 320 890  2.6815401 2.7807137}%
\special{ar 5518 1900 320 890  2.8402178 2.9393914}%
\special{ar 5518 1900 320 890  2.9988955 3.0980691}%
\special{ar 5518 1900 320 890  3.1575732 3.2567467}%
\special{ar 5518 1900 320 890  3.3162509 3.4154244}%
\special{ar 5518 1900 320 890  3.4749286 3.5741021}%
\special{ar 5518 1900 320 890  3.6336062 3.7327798}%
\special{ar 5518 1900 320 890  3.7922839 3.8914575}%
\special{ar 5518 1900 320 890  3.9509616 4.0501352}%
\special{ar 5518 1900 320 890  4.1096393 4.2088129}%
\special{ar 5518 1900 320 890  4.2683170 4.3674905}%
\special{ar 5518 1900 320 890  4.4269947 4.5261682}%
\special{ar 5518 1900 320 890  4.5856724 4.6848459}%
% ELLIPSE 2 1 3 0
% 4 4398 1900 4718 2790 4398 1010 4398 2820
% 
\special{pn 8}%
\special{ar 4398 1900 320 890  1.5707963 1.6699699}%
\special{ar 4398 1900 320 890  1.7294740 1.8286476}%
\special{ar 4398 1900 320 890  1.8881517 1.9873253}%
\special{ar 4398 1900 320 890  2.0468294 2.1460029}%
\special{ar 4398 1900 320 890  2.2055071 2.3046806}%
\special{ar 4398 1900 320 890  2.3641848 2.4633583}%
\special{ar 4398 1900 320 890  2.5228624 2.6220360}%
\special{ar 4398 1900 320 890  2.6815401 2.7807137}%
\special{ar 4398 1900 320 890  2.8402178 2.9393914}%
\special{ar 4398 1900 320 890  2.9988955 3.0980691}%
\special{ar 4398 1900 320 890  3.1575732 3.2567467}%
\special{ar 4398 1900 320 890  3.3162509 3.4154244}%
\special{ar 4398 1900 320 890  3.4749286 3.5741021}%
\special{ar 4398 1900 320 890  3.6336062 3.7327798}%
\special{ar 4398 1900 320 890  3.7922839 3.8914575}%
\special{ar 4398 1900 320 890  3.9509616 4.0501352}%
\special{ar 4398 1900 320 890  4.1096393 4.2088129}%
\special{ar 4398 1900 320 890  4.2683170 4.3674905}%
\special{ar 4398 1900 320 890  4.4269947 4.5261682}%
\special{ar 4398 1900 320 890  4.5856724 4.6848459}%
% ELLIPSE 2 0 3 0
% 4 2478 1900 2798 2790 2478 2830 2478 630
% 
\special{pn 8}%
\special{ar 2478 1900 320 890  4.7123890 6.2831853}%
\special{ar 2478 1900 320 890  0.0000000 1.5707963}%
% ELLIPSE 2 0 3 0
% 4 5518 1900 5838 2790 5518 2830 5518 630
% 
\special{pn 8}%
\special{ar 5518 1900 320 890  4.7123890 6.2831853}%
\special{ar 5518 1900 320 890  0.0000000 1.5707963}%
% ELLIPSE 2 1 3 0
% 4 2478 1900 2798 2790 2478 1010 2478 2820
% 
\special{pn 8}%
\special{ar 2478 1900 320 890  1.5707963 1.6699699}%
\special{ar 2478 1900 320 890  1.7294740 1.8286476}%
\special{ar 2478 1900 320 890  1.8881517 1.9873253}%
\special{ar 2478 1900 320 890  2.0468294 2.1460029}%
\special{ar 2478 1900 320 890  2.2055071 2.3046806}%
\special{ar 2478 1900 320 890  2.3641848 2.4633583}%
\special{ar 2478 1900 320 890  2.5228624 2.6220360}%
\special{ar 2478 1900 320 890  2.6815401 2.7807137}%
\special{ar 2478 1900 320 890  2.8402178 2.9393914}%
\special{ar 2478 1900 320 890  2.9988955 3.0980691}%
\special{ar 2478 1900 320 890  3.1575732 3.2567467}%
\special{ar 2478 1900 320 890  3.3162509 3.4154244}%
\special{ar 2478 1900 320 890  3.4749286 3.5741021}%
\special{ar 2478 1900 320 890  3.6336062 3.7327798}%
\special{ar 2478 1900 320 890  3.7922839 3.8914575}%
\special{ar 2478 1900 320 890  3.9509616 4.0501352}%
\special{ar 2478 1900 320 890  4.1096393 4.2088129}%
\special{ar 2478 1900 320 890  4.2683170 4.3674905}%
\special{ar 2478 1900 320 890  4.4269947 4.5261682}%
\special{ar 2478 1900 320 890  4.5856724 4.6848459}%
% ELLIPSE 2 1 3 0
% 4 1518 1900 1838 2790 1518 1010 1518 2820
% 
\special{pn 8}%
\special{ar 1518 1900 320 890  1.5707963 1.6699699}%
\special{ar 1518 1900 320 890  1.7294740 1.8286476}%
\special{ar 1518 1900 320 890  1.8881517 1.9873253}%
\special{ar 1518 1900 320 890  2.0468294 2.1460029}%
\special{ar 1518 1900 320 890  2.2055071 2.3046806}%
\special{ar 1518 1900 320 890  2.3641848 2.4633583}%
\special{ar 1518 1900 320 890  2.5228624 2.6220360}%
\special{ar 1518 1900 320 890  2.6815401 2.7807137}%
\special{ar 1518 1900 320 890  2.8402178 2.9393914}%
\special{ar 1518 1900 320 890  2.9988955 3.0980691}%
\special{ar 1518 1900 320 890  3.1575732 3.2567467}%
\special{ar 1518 1900 320 890  3.3162509 3.4154244}%
\special{ar 1518 1900 320 890  3.4749286 3.5741021}%
\special{ar 1518 1900 320 890  3.6336062 3.7327798}%
\special{ar 1518 1900 320 890  3.7922839 3.8914575}%
\special{ar 1518 1900 320 890  3.9509616 4.0501352}%
\special{ar 1518 1900 320 890  4.1096393 4.2088129}%
\special{ar 1518 1900 320 890  4.2683170 4.3674905}%
\special{ar 1518 1900 320 890  4.4269947 4.5261682}%
\special{ar 1518 1900 320 890  4.5856724 4.6848459}%
% ELLIPSE 2 0 3 0
% 4 4398 1900 4718 2790 4398 2830 4398 630
% 
\special{pn 8}%
\special{ar 4398 1900 320 890  4.7123890 6.2831853}%
\special{ar 4398 1900 320 890  0.0000000 1.5707963}%
% ELLIPSE 2 0 3 0
% 4 1518 1900 1838 2790 1518 2830 1518 630
% 
\special{pn 8}%
\special{ar 1518 1900 320 890  4.7123890 6.2831853}%
\special{ar 1518 1900 320 890  0.0000000 1.5707963}%
% LINE 2 0 3 0
% 2 718 1900 2478 1900
% 
\special{pn 8}%
\special{pa 718 1900}%
\special{pa 2478 1900}%
\special{fp}%
% SPLINE 2 0 3 0
% 7 1518 1000 1542 1200 1566 1400 1678 1610 2038 1790 2486 1850 2486 1850
% 
\special{pn 8}%
\special{pa 1518 1000}%
\special{pa 1522 1032}%
\special{pa 1528 1064}%
\special{pa 1532 1096}%
\special{pa 1536 1128}%
\special{pa 1538 1160}%
\special{pa 1542 1192}%
\special{pa 1544 1224}%
\special{pa 1546 1256}%
\special{pa 1550 1286}%
\special{pa 1552 1318}%
\special{pa 1556 1350}%
\special{pa 1562 1382}%
\special{pa 1570 1414}%
\special{pa 1580 1444}%
\special{pa 1592 1474}%
\special{pa 1604 1504}%
\special{pa 1620 1532}%
\special{pa 1638 1560}%
\special{pa 1656 1586}%
\special{pa 1678 1610}%
\special{pa 1700 1632}%
\special{pa 1724 1654}%
\special{pa 1750 1672}%
\special{pa 1776 1690}%
\special{pa 1804 1706}%
\special{pa 1834 1720}%
\special{pa 1862 1734}%
\special{pa 1894 1746}%
\special{pa 1924 1758}%
\special{pa 1956 1768}%
\special{pa 1988 1778}%
\special{pa 2020 1786}%
\special{pa 2052 1794}%
\special{pa 2082 1800}%
\special{pa 2114 1806}%
\special{pa 2146 1812}%
\special{pa 2178 1818}%
\special{pa 2210 1822}%
\special{pa 2242 1826}%
\special{pa 2274 1830}%
\special{pa 2306 1834}%
\special{pa 2336 1838}%
\special{pa 2368 1840}%
\special{pa 2400 1844}%
\special{pa 2432 1846}%
\special{pa 2464 1848}%
\special{pa 2486 1850}%
\special{sp}%
% SPLINE 2 0 3 0
% 7 1518 2810 1542 2610 1566 2410 1678 2200 2038 2020 2486 1960 2486 1960
% 
\special{pn 8}%
\special{pa 1518 2810}%
\special{pa 1522 2778}%
\special{pa 1528 2748}%
\special{pa 1532 2716}%
\special{pa 1536 2684}%
\special{pa 1538 2652}%
\special{pa 1542 2620}%
\special{pa 1544 2588}%
\special{pa 1546 2556}%
\special{pa 1550 2524}%
\special{pa 1552 2492}%
\special{pa 1556 2460}%
\special{pa 1562 2428}%
\special{pa 1570 2398}%
\special{pa 1580 2366}%
\special{pa 1592 2336}%
\special{pa 1604 2308}%
\special{pa 1620 2278}%
\special{pa 1638 2252}%
\special{pa 1656 2226}%
\special{pa 1678 2202}%
\special{pa 1700 2178}%
\special{pa 1724 2158}%
\special{pa 1750 2138}%
\special{pa 1776 2120}%
\special{pa 1804 2104}%
\special{pa 1834 2090}%
\special{pa 1862 2076}%
\special{pa 1894 2064}%
\special{pa 1924 2054}%
\special{pa 1956 2042}%
\special{pa 1988 2034}%
\special{pa 2020 2026}%
\special{pa 2052 2018}%
\special{pa 2082 2010}%
\special{pa 2114 2004}%
\special{pa 2146 1998}%
\special{pa 2178 1994}%
\special{pa 2210 1988}%
\special{pa 2242 1984}%
\special{pa 2274 1980}%
\special{pa 2306 1976}%
\special{pa 2336 1974}%
\special{pa 2368 1970}%
\special{pa 2400 1968}%
\special{pa 2432 1966}%
\special{pa 2464 1962}%
\special{pa 2486 1960}%
\special{sp}%
% SPLINE 2 0 3 0
% 7 718 1000 742 1200 766 1400 878 1610 1238 1790 1686 1850 1686 1850
% 
\special{pn 8}%
\special{pa 718 1000}%
\special{pa 722 1032}%
\special{pa 728 1064}%
\special{pa 732 1096}%
\special{pa 736 1128}%
\special{pa 738 1160}%
\special{pa 742 1192}%
\special{pa 744 1224}%
\special{pa 746 1256}%
\special{pa 750 1286}%
\special{pa 752 1318}%
\special{pa 756 1350}%
\special{pa 762 1382}%
\special{pa 770 1414}%
\special{pa 780 1444}%
\special{pa 792 1474}%
\special{pa 804 1504}%
\special{pa 820 1532}%
\special{pa 838 1560}%
\special{pa 856 1586}%
\special{pa 878 1610}%
\special{pa 900 1632}%
\special{pa 924 1654}%
\special{pa 950 1672}%
\special{pa 976 1690}%
\special{pa 1004 1706}%
\special{pa 1034 1720}%
\special{pa 1062 1734}%
\special{pa 1094 1746}%
\special{pa 1124 1758}%
\special{pa 1156 1768}%
\special{pa 1188 1778}%
\special{pa 1220 1786}%
\special{pa 1252 1794}%
\special{pa 1282 1800}%
\special{pa 1314 1806}%
\special{pa 1346 1812}%
\special{pa 1378 1818}%
\special{pa 1410 1822}%
\special{pa 1442 1826}%
\special{pa 1474 1830}%
\special{pa 1506 1834}%
\special{pa 1536 1838}%
\special{pa 1568 1840}%
\special{pa 1600 1844}%
\special{pa 1632 1846}%
\special{pa 1664 1848}%
\special{pa 1686 1850}%
\special{sp}%
% SPLINE 2 0 3 0
% 7 718 2810 742 2610 766 2410 878 2200 1238 2020 1686 1960 1686 1960
% 
\special{pn 8}%
\special{pa 718 2810}%
\special{pa 722 2778}%
\special{pa 728 2748}%
\special{pa 732 2716}%
\special{pa 736 2684}%
\special{pa 738 2652}%
\special{pa 742 2620}%
\special{pa 744 2588}%
\special{pa 746 2556}%
\special{pa 750 2524}%
\special{pa 752 2492}%
\special{pa 756 2460}%
\special{pa 762 2428}%
\special{pa 770 2398}%
\special{pa 780 2366}%
\special{pa 792 2336}%
\special{pa 804 2308}%
\special{pa 820 2278}%
\special{pa 838 2252}%
\special{pa 856 2226}%
\special{pa 878 2202}%
\special{pa 900 2178}%
\special{pa 924 2158}%
\special{pa 950 2138}%
\special{pa 976 2120}%
\special{pa 1004 2104}%
\special{pa 1034 2090}%
\special{pa 1062 2076}%
\special{pa 1094 2064}%
\special{pa 1124 2054}%
\special{pa 1156 2042}%
\special{pa 1188 2034}%
\special{pa 1220 2026}%
\special{pa 1252 2018}%
\special{pa 1282 2010}%
\special{pa 1314 2004}%
\special{pa 1346 1998}%
\special{pa 1378 1994}%
\special{pa 1410 1988}%
\special{pa 1442 1984}%
\special{pa 1474 1980}%
\special{pa 1506 1976}%
\special{pa 1536 1974}%
\special{pa 1568 1970}%
\special{pa 1600 1968}%
\special{pa 1632 1966}%
\special{pa 1664 1962}%
\special{pa 1686 1960}%
\special{sp}%
% ELLIPSE 2 0 3 0
% 4 3598 1900 3710 2200 3710 2200 3710 2200
% 
\special{pn 8}%
\special{ar 3598 1900 112 300  0.0000000 6.2831853}%
% ELLIPSE 2 1 3 0
% 4 5518 1900 5630 2200 5630 2200 5630 2200
% 
\special{pn 8}%
\special{ar 5518 1900 112 300  0.0000000 0.2912621}%
\special{ar 5518 1900 112 300  0.4660194 0.7572816}%
\special{ar 5518 1900 112 300  0.9320388 1.2233010}%
\special{ar 5518 1900 112 300  1.3980583 1.6893204}%
\special{ar 5518 1900 112 300  1.8640777 2.1553398}%
\special{ar 5518 1900 112 300  2.3300971 2.6213592}%
\special{ar 5518 1900 112 300  2.7961165 3.0873786}%
\special{ar 5518 1900 112 300  3.2621359 3.5533981}%
\special{ar 5518 1900 112 300  3.7281553 4.0194175}%
\special{ar 5518 1900 112 300  4.1941748 4.4854369}%
\special{ar 5518 1900 112 300  4.6601942 4.9514563}%
\special{ar 5518 1900 112 300  5.1262136 5.4174757}%
\special{ar 5518 1900 112 300  5.5922330 5.8834951}%
\special{ar 5518 1900 112 300  6.0582524 6.2832853}%
% LINE 2 1 3 0
% 2 718 2810 718 1020
% 
\special{pn 8}%
\special{pa 718 2810}%
\special{pa 718 1020}%
\special{da 0.070}%
% LINE 2 1 3 0
% 2 2478 2810 2478 1020
% 
\special{pn 8}%
\special{pa 2478 2810}%
\special{pa 2478 1020}%
\special{da 0.070}%
% LINE 2 1 3 0
% 2 5518 2810 5518 1020
% 
\special{pn 8}%
\special{pa 5518 2810}%
\special{pa 5518 1020}%
\special{da 0.070}%
% LINE 2 1 3 0
% 2 3598 2810 3598 1020
% 
\special{pn 8}%
\special{pa 3598 2810}%
\special{pa 3598 1020}%
\special{da 0.070}%
% LINE 2 0 3 0
% 4 3598 1900 5518 1900 5054 1900 5038 1900
% 
\special{pn 8}%
\special{pa 3598 1900}%
\special{pa 5518 1900}%
\special{fp}%
\special{pa 5054 1900}%
\special{pa 5038 1900}%
\special{fp}%
% LINE 2 0 3 0
% 4 3598 2200 5518 2200 5054 2200 5038 2200
% 
\special{pn 8}%
\special{pa 3598 2200}%
\special{pa 5518 2200}%
\special{fp}%
\special{pa 5054 2200}%
\special{pa 5038 2200}%
\special{fp}%
% LINE 2 0 3 0
% 4 3598 1800 5518 1800 5054 1800 5038 1800
% 
\special{pn 8}%
\special{pa 3598 1800}%
\special{pa 5518 1800}%
\special{fp}%
\special{pa 5054 1800}%
\special{pa 5038 1800}%
\special{fp}%
% LINE 2 0 3 0
% 4 3598 1600 5518 1600 5054 1600 5038 1600
% 
\special{pn 8}%
\special{pa 3598 1600}%
\special{pa 5518 1600}%
\special{fp}%
\special{pa 5054 1600}%
\special{pa 5038 1600}%
\special{fp}%
% LINE 2 0 3 0
% 4 3598 1700 5518 1700 5054 1700 5038 1700
% 
\special{pn 8}%
\special{pa 3598 1700}%
\special{pa 5518 1700}%
\special{fp}%
\special{pa 5054 1700}%
\special{pa 5038 1700}%
\special{fp}%
% SPLINE 2 0 3 0
% 5 4398 2780 4430 2580 4542 2390 5086 2250 5518 2250
% 
\special{pn 8}%
\special{pa 4398 2780}%
\special{pa 4402 2748}%
\special{pa 4404 2716}%
\special{pa 4408 2684}%
\special{pa 4414 2654}%
\special{pa 4420 2622}%
\special{pa 4428 2590}%
\special{pa 4438 2560}%
\special{pa 4450 2530}%
\special{pa 4464 2500}%
\special{pa 4480 2472}%
\special{pa 4496 2444}%
\special{pa 4516 2418}%
\special{pa 4538 2396}%
\special{pa 4560 2374}%
\special{pa 4584 2356}%
\special{pa 4610 2338}%
\special{pa 4638 2324}%
\special{pa 4666 2310}%
\special{pa 4696 2298}%
\special{pa 4726 2288}%
\special{pa 4758 2280}%
\special{pa 4790 2274}%
\special{pa 4822 2268}%
\special{pa 4856 2264}%
\special{pa 4890 2260}%
\special{pa 4924 2256}%
\special{pa 4958 2254}%
\special{pa 4992 2252}%
\special{pa 5026 2252}%
\special{pa 5060 2250}%
\special{pa 5094 2250}%
\special{pa 5128 2250}%
\special{pa 5160 2250}%
\special{pa 5194 2250}%
\special{pa 5226 2250}%
\special{pa 5258 2250}%
\special{pa 5290 2250}%
\special{pa 5322 2250}%
\special{pa 5354 2250}%
\special{pa 5384 2250}%
\special{pa 5416 2250}%
\special{pa 5448 2250}%
\special{pa 5478 2250}%
\special{pa 5510 2250}%
\special{pa 5518 2250}%
\special{sp}%
% SPLINE 2 0 3 0
% 5 3598 1010 3630 1210 3742 1400 4286 1540 4718 1540
% 
\special{pn 8}%
\special{pa 3598 1010}%
\special{pa 3602 1042}%
\special{pa 3604 1074}%
\special{pa 3608 1106}%
\special{pa 3614 1138}%
\special{pa 3620 1170}%
\special{pa 3628 1200}%
\special{pa 3638 1232}%
\special{pa 3650 1262}%
\special{pa 3664 1292}%
\special{pa 3680 1320}%
\special{pa 3696 1346}%
\special{pa 3716 1372}%
\special{pa 3738 1396}%
\special{pa 3760 1416}%
\special{pa 3784 1436}%
\special{pa 3810 1452}%
\special{pa 3838 1468}%
\special{pa 3866 1480}%
\special{pa 3896 1492}%
\special{pa 3926 1502}%
\special{pa 3958 1510}%
\special{pa 3990 1518}%
\special{pa 4022 1524}%
\special{pa 4056 1528}%
\special{pa 4090 1532}%
\special{pa 4124 1534}%
\special{pa 4158 1536}%
\special{pa 4192 1538}%
\special{pa 4226 1540}%
\special{pa 4260 1540}%
\special{pa 4294 1540}%
\special{pa 4328 1542}%
\special{pa 4360 1542}%
\special{pa 4394 1542}%
\special{pa 4426 1542}%
\special{pa 4458 1542}%
\special{pa 4490 1542}%
\special{pa 4522 1542}%
\special{pa 4554 1542}%
\special{pa 4584 1542}%
\special{pa 4616 1542}%
\special{pa 4648 1542}%
\special{pa 4678 1540}%
\special{pa 4710 1540}%
\special{pa 4718 1540}%
\special{sp}%
% SPLINE 2 0 3 0
% 5 4398 1010 4430 1210 4542 1400 5086 1540 5518 1540
% 
\special{pn 8}%
\special{pa 4398 1010}%
\special{pa 4402 1042}%
\special{pa 4404 1074}%
\special{pa 4408 1106}%
\special{pa 4414 1138}%
\special{pa 4420 1170}%
\special{pa 4428 1200}%
\special{pa 4438 1232}%
\special{pa 4450 1262}%
\special{pa 4464 1292}%
\special{pa 4480 1320}%
\special{pa 4496 1346}%
\special{pa 4516 1372}%
\special{pa 4538 1396}%
\special{pa 4560 1416}%
\special{pa 4584 1436}%
\special{pa 4610 1452}%
\special{pa 4638 1468}%
\special{pa 4666 1480}%
\special{pa 4696 1492}%
\special{pa 4726 1502}%
\special{pa 4758 1510}%
\special{pa 4790 1518}%
\special{pa 4822 1524}%
\special{pa 4856 1528}%
\special{pa 4890 1532}%
\special{pa 4924 1534}%
\special{pa 4958 1536}%
\special{pa 4992 1538}%
\special{pa 5026 1540}%
\special{pa 5060 1540}%
\special{pa 5094 1540}%
\special{pa 5128 1542}%
\special{pa 5160 1542}%
\special{pa 5194 1542}%
\special{pa 5226 1542}%
\special{pa 5258 1542}%
\special{pa 5290 1542}%
\special{pa 5322 1542}%
\special{pa 5354 1542}%
\special{pa 5384 1542}%
\special{pa 5416 1542}%
\special{pa 5448 1542}%
\special{pa 5478 1540}%
\special{pa 5510 1540}%
\special{pa 5518 1540}%
\special{sp}%
% SPLINE 2 0 3 0
% 5 3598 2780 3630 2580 3742 2390 4286 2250 4718 2250
% 
\special{pn 8}%
\special{pa 3598 2780}%
\special{pa 3602 2748}%
\special{pa 3604 2716}%
\special{pa 3608 2684}%
\special{pa 3614 2654}%
\special{pa 3620 2622}%
\special{pa 3628 2590}%
\special{pa 3638 2560}%
\special{pa 3650 2530}%
\special{pa 3664 2500}%
\special{pa 3680 2472}%
\special{pa 3696 2444}%
\special{pa 3716 2418}%
\special{pa 3738 2396}%
\special{pa 3760 2374}%
\special{pa 3784 2356}%
\special{pa 3810 2338}%
\special{pa 3838 2324}%
\special{pa 3866 2310}%
\special{pa 3896 2298}%
\special{pa 3926 2288}%
\special{pa 3958 2280}%
\special{pa 3990 2274}%
\special{pa 4022 2268}%
\special{pa 4056 2264}%
\special{pa 4090 2260}%
\special{pa 4124 2256}%
\special{pa 4158 2254}%
\special{pa 4192 2252}%
\special{pa 4226 2252}%
\special{pa 4260 2250}%
\special{pa 4294 2250}%
\special{pa 4328 2250}%
\special{pa 4360 2250}%
\special{pa 4394 2250}%
\special{pa 4426 2250}%
\special{pa 4458 2250}%
\special{pa 4490 2250}%
\special{pa 4522 2250}%
\special{pa 4554 2250}%
\special{pa 4584 2250}%
\special{pa 4616 2250}%
\special{pa 4648 2250}%
\special{pa 4678 2250}%
\special{pa 4710 2250}%
\special{pa 4718 2250}%
\special{sp}%
% LINE 2 0 3 0
% 4 3598 2000 5518 2000 5054 2000 5038 2000
% 
\special{pn 8}%
\special{pa 3598 2000}%
\special{pa 5518 2000}%
\special{fp}%
\special{pa 5054 2000}%
\special{pa 5038 2000}%
\special{fp}%
% LINE 2 0 3 0
% 4 3598 2100 5518 2100 5054 2100 5038 2100
% 
\special{pn 8}%
\special{pa 3598 2100}%
\special{pa 5518 2100}%
\special{fp}%
\special{pa 5054 2100}%
\special{pa 5038 2100}%
\special{fp}%
% STR 2 0 3 0
% 3 1200 3100 1200 3200 2 0
% $(S^{1} \times T^{2},\mathcal{G}^{0})$
\put(12.0000,-32.0000){\makebox(0,0)[lb]{$(S^{1} \times T^{2},\mathcal{G}^{0})$}}%
% STR 2 0 3 0
% 3 3600 3090 3600 3190 2 0
% After inserting a sorid torus.
\put(36.0000,-31.9000){\makebox(0,0)[lb]{After inserting a sorid torus.}}%
\end{picture}%

\end{center}

\vspace{30pt}

We can decompose $S^{3}$ into $\mathcal{F}^{t}$-saturated subsets $K^{t}_{1}$, $K^{t}_{2}$ and $T^{2} \times [0,1]$. Let $(\theta_{1},\theta_{2},s)$ be the coordinates on $T^2 \times [0,1]$ such that 
\begin{itemize}
\item $\theta_{1}$ parametrizes a meridian of $K^{t}_{1}$ and a longitude of $K^{t}_{2}$, and
\item $\theta_{2}$ parametrizes a meridian of $K^{t}_{2}$ and a longitude of $K^{t}_{1}$.
\end{itemize}
By the construction, we can construct $\mathcal{F}^{t}$ so that the leaves of $\mathcal{F}^{t}$ are transverse to a $1$-form $d\theta_{1} + d\theta_{2}$ on $T^{2} \times [0,1]$.

By the Rummler-Sullivan criterion, let us show
\begin{proposition}\label{Proposition : F t}
$\mathcal{F}^{t}$ is minimizable for nonzero $t$ in $[0,1]$.
\end{proposition}

\begin{proof}
By the Rummler-Sullivan criterion (see Sullivan \cite{Sullivan}), $\mathcal{F}^{t}$ is minimizable if and only if there exist a $1$-form $\chi$ on $S^{3}$ such that $\chi|_{T\mathcal{F}^{t}}$ has no zero and $d\chi|_{T\mathcal{F}^{t}}=0$.

We take the decomposition of $S^{3}$ into $\mathcal{F}^{t}$-saturated subsets 
\begin{equation}\label{Equation : Decomposition}
S^{3} = K^{t}_{1} \sqcup K^{t}_{2} \sqcup (T^{2} \times [0,1])
\end{equation}
as above. We take a coordinate $(\theta_{1},\theta_{2},s)$ on $T^{2} \times [0,1]$ as noted in the paragraph previous to Proposition~\ref{Proposition : F t}. We assume that $\mathcal{F}^{t}$ is transverse to $d\theta_{1} + d\theta_{2}$ on $T^{2} \times [0,1]$, while $\mathcal{F}^{t}|_{K^{t}_{i}}$ is the product foliation on a solid torus for $i=1$ and $2$. Let $A_{i}$ be the axis of $K^{t}_{i}$. We can extend $\theta_{1}$ from $T^{2} \times [0,1]$ to $S^{3} - A_{1}$ so that 
\begin{itemize}
\item $\theta_{1}$ is the composite of a diffeomorphism $S^{3} - A_{1} \cong S^{1} \times D^{2}$ and the first projection $S^{1} \times D^{2} \longrightarrow S^{1}$ and
\item $d\theta_{1}$ is transverse to $\mathcal{F}^{t}$ on $S^{3} - K^{t}_{2}$.
\end{itemize}
We extend $\theta_{2}$ to $S^{3} - A_{2}$ in a similar way.

Let $r_{i}$ be the radius coordinate on the $D^{2}$-component of $K^{t}_{i}$. Let $\phi_{i}$ be a nonnegative smooth function on $S^{3}$ such that
\begin{itemize}
\item $\phi_{i} = 0$ on $S^{3} - K^{t}_{i}$,
\item $\phi_{i}$ is a function of $r_{i}$ on $K^{t}_{i}$ and
\item $\phi_{i}$ is $1$ on an open neighborhood of $A_{i}$.
\end{itemize}

We define a $1$-form $\chi$ on $S^{3}$ by
\begin{equation}
\chi = (1 - \phi_{1}) d\theta_{1} + (1 - \phi_{2}) d\theta_{2}.
\end{equation}
Note that $(1 - \phi_{i}) d\theta_{i}$ is well-defined on $S^{3}$, though $d\theta_{i}$ is not defined on the axis $A_{i}$ of $K^{t}_{i}$. 

We will confirm that $\chi$ satisfies the conditions in the Rummler-Sullivan's criterion for $\mathcal{F}^{t}$ on the components of the decomposition \eqref{Equation : Decomposition}. Since the restriction of $\chi$ to $T^{2} \times [0,1]$ is equal to $d\theta_{1} + d\theta_{2}$, $\chi|_{T^{2} \times [0,1]}$ is transverse to $\mathcal{F}^{t}|_{T^{2} \times [0,1]}$ and $d\chi = 0$. On $K^{t}_{1}$, we have
\begin{equation}
\chi|_{K^{t}_{1}} = (1 - \phi_{1}) d\theta_{1} + d\theta_{2}
\end{equation}
Since $d\theta_{2}$ is transverse to $\mathcal{F}^{t}|_{K^{t}_{1}}$ and $d\theta_{1}|_{T\mathcal{F}^{t}}$ is zero, $\chi|_{K^{t}_{1}}$ is transverse to $\mathcal{F}^{t}|_{K^{t}_{1}}$. We have $d\chi|_{K^{t}_{1}} = d\phi_{1} \wedge d\theta_{1}$, and hence $(d\chi|_{K^{t}_{1}})|_{T\mathcal{F}^{t}}=0$. We can prove that $\chi|_{K^{t}_{2}}$ also satisfies the two conditions in the Rummler-Sullivan's characterization in the same way. Hence $\mathcal{F}^{t}$ is minimizable.
\end{proof}

Let $X^{s}$ be a nowhere vanishing vector field tangent to $\mathcal{F}^{s}$ for $s=0$ and $1$. Let $X^{t} = t X^{1} + (1-t) X^{0}$ for $0 < t < 1$. Then $X^{t}$ is nowhere vanishing and defines a foliation $\mathcal{H}^{t}$. In this family $\{\mathcal{H}^{t}\}_{t \in [0,1]}$ of foliations, $\mathcal{H}^{t}$ is non-minimizable for $0 \leq t < 1$ and $\mathcal{H}^{1}$ is minimizable. In fact, $\mathcal{H}^{t}$ is shown to be non-minimizable for $0 \leq t < 1$ by an argument similar to the proof of non-minimizability of $\mathcal{F}^{0}$ by Candel and Conlon in Example 10.5.19 of \cite{Candel Conlon}. The proof of the minimizability of $\mathcal{H}^{1}$ is similar to the proof of Proposition \ref{Proposition : F t}.

\end{document}